\documentclass[11pt,reqno]{amsart}
\usepackage{amssymb}
\usepackage{amsmath, amssymb, amsthm}
\usepackage{mathrsfs}

\newtheorem{theorem}{Theorem}[section]
\newtheorem{definition}{Definition}[section]
\newtheorem{lemma}{Lemma}[section]
\newtheorem{remark}{Remark}[section]

\numberwithin{equation}{section}

\begin{document}
\title[Navier-Stokes-Boltzmann equations]{Existence results and blow-up criterion of    compressible  radiation hydrodynamic equations  }

\author{yachun li }
\address[Y. C. Li]{Department of Mathematics and Key Lab of Scientific and Engineering Computing (MOE), Shanghai Jiao Tong University,
Shanghai 200240, P.R.China} \email{\tt ycli@sjtu.edu.cn}

\author{Shengguo Zhu}
\address[S. G. Zhu]{Department of Mathematics, Shanghai Jiao Tong University,
Shanghai 200240, P.R.China; School of Mathematics, Georgia Tech,
Atlanta 30332, U.S.A.}
\email{\tt zhushengguo@sjtu.edu.cn}

\begin{abstract}In this paper, we consider the  $3$D compressible  radiation hydrodynamic (RHD) equations with  thermal conductivity  in a bounded domain. The  existence of unique local strong solutions with vacuum  is firstly established  when the initial data are  arbitrarily large and satisfy some  initial layer compatibility condition.   Moreover, we show that if the initial vacuum domain is not so irregular, then the  compatibility condition is necessary and sufficient to guarantee the existence of the unique strong solution.  Finally, a Beal-Kato-Majda type  blow-up criterion is shown in terms of  $(\nabla I,\rho,\theta)$.
\end{abstract}

\date{May. 14, 2012}
 \subjclass{Primary: 35A09, 35B44, 35M31; Secondary: 35Q31, 35Q35}
\keywords{Radiation, Navier-slip, Strong solution, Vacuum, Blow-up criterion.
}

\maketitle

\section{Introduction}
The purpose of our paper is to provide a  local theory of strong solutions    with vacuum  to the RHD systm in the framework of Sobolev space. 
This system appears in various astrophysical contexts \cite{kr} and high-temperature plasma physics \cite{sjx}. 
Suppose that  the matter is in local thermodynamical equilibrium,  the coupled system  of the RHD equations with heat conduction for the mass density $\rho(t,x)$, the fluid velocity $u(t,x)=(u^{(1)},u^{(2)},u^{(3)})$, the specific internal energy $e(t,x)$, and the specific radiation intensity $I(v,\Omega,t,x)$ in a domain  $\mathbb{V} \subset \mathbb{R}^3$ reads as  \cite{gp}:
\begin{equation}
\label{eq:1.1}
\begin{cases}
\displaystyle
\frac{1}{c}I_t+\Omega\cdot\nabla I=A_r,\\[10pt]
\displaystyle
\rho_t+\text{div}(\rho u)=0,\\[10pt]
\displaystyle
\Big(\rho u+\frac{1}{c^{2}}F_r\Big)_t+\text{div}(\rho u\otimes u+P_r)
  +\nabla P_m =\text{div}\mathbb{T},\\[10pt]
\displaystyle
\big(\rho E_m+E_r\big)_t+
\text{div}\big((\rho E_m+P_m)u+F_r\big)
=\text{div} (u \mathbb{T})+\kappa\triangle \theta.
\end{cases}
\end{equation}

In this system, $t\geq 0$ is the time; $x\in\mathbb{V}$ is the spatial coordinate; $v\in \mathbb{R}^+$ is the frequency of photon;  $\Omega\in S^2 $ is the travel direction of photon,  here $S^2$ is the unit sphere in $\mathbb{R}^3$; $E_m=\frac{1}{2} u^2+e$ is the specific total material energy;
$P_m$ is the material pressure satisfying the following equations of state
\begin{equation}
\label{eq:1.3}
P_m=R\rho \theta=(\gamma-1)\rho e, \quad e=c_v \theta,
\end{equation}
where $R$ and $c_v$ are both positive constants, $\gamma > 1 $ is the adiabatic index and  $\theta$ is the absolute temperature.
$\mathbb{T}$ is the viscosity stress tensor given  by
\begin{equation}
\label{eq:1.4}
\mathbb{T}=2\mu D(u)+\lambda \text{div}u \mathbb{I}_3, \quad D(u)=\frac{\nabla u+(\nabla u)^\top}{2},
\end{equation}
where $D(u)$ is the deformation tensor,  $\mathbb{I}_3$ is the $3\times 3$ unit matrix,
$\mu$ is the shear viscosity, $\lambda+\frac{2}{3}\mu$ is the bulk viscosity, $\kappa$ is the thermal conductivity, $(\mu,\lambda,\kappa)$ are  constants satisfying
\begin{equation}
\label{eq:1.5}
  \mu > 0, \quad 3\lambda+2\mu \geq 0,\quad \kappa>0.\end{equation}
 
The radiation energy density $E_r$, the radiation flux $F_r$, and  the radiation pressure tensor $P_r$ in (\ref{eq:1.1}) are defined by
\begin{equation}\label{fan111}
\begin{cases}
\displaystyle
E_r=\frac{1}{c}\int_0^\infty \int_{S^{2}} I(v,\Omega,t,x)\text{d}\Omega \text{d}v ,\\[8pt]
\displaystyle
F_r=\int_0^\infty \int_{S^{2}}  I(v,\Omega,t,x)\Omega \text{d}\Omega \text{d}v,\\[8pt]
\displaystyle
P_r=\frac{1}{c}\int_0^\infty \int_{S^{2}}  I(v,\Omega,t,x)\Omega\otimes\Omega\text{d}\Omega \text{d}v.
\end{cases}
\end{equation}
The collision term in radiation transfer equation $(\ref{eq:1.1})_1$ reads:
\begin{equation}\label{fan222}
\begin{split}
A_r=S-\sigma_aI
+\int_0^\infty \int_{S^{2}} \Big(\frac{v}{v'}\sigma_sI'
-\sigma'_sI\Big) \text{d}\Omega' \text{d}v',
\end{split}
\end{equation}
where $I'=I(v',\Omega',t,x)$;  $S=S(v,t,x)\geq 0$ is the rate of energy emission due to spontaneous process; $\sigma_a=\sigma_a(v,t,x,\rho,\theta)\geq 0$ denotes the absorption coefficient that may also depend on the mass density $\rho$ and the temperature $\theta $ of the matter; $\sigma_s$ is the ``differential scattering coefficient'' such that the probability of a photon being
scattered from $v'$ to $v$ contained in $\text{d}v$,  from $\Omega'$ to
$\Omega$ contained in $\text{d}\Omega$, and travelling a distance $\text{d}s$ is
given by $\sigma_s(v' \rightarrow v,\Omega'\cdot\Omega)\text{d}v \text{d}\Omega
\text{d}s$, and
\begin{equation*}
\begin{split}
\sigma_s\equiv  \sigma_s(v' \rightarrow v, \Omega'\cdot\Omega,\rho)=O(\rho),\ \sigma'_s\equiv  \sigma_s(v \rightarrow v', \Omega\cdot\Omega',\rho)=O(\rho).
\end{split}
\end{equation*}

When there is no radiation effect,  the existence of unique local strong solutions with vacuum has been solved by many papers, and we refer the readers to \cite{CK3}\cite{CK}.  For the existence of solutions with arbitrary data in $3$D space, the main breakthrough is due to Lions \cite{lion}, where he established the global existence of weak solutions for $\mathbb{R}^3$, perodic domains or bouned domains with homogenous Dirichlet boundary conditions provided $\gamma >9/5$. The restriction on $\gamma$ is improved to $\gamma>3/2$ by Feireisl \cite{fu1}. Recently, Huang-Li-Xin obtained the global well-posedness of classical solutions with large oscillations and vacuum to Cauchy problem \cite{HX1} for  isentropic flow with small energy.

In general, studying the  radiation hydrodynamic equations is challenging because of its complexity and mathematical difficulty. For Euler-Boltzmann equations, recently, Jiang-Zhong \cite{sjx} obtained the local existence of $C^1$ solutions for the Cauchy problem  away from vacuum.  Jiang-Wang \cite{pjd} showed that some $ C^1$ solutions  will blow up in  finite time, regardless of the size of the initial disturbance. Li-Zhu \cite{sz1} established the local existence of Makino-Ukai-Kawashima type's regular solution (see \cite{tms1}) with vacuum, and also proved that the regular solutions with compact density will blow up in finite time.

For Navier-Stokes-Boltzmann equations,  Ducomet-Ne$\check{\text{c}}$asov$\acute{\text{a}}$ \cite{BD}\cite{BS}  studied the global weak solutions to the Navier-Stokes-Boltzmann equations and its large time behavior in 1-D space. Li-Zhu \cite{sz2} considered the formation of singularities to classical solutions with compact density.  Some special phenomenon has been observed, for example, it is known in contrast with the second law of thermodynamics, the associated entropy equation may contain a negative production term for RHD system (see also Buet-Despr$\acute{\text{e}}$s \cite{add1}). Moreover, from  Ducomet-Feireisl-Ne$\check{\text{c}}$asov$\acute{\text{a}}$ \cite{add2},  in which they  obtained the existence of global weak solution for some RHD model,  we know that the velocity field $u$ may develop uncontrolled time oscillations on  vacuum zones.

However, in this paper,  due to the radiation transfer equation $ (\ref{eq:1.1})_1$,  (\ref{eq:1.3}) and (\ref{fan111})-(\ref{fan222}), system (\ref{eq:1.1}) can be written as
\begin{equation}
\label{eq:1.2}
\begin{cases}
\displaystyle
\frac{1}{c}I_t+\Omega\cdot\nabla I=A_r,\\[6pt]
\displaystyle
\rho_t+\text{div}(\rho u)=0,\\[6pt]
\displaystyle
(\rho u)_t+\text{div}(\rho u\otimes u)
  +\nabla P_m +Lu=-\frac{1}{c}\int_0^\infty \int_{S^2}A_r\Omega \text{d}\Omega \text{d}v,\\[10pt]
\displaystyle
(\rho \theta)_t+
\text{div}(\rho \theta u)+\frac{1}{c_v}(P_m\text{div} u-\kappa\triangle \theta)
=\frac{1}{c_v}(Q( u)+N_r),
\end{cases}
\end{equation}
where $N_r$, $Lu$ and $Q( u)$ are defined by
\begin{equation*}
\begin{cases}
\displaystyle
N_r=\int_0^\infty \int_{S^2} \Big(1-\frac{u\cdot \Omega}{c}\Big)A_r \text{d}\Omega \text{d}v,\\[8pt]
\displaystyle
Lu=-\mu\triangle u-(\lambda+\mu)\nabla \text{div} u,\ Q( u)=\frac{\mu}{2}|\nabla u+(\nabla u)^\top|^2+\lambda |\text{div} u|^2.
\end{cases}
\end{equation*}
Now we consider the initial-boundary value problem (IBVP) for system (\ref{eq:1.2}) with initial data
\begin{equation} \label{eq:2.2hh}
I|_{t=0}=I_0(v,\Omega,x),\ (\rho,u,\theta)|_{t=0}=(\rho_0(x),u_0(x), \theta_0(x)), \  (v,\Omega, x) \in \mathbb{R}^+ \times S^2 \times \mathbb{V},
\end{equation}
and one of the following two types of boundary conditions:

(1) Transparency, Dirichlet and Neumann   boundary conditions for $(I,u,\theta) $: $\mathbb{V} \subset \mathbb{R}^3$ is a bounded   smooth domain and

\begin{equation}\label{fan1}
\begin{split}
I|_{\partial \mathbb{V}}=0,\ \Omega \cdot n\leq 0;\quad   u|_{\partial \mathbb{V}}=0; \quad  \nabla \theta \cdot n|_{ \partial \mathbb{V}}=0,
\end{split}
\end{equation}
where $n = (n_1, n_2, n_3)$ is the unit outward normal to $\partial \mathbb{V}$.

(2) Transparency, Navier-slip and Neumann   boundary conditions for $(I,u,\theta) $: $\mathbb{V} \subset \mathbb{R}^3$  is bounded, simply connected,
￼smooth domain, and
\begin{equation}
\label{fan2}
\begin{split}
I|_{\partial \mathbb{V}}=0,\ \Omega \cdot n\leq 0;\quad (u\cdot n, (\nabla \times u) \cdot n)|_{\partial \mathbb{V}}=(0,0);\quad \nabla \theta \cdot n |_{\partial \mathbb{V}}=0.
\end{split}
\end{equation}

 The first  condition in (\ref{fan2})   is the physical non-penetration  boundary condition for radiation flow.  While the first part $u\cdot n=0 $ of second one in (\ref{fan2}) is  the non-penetration  boundary condition for the fluid, and $(\nabla \times u) \cdot n=0$ is also known in the following form
$$
(D(u)\cdot n)_\tau =−k_\tau u_\tau, 
$$
where $k_\tau$ is the corresponding principal curvature of $\mathbb{V}$.  Then the second one in  (\ref{fan2}) implies the tangential component of $D(u) \cdot n$ vanishes on flat portions of the boundary  $\partial \mathbb{V}$.


Throughout this paper, we adopt the following simplified notations for the standard homogeneous and inhomogeneous Sobolev space:
\begin{equation*}\begin{split}
&  \|f\|_{W^{m,r}}=\|f\|_{W^{m,r}(\mathbb{V})},\ \     \|f\|_s=\|f\|_{H^s(\mathbb{V})},\ \   |f|_p=\|f\|_{L^p(\mathbb{V})},\  \  \|(f,g)\|_X=\|f\|_X+\|g\|_X, \\[6pt]
&D^{k,r}=\{f\in L^1_{loc}(\mathbb{V}): |\nabla^kf|_{r}<+\infty\},\quad D^k=D^{k,2}, \quad |f|_{D^k}=\|f\|_{D^k(\mathbb{V})},\\[6pt]
&
  |f|_{D^{k,r}}=\|f\|_{D^{k,r}(\mathbb{V})},\quad \int_0^\infty \int_{S^2}\int_0^\infty \int_{S^2}  f(v,\Omega,v',\Omega',t,x) \text{d}\Omega' \text{d}v'\text{d}\Omega \text{d}v=\int_\mathbb{I} f \text{d}\mathbb{I}.
\end{split}
\end{equation*}
 A detailed study of homogeneous Sobolev spaces could be found in \cite{gandi}.

Next we make some assumptions for the physical coefficients $\sigma_a$ and  $\sigma_s$. First, let
$$ \sigma_s=\overline{\sigma}_s(v' \rightarrow v, \Omega'\cdot\Omega,t,x)\rho=\overline{\sigma}_s\rho,\quad \sigma'_s=\overline{\sigma}'_s(v \rightarrow v', \Omega\cdot\Omega',t,x)\rho=\overline{\sigma}'_s\rho,$$
where the functions $\overline{\sigma}_s\geq 0$ and $\overline{\sigma}'_s\geq 0$ are $C^1$ for $(v',v, \Omega',\Omega,t,x)$, and  satisfy
\begin{equation}\label{zhen1}
\begin{cases}
\displaystyle
\int_0^\infty \int_{S^{2}}\Big(\int_{0}^{\infty}\int_{S^{2}} \Big|\frac{v}{v'}\Big|^2\big\|(\overline{\sigma}_s,\nabla_{t,x} \overline{\sigma}_s)\big\|^2_{C([0,T]\times \mathbb{V})}\text{d}\Omega' \text{d}v'\Big)^{\lambda_1}\text{d}\Omega \text{d}v\leq \alpha,\\[12pt]
\displaystyle
\int_0^\infty \int_{S^{2}}\Big(\int_{0}^{\infty}\int_{S^{2}} \big\|(\overline{\sigma}'_s,\nabla_{t,x}\overline{\sigma}'_s)\big\|^2_{C([0,T]\times \mathbb{V})}\text{d}\Omega' \text{d}v'\Big)^{\lambda_2}\text{d}\Omega \text{d}v\leq \alpha,\\[12pt]
\displaystyle
\int_{0}^{\infty}\int_{S^{2}} \big\|(\overline{\sigma}'_s, \nabla_{t,x} \overline{\sigma}'_s)\big\|_{C([0,T]\times \mathbb{V})}\text{d}\Omega' \text{d}v'\leq \alpha,\\[12pt]
\displaystyle
S(v,t,x)|_{\partial \mathbb{V}}=\overline{\sigma}_s(v' \rightarrow v, \Omega'\cdot\Omega,t,x)\big|_{\partial \mathbb{V}}=0, \quad \text{when} \quad n\cdot \Omega < 0 
\end{cases} 
\end{equation}
for $(t,x) \in [0,T] \times \mathbb{V}$, where $\lambda_1=1$ or $\frac{1}{2}$,  $\lambda_2=1$ or $2$, and $\alpha>1$  is a fixed constant.

Second, let 
$$\sigma_a=\sigma(v,\Omega,t,x,\rho,\theta)\rho=\sigma \rho,$$
then for $(\rho^i(t,x),\theta^i(t,x))$ $(i=1,2)$ satisfying 
\begin{equation*}
\begin{cases}
\displaystyle
\|\rho^i(t,x)\|_{C([0,T]; W^{1,q}(\mathbb{V}))}+\|\rho^i_{t}(t,x)\|_{C([0,T]; L^q(\mathbb{V}))}=\Lambda_1<+\infty,\\[8pt]
\displaystyle
\|\theta^i(t,x)\|_{C([0,T];H^2(\mathbb{V}))}+\|\theta^i_{t}(t,x)\|_{L^2([0,T];D^1(\mathbb{V}))}=\Lambda_2<+\infty,
\end{cases}
\end{equation*}
where $\Lambda_1$ and $\Lambda_2$ are both positive constants,  we assume that
\begin{equation}\label{jia345}
\begin{cases}
\displaystyle
 \|\sigma^i\|_{L^\infty\cap L^2(\mathbb{R}^+\times S^2;C([0,T]\times \mathbb{V}))}
\leq  \beta,\\[8pt]
\displaystyle
 \|\sigma^i\|_{L^\infty\cap L^2(\mathbb{R}^+\times S^2; D^{1,q} ( \mathbb{V}))}
\leq  M(|\rho^i|_\infty+|\theta^i|_\infty)(|\rho|_{D^{1,q}} +|\theta|_{D^{1,q}} +1),\\[8pt]
\displaystyle
\| \sigma^i_t\|_{L^2 (\mathbb{R}^+\times S^2; L^2(\mathbb{V}))}
\leq  M(|\rho^i|_\infty+|\theta^i|_\infty)(|\rho_t|_2+|\theta_t|_2+1),\\[8pt]
\displaystyle
 |\sigma(v,\Omega,t,x,\rho^i,\theta^1)-\sigma(v,\Omega,t,x,\rho^i,\theta^2)|\leq  \overline{\sigma}(v,\Omega,t,x)|\theta^1-\theta^2|, \\[8pt]
\displaystyle
 |\sigma(v,\Omega,t,x,\rho_1,\theta^i)-\sigma(v,\Omega,t,x,\rho_2,\theta^i)|\leq  \overline{\sigma}(v,\Omega,t,x)|\rho^1-\rho^2|, \\[8pt]
\displaystyle
 \|\overline{\sigma}(v,\Omega,t,x)\|_{L^\infty\cap L^2(\mathbb{R}^+\times S^2;L^\infty( \mathbb{V}))}\leq M(|(\rho^1,\rho^2)|_\infty+|(\theta^1,\theta^2)|_\infty)
\end{cases}
\end{equation}
for $t\in [0,T]$ and $r \in [2,q]$, where $\beta>1$ is a fixed positive constant independent of $\Lambda_i$ ($i=1,2$), $M=M(\cdot)$ denotes a strictly increasing continuous function from $[0,\infty)$ to $[1,\infty)$,
 and $\sigma(v,\Omega,t,x,\rho^i,\theta^i)\in C([0,T]; L^2(\mathbb{R}^+\times S^2; L^{\infty}\cap D^{1,q}(\mathbb{V}) ))$.
\begin{remark}\label{con}
Assumptions (\ref{zhen1})-(\ref{jia345}) are similar to those of \cite{sjx} \cite{sz1} for the  existence theory to  Euler-Boltzmann equations.  The evaluation of these physical coefficients is a difficult problem of  quantum mechanics and their general form is not known. An general expression of $\sigma_a$ and $\sigma_s$ used for describing Compton Scattering process  can be given as
\begin{equation}\label{kk}
\begin{split}
&\sigma_a(v,t,x,\rho,\theta)=D_1\rho \theta^{-\frac{1}{2}}\exp\Big(-\frac{D_2}{\theta^{\frac{1}{2}}}\Big(\frac{v-v_0}{v_0}\Big)^2\Big),\  \sigma_s=\overline{\sigma}_s(v \rightarrow v', \Omega\cdot\Omega',t,x)\rho,
\end{split}
\end{equation}
where $v_0$ is the fixed frequency, $D_i(i=1,2)$ are positive constants (see \cite{gp}).
\end{remark}

The rest of  this paper is organized as follows. Our main results will be shown in Section $2$.  In Section $3$, we give some important lemmas which will be used frequently in our proofs.
In Section $4$, we prove the existence of  the unique local strong solution under the boundary conditions  (\ref{fan2}) via establishing  a priori estimate which is independent of the lower bound of $\rho_0$. In Section $5$, we make a discusssion  on the necessity and sufficiency of the initial layer compatibility condition. Finally in  Section $6$,  we give the proof for the blow-up criterion that we claimed in Section $ 2.2$.

\section{Main results}
\subsection{Existence results for strong solutions with vacuum}\ \\
First, we will give the definition of strong solutions to IBVP (\ref{eq:1.2})-(\ref{eq:2.2hh}) with (\ref{fan1}) or (\ref{fan2}).

\begin{definition}[\textrm{Strong solutions with vacuum to RHD}]\label{strong}
Functions  $(I,\rho,u,\theta)$ is called  a strong solution on $\mathbb{R}^+\times S^2\times [0,T]\times\mathbb{V} $ to IBVP (\ref{eq:1.2})-(\ref{eq:2.2hh}) with (\ref{fan1}) or (\ref{fan2}), if 
\begin{enumerate}
\item
$(I,\rho,u,\theta)$  satisfies the  system (\ref{eq:1.2}) a.e. in $\mathbb{R}^+\times S^2\times (0,T)\times \mathbb{V}$;
\item
$(I,\rho,u,\theta)$ belongs to the following class with some regularities:
\begin{equation}\label{reg}\begin{split}
\Phi=&\{ (I,\rho,u,\theta)| 0\leq  I\in L^2(\mathbb{R}^+\times S^2;C([0,T]; W^{1,q})),\\
&I_t\in L^2(\mathbb{R}^+\times S^2;C([0,T]; L^{q})),\quad  0\leq \rho\in C([0,T]; W^{1,q}),\\
&  \rho_t\in C([0,T]; L^{q}),\quad  (\theta,u)\in C([0,T];H^2)\cap  L^2([0,T];D^{2,q}),\\
&(\theta_t,u_t)\in  L^2([0,T];H^1), \ (\sqrt{\rho}\theta_t,\sqrt{\rho}u_t)\in  L^\infty([0,T];L^2)\};
\end{split}
\end{equation}
\item
$ (I,\rho,u,\theta)$ satifies the corresponding initial  conditions a.e. in $\mathbb{R}^+\times S^2\times \{t=0\} \times \mathbb{V}$, and also satisfis  the corresponding boundary  conditions in the sense of traces.
\end{enumerate}
\end{definition}
As has been observed in Navier-Stokes equations \cite{CK}, the lack of a positive lower bound of $\rho_0$ should be compensated with some initial layer compatibility conditions on the initial data. Similarly, via denoting  $P^0_m=R\rho_0 \theta_0$, and
\begin{equation*}
\begin{split} A^0_r=&S(v,0,x)-\sigma_a(v,\Omega,0,x,\rho_0,\theta_0)I_0
+\int_0^\infty \int_{S^{2}} \Big(\frac{v}{v'}\sigma_s(0)I'_0
-\sigma'_s(0)I_0\Big) \text{d}\Omega' \text{d}v',
\end{split}
\end{equation*}
where $\sigma_s(0)=\overline{\sigma}_s(t=0)\rho_0$ and   $\sigma'_s(0)=\overline{\sigma}'_s(t=0)\rho_0$, we have the following  theorem:

\begin{theorem}
\label{th1}
Let assumptions (\ref{zhen1})-(\ref{jia345}) hold, and
$$\|S(v,t,x)\|_{ L^2(\mathbb{R}^+;C^1([0,\infty); W^{1,q}(\mathbb{V}))\cap C^1([0,\infty); L^1(\mathbb{R}^+; L^2(\mathbb{V})))}< +\infty.$$
\begin{enumerate}
\item
Assume the initial data $(I_0, \rho_0,u_0,\theta_0)$ satisfies the regularity condition
\begin{equation}\label{gogo}
\begin{split}
&0\leq  I_0(v,\Omega,x)\in L^2(\mathbb{R}^+\times S^2;  W^{1,q}),\ 0\leq \rho_0\in  W^{1,q},\  u_0 \in H^1_0\cap H^2,\ \theta_0 \in H^2,
\end{split}
\end{equation}
and the initial layer compatibility conditions
\begin{equation}\label{kkk}
\begin{split}
Lu_0+\nabla P^0_m+\frac{1}{c}\int_0^\infty \int_{S^2}A^0_r\Omega \text{d}\Omega \text{d}v=&\rho^{\frac{1}{2}}_0 g_1,\\
-\frac{1}{c_v}(k\triangle \theta_0+Q( u_0))-\int_0^\infty \int_{S^{2}}\frac{1}{c_v} \Big(1-\frac{u_0\cdot \Omega}{c}\Big)A^0_r \text{d}\Omega \text{d}v=&\rho^{\frac{1}{2}}_0 g_2
\end{split}
\end{equation}
for some $(g_1, g_2) \in L^2$.
Then there exists a time $T_*$ and a unique strong solution $(I,\rho,u,\theta)$ on $\mathbb{R}^+\times S^2\times [0,T_*]\times\mathbb{V} $ to IBVP (\ref{eq:1.2})-(\ref{eq:2.2hh}) with (\ref{fan1}).
\item
Assume the initial data $(I_0, \rho_0,u_0,\theta_0)$ satisfies the regularity condition
\begin{equation}\label{gogogo}
\begin{split}
& 0\leq I_0(v,\Omega,x)\in L^2(\mathbb{R}^+\times S^2;  W^{1,q}),\ 0\leq \rho_0\in  W^{1,q},\ u_0 \in  H^2,\ \theta_0 \in H^2,
\end{split}
\end{equation}
and  (\ref{kkk}).
Then there exists a  time $T_*$ and a unique strong solution $(I,\rho,u,\theta)$ on $\mathbb{R}^+\times S^2\times [0,T_*]\times\mathbb{V} $ to IBVP (\ref{eq:1.2})-(\ref{eq:2.2hh}) with (\ref{fan2}).
\end{enumerate}
\end{theorem}

Our second result can be seen as an explanation for  the compatibility  between  (\ref{gogo}) and  (\ref{kkk}). We denote V the initial vacuum set, i.e, the interior of the zero-set of the initial density in $\mathbb{V}$. 
Then when the initial vacuum domain $V$ is not so irregular, we have
\begin{theorem}
\label{th2}
Let conditions supposed in Theorem \ref{th1} hold. We assume that the initial vacuum  only appears in the far field, or $V$  has zero  $3$-D Lebesgue measure,  or the elliptic system
\begin{equation}\label{zhen101}
\begin{cases}
-\mu\triangle \phi-(\lambda+\mu)\nabla \text{div}\phi=0\quad  \text{in}\ V,\\[8pt]
\qquad \ \quad \ -\kappa\triangle h-Q( \phi)=0 \quad  \text{in}\ V
\end{cases}
\end{equation}
has only zero solution $(\phi,h)$  in $D^1_0(V)\cap D^2(V)$. Then there exists a unique (local) strong solution $(I,\rho,u,\theta)$ with the regularity shown in Definition \ref{strong} such that
\begin{equation}\label{coco}
\begin{split}
&\|I(t)-I_0\|_{ W^{1,q}(\mathbb{V})}\rightarrow 0, \quad \text{as}\ t\rightarrow 0, \ \forall \ (v,\Omega) \in \mathbb{R}^+\times S^2,\\
&\|\rho(t)-\rho_0\|_{ W^{1,q}(\mathbb{V})}+\|(u(t)-u_0,\theta(t)-\theta_0)\|_{ H^2(\mathbb{V})}\rightarrow 0, \quad \text{as}\ t\rightarrow 0,
\end{split}
\end{equation}
if and only if  initial data $(I_0,\rho_0,u_0,\theta_0)$ satisfies the  compatibility condition (\ref{kkk}).
\end{theorem}
\begin{remark}\label{initial}
Because   $(I,\rho,u,\theta)$ only satisfies (\ref{eq:1.2})-(\ref{eq:2.2hh}) in the sense of distribution ,  we know
$$I(v,\Omega,0,x)=I_0,\ \rho(0,x)=\rho_0,\  \rho u(0,x)=\rho_0 u_0,\ \rho \theta(0,x)=\rho_0 \theta_0.$$
 In the vacuum domain $V$, the relations $u(t=0,x)=u_0$ and $\theta(t=0,x)=\theta_0$  maybe not hold.
The conclusions obtained in Theorem \ref{th2} tells us  that if the vacuum domain $V$ has a sufficient simple geometry, for instance, the  Lipschitz continuous domain, we have $u(t=0,x)=u_0$ and $\theta(t=0,x)=\theta_0$ a.e. in $V$.
\end{remark}

\subsection{  Beal-Kato-Majda Blow-up criterion}\ \\
Next we naturally  consider  that the local  strong solutions to IBVP (\ref{eq:1.2})-(\ref{eq:2.2hh}) with (\ref{fan1}) or (\ref{fan2})  may cease to exist globally (see \cite{sz2}),  or what is the key point to make sure that the solution obtained in Theorem \ref{th1} could become a  global one?
The similar question has been studied for the incompressible Euler equation  by Beale-Kato-Majda (BKM) in their pioneering work \cite{TBK}, which showed that the $L^\infty$-bound of vorticity $\nabla \times u$ must blow up. Later, Ponce  \cite{pc} rephrased the BKM-criterion in terms of the deformation tensor $D(u)$. The same result as \cite{pc} has been proved by Huang-Li-Xin \cite{hup} for the $3$D compressible isentropic Navier-Stokes equations, which can be shown: if  $0 < \overline{T} < +\infty$  is the
maximum time for the strong solutions, then
\begin{equation}\label{kaka1}
\lim \sup_{ T \rightarrow \overline{T}} \int_0^T |D( u)|_{ L^\infty(\Omega)}\text{d}t=\infty.
\end{equation}
Later on, under the physical assumption  (\ref{eq:1.5})  and $\lambda <7\mu$,  Sun-Wang-Zhang \cite{zif} improved this result based on some inequalities in BMO space such that 
$$
\lim \sup_{ T \rightarrow \overline{T}} |\rho|_{ L^\infty([0,T];L^\infty(\Omega))}=\infty,
$$
which has been extended to non-isentropic flow in Wen-Zhu \cite{jiangzhu}.

 Our main result in the following   theorem shows  that the $L^\infty$ norms of $(\rho,\theta)$ and $L^2$ norm of $\nabla I$ control the possible blow-up  (see \cite{sz2}) for strong solutions, which  means that if a solution of the compressible RHD equations is initially regular and loses its regularity at some later time, then the formation of singularity must be caused  by  losing  the bound of  $\nabla I$, $\rho$ or $\theta$ as the critical time approaches. We first assume that 
\begin{equation}\label{kll}
\begin{cases}
\sigma_a=\sigma(v,\Omega,\theta)\rho=\sigma \rho,\ \sigma_\theta=\frac{\partial \sigma}{\partial \theta},\\[8pt]
 \|(\sigma,\sigma_\theta)(v,\Omega,\theta)\|_{L^2 \cap L^\infty(\mathbb{R}^+\times S^2;  L^\infty([0,T]\times \mathbb{V}))}\leq \beta,
\end{cases}
\end{equation}
which is similar to (\ref{kk}), and obviously satisfies the assumption  (\ref{jia345}).
\begin{theorem} \label{th3}
Let (\ref{gogo})-(\ref{kkk}) and (\ref{kll}) hold,  $(\mu, \lambda)$ satisfy
\begin{equation}\label{ass}
 \mu > 0, \quad 3\lambda+2\mu \geq 0,\quad \lambda < 3 \mu.\end{equation}
 If  $(I,\rho,u,\theta)$  is a strong solution obtained in Theorem \ref{th1} to IBVP  (\ref{eq:1.2})-(\ref{eq:2.2hh}) with boundary condition (\ref{fan1}),   and $0< \overline{T} <\infty$ is the maximal time of  its existence, then we have
\begin{equation}\label{eq:2.91}
\lim \sup_{ T \rightarrow \overline{T}}(|\nabla I|_{L^2(\mathbb{R}^+\times S^2;L^\infty [0,T]; L^2( \mathbb{V})))}+| (\rho,\theta)|_{L^\infty([0,T]\times \mathbb{V})})=\infty.
\end{equation}
\end{theorem}

\begin{remark}\label{rrr2}
We introduce the main ideas of our proof   for Theorems \ref{th1}-\ref{th3}, some of which are inspired by  the arguments used  in \cite{oula}\cite{CK}\cite{huangtao}\cite{hup}\cite{zif}\cite{jiangzhu}.

\text{I}) In Theorem \ref{th1},  in order to get a a priori estimate which is independent of the lower bound of $\rho_0$ under  boundary condition  (\ref{fan2}),  some new arguments have been introduced compared with \cite{CK}. First, due to  $I=0$  on $\partial \mathbb{V}$ when $n\cdot \Omega \leq 0$, we  observe that 
$$
I_t=0,\ \text{and}\ \nabla I=(\nabla I \cdot n)n, \ \text{on} \  \partial \mathbb{V}\  \text{when}\  n\cdot \Omega \leq 0,
$$
which will be used to deal with  boundary terms for the  estimate on  $\nabla I$. Second, in order to deal with terms $(|\theta_t|_6, |\theta|_6)$ under (\ref{fan2}), a  Poincar$\acute{\text{e}}$ type inequality is introduced as
\begin{equation*}\begin{split}
|\theta_t|_6\leq C(|\sqrt{\rho}\theta_t|_2+(1+|\rho|_2)|\nabla \theta_t|_2).
\end{split}
\end{equation*}
Finally, the  Minkowski's inequality (see Section $4.1$) and some regualrity  theory (see Section $3$) introduced in \cite{oula} for incompressible Euler equations has been applied to the energy estimate for the velocity $u$ of the fluid  under the Navier-slip boundary condition.

\text{II}) In Theorem \ref{th3},  in order to get a restriction of $\mu$ and $\lambda$ as better as possible,
  the crucial ingredient to relax the additional restrictions to
$\lambda < 3\mu$    has been observed that 
\begin{equation}\label{gpkk}
\displaystyle
|\nabla u|^2=|u|^2\Big| \nabla \Big(\frac{u}{|u|}\Big)\Big|^2+\big| \nabla  |u|\big|^2
\end{equation}
for  $ |u|>0$, and thus we deduce that 
\begin{equation}\label{ghukk22}
\begin{split}
\int_{\mathbb{V} \cap |u|>0}  |u|^{r-2} | \nabla  u|^2\text{d}x\geq (1+ \phi(\epsilon_0,\epsilon_1,r) )\int_{\mathbb{V} \cap |u|>0}  |u|^{r-2} \big| \nabla  |u|\big|^2\text{d}x,
\end{split}
\end{equation}
if the following assumption holds:
\begin{equation}\label{ghukk11}
\begin{split}
&\int_{\mathbb{V} \cap |u|>0}  |u|^r \Big| \nabla \Big(\frac{u}{|u|}\Big)\Big|^2\text{d}x\geq  \phi(\epsilon_0,\epsilon_1,r)\int_{\mathbb{V} \cap |u|>0}  |u|^{r-2} \big| \nabla  |u|\big|^2\text{d}x
\end{split}
\end{equation}
for some positive function $\phi(\epsilon_0,\epsilon_1,r)$ near $r=4$. The details can be seen in Lemma \ref{abs3}.

\text{III}) As it was pointed out in \cite{jiangzhu}, we need to deal with the essential difficulties caused by the nonlinear term $Q(u)$ in $(\ref{eq:1.1})_4$.  However, an important fact has been observed that
$$(P_m)_t=(\rho E_m)_t-\frac{1}{2}\big(\rho |u|^2\big)_t.$$
Then we have
\begin{equation*}
\begin{split}
-\int_{\mathbb{V}} (P_m)_t G\text{d}x=&-\int_{\mathbb{V}} (\rho E_m)_t G\text{d}x+...=-\int_{\mathbb{V}} \mathbb{T}: \nabla G\text{d}x
\leq C||u||\nabla u||_2 |\nabla G|_2+...,
\end{split}
\end{equation*}
where $G=(2\mu+\lambda)\text{div}u-P_m$  is the effective viscous flux, which plays an important role in our proof for the blow-up criterion.

\end{remark}

\section{Preliminary}

In this section, we give some important lemmas which will be used frequently in our proof.
The first  one comes from  Gagliardo-Nirenberg inequality and Poincar$\acute{\text{e}}$ inequality:

\begin{lemma}\label{gag}
For  $q\in (3,6]$, there exist constants $C> 0$ (depend  on $q$) and $C_1>0$ (depending on $\mathbb{V}$) such that for 
$$f\in D^1_0(\mathbb{V}),\quad g\in D^1_0\cap D^2(\mathbb{V}),\quad h \in W^{1,q}(\mathbb{V})$$ and $w\in H^1(\mathbb{V})$ with $w\cdot n|_{\partial \mathbb{V}}=0$, we have
\begin{equation}\label{gaga1}
\begin{split}
|f|_6\leq C|f|_{D^1_0},\quad |g|_{\infty}\leq C|g|_{D^1_0\cap D^2}, \quad |h|_{\infty}\leq C\|h\|_{W^{1,q}},\quad \|w\|_1\leq C_1|\nabla w|_2.
\end{split}
\end{equation}
\end{lemma}

In order to dealing with the Neumann   boundary condition for $\theta$,  we also need another Poincar$\acute{\text{e}}$ type inequality  (see Chapter $8$ in \cite{lion}):
\begin{lemma}\label{pang}\cite{lion}
There exists a constant $C$ depending only on $\mathbb{V}$  and $|\rho|_{r}$ $(r\geq 1)$ ($\rho\geq 0$  is a  real funciton satisfying $|\rho|_1>0$), 
such that for every $F\geq 0$ satisfying  
$$\rho F \in L^{1}(\mathbb{V}),\quad   \sqrt{\rho} F\in L^2(\mathbb{V}),\quad \nabla F\in  L^2(\mathbb{V}),$$ we have
\begin{equation*}
|F|_{6} \leq C\big( |\sqrt{\rho} F|_{2} +(1+|\rho|_2) |\nabla F|_{2} \big).
\end{equation*}
\end{lemma}

Next we consider the following homogenous Dirichlet  boundary value problem for the Lam$\acute{\text{e}}$ operator $L$: let $U = (U^1, U^2, U^3), \ F = (F^1, F^2, F^3)$ and 
\begin{equation}
\label{tvd}
Lu=-\mu \triangle U-(\mu+\lambda)\nabla \text{div} U=F\ \text{in}\ \Omega,\quad U|_{ \partial \mathbb{V}}=0.
\end{equation}
 If  $F \in W^{-1,2}(\mathbb{V})$, then there exists a unique weak solution  $U \in H^1_0(\mathbb{V})$. We begin with recalling various estimates for this system in $L^l(\mathbb{V})$ spaces, which can be seen in \cite{dd9}.

\begin{lemma}\label{tvd1}\cite{dd9} Let  $(\ref{eq:1.5})$ hold and  $\mathbb{V}$ be a bounded, smooth domain and  $l \in (1, +\infty)$. There exists a constant $C$ depending only on $\lambda$, $\mu$, $l$ and $\mathbb{V}$ such that:

(1)if $F\in L^l(\mathbb{V})$, then we have
\begin{equation}\label{tvd2}
\|U\|_{W^{2,l}}\leq C|F|_l;
\end{equation}

(2)
 if $F \in W^{-1,l}(\mathbb{V})$ (i.e., $F =\text{div}f$ with $f =(f_{ij})_{3\times3}$,  $f_{ij} \in L^l(\mathbb{V})$), then we have
\begin{equation}\label{tvd3}
\| U \|_{W^{1,l}} \leq C|f|_l.
\end{equation}

Morevoer, if $\triangle U=F$ with $\nabla U(t,x)\cdot n |_{ \partial \mathbb{V}}=0$, for weak solution $U\in H^1$,  we also have
\begin{equation}\label{tvd2kll}
\|U\|_{W^{2,l}}\leq C|F|_l,\quad \text{for}\ F\in L^l(\mathbb{V}).
\end{equation}

\end{lemma}

%
%

In the next lemma, we give the following Beale-Kato-Majda type inequality, which was proved in \cite{serr}, and will be used later to estimate $|\nabla u|_\infty$  and $ |\nabla \rho|_{q}$. 
\begin{lemma}\cite{serr}\label{tvd10} Let $\mathbb{V}$ be a bounded smooth domain and $\nabla u  \in L^2\cap D^{1,q} (\mathbb{V})$ with $q\in  (3, \infty)$. There exists a constant $C$ depending on $q$  such that
\begin{equation}\label{tvd9}
\begin{split}
\|\nabla u\|_{L^\infty(\mathbb{V})} \leq& C\big(     |\text{div}u|_\infty+|\nabla \times u|_\infty \big) \text{ln} (e+|\nabla^2  u|_q) +C|\nabla u|_2+C.
\end{split}
\end{equation}
\end{lemma}

The next two lemmas will be used to deal with the Navier-slip boundary condtions.
\begin{lemma}\cite{oula}\label{gag11}
Let $\mathbb{V}$ be a bounded smooth domain and  $u\in H^s$ be a vector valued function satisfying $u\cdot n|_{\partial \mathbb{V}}=0$, where $n$ is the unit outer normal of $\partial \mathbb{V}$, then
\begin{equation}\label{gaga1aa}
\begin{split}
\|u\|_s\leq C(\|\text{div}u\|_{s-1}+\|\text{rot}u\|_{s-1}+\|u\|_{s-1}),
\end{split}
\end{equation}
for $s\geq 1$ and the constant $C$ only depends on $s$ and $\mathbb{V}$.
\end{lemma}

\begin{lemma}\cite{wv}\label{gag22}
Let  $\mathbb{V}$ be a bounded smooth domain and $u\in D^1$ be a vector valued function satisfying 
$$u\cdot n|_{\partial \mathbb{V}}=0\quad \text{or} \quad u\times n|_{\partial \mathbb{V}}=0,$$ where $n$ is the unit outer normal of $\partial \mathbb{V}$, then for any $l\in (1,+\infty)$, there exists a constant $C$ only depdends on $l$ and $\mathbb{V}$:
\begin{equation}\label{gaga1bb}
\begin{split}
|\nabla u|_l\leq C(|\text{div}u|_{l}+|\nabla \times u|_{l}).
\end{split}
\end{equation}
\end{lemma}

Finally we show  $W^{2,p}$-estimate for Lam$\acute{\text{e}}$ operator  under  Navier-slip boundary condition.
\begin{lemma}\cite{huangtao}\label{tslip} 
For any simply connected, bounded and smooth domain $\mathbb{V}\subset \mathbb{R}^3$, $1<p<+\infty$, and   if $f \in L^p(\mathbb{V}; \mathbb{R}^3)$.  If $u\in H^2(\mathbb{V}; \mathbb{R}^3)$ is a weak solution of 
\begin{equation}\label{naslip}
L u=f\quad \text{in} \ \mathbb{V},\quad
(u\cdot n, (\nabla \times u) \cdot n)|_{\partial \mathbb{V}}=(0,0),
\end{equation}
then $u\in W^{2,p}(\mathbb{V})$, and there exists $C>0$ dependeing on $p$, $\mathbb{V}$ and $L$ such that 
\begin{equation}\label{tvd2kkkk}
|u|_{D^{2,p}}\leq C(|f|_{p}+| \nabla u|_{2}\big).
\end{equation}
\end{lemma}

\section{Well-posedness of strong solutions}
In this section, we always assume that
$$\|S(v,t,x)\|_{ L^2(\mathbb{R}^+;C^1([0,T]; W^{1,q}))\cap C^1([0,T]; L^1(\mathbb{R}^+; L^2))}< +\infty.$$
In order to prove Theorem \ref{th1},  it is sufficient to prove the exitsence of  IBVP (\ref{eq:1.2})-(\ref{eq:2.2hh}) with (\ref{fan2}).  Now we need to consider the  following linearized problem:
\begin{equation}
\label{eq:1.weeer}
\begin{cases}
\displaystyle
\rho_t+\text{div}(\rho w)=0,\\[5pt]
\displaystyle
\frac{1}{c}I_t+\Omega\cdot\nabla I=\overline{A}_r,\\[5pt]
\displaystyle
(\rho \theta)_t+
\text{div}(\rho \theta w)+\frac{1}{c_v}(P_m\text{div} w-k\triangle \theta)
=\frac{1}{c_v}(Q( w)+\overline{N}_r),\\[8pt]
\displaystyle
(\rho u)_t+\text{div}(\rho w\otimes u)
  +\nabla P_m +Lu=-\frac{1}{c}\int_0^\infty \int_{S^2}\overline{C}_r\Omega \text{d}\Omega \text{d}v,
\end{cases}
\end{equation}
where the terms $\overline{A}_r$, $\overline{B}_r$, $\overline{C}_r$ and $\overline{N}_r$ are given by
\begin{equation*}
\begin{split}
\overline{A}_r=&S-\sigma_a(\rho,\phi)I+\int_0^\infty \int_{S^2} \Big(\frac{v}{v'}\sigma_s(\rho)\psi-\sigma'_s(\rho)I \Big)\text{d}\Omega' \text{d}v',\\
\overline{B}_r=&S-\sigma_a(\rho,\phi)I+\int_0^\infty \int_{S^2} \Big(\frac{v}{v'}\sigma_s(\rho)I'-\sigma'_s(\rho)I \Big)\text{d}\Omega' \text{d}v',\\
\overline{C}_r=&S-\sigma_a(\rho,\theta)I+\int_0^\infty \int_{S^2} \Big(\frac{v}{v'}\sigma_s(\rho)I'-\sigma'_s(\rho)I \Big)\text{d}\Omega' \text{d}v',\\
\overline{N}_r=&\int_0^\infty \int_{S^2} \Big(1-\frac{w\cdot \Omega}{c}\Big) \overline{B}_r \text{d}\Omega \text{d}v,\ \sigma_a(\rho,\phi)=\sigma_a(v,\Omega,t,x,\rho,\phi),
\end{split}
\end{equation*}
and $w(t,x)\in \mathbb{R}^3$ is a known vector, $(\phi(t,x),\ \psi(v',\Omega',t,x))$ are known functions. Assume
\begin{equation}\label{tyuu}
\begin{split}
& (w,\phi,\psi)|_{t=0}=(u_0,\theta_0,I_0), \\
& (w,\phi)\in C([0,T];H^2)\cap  L^2([0,T];D^{2,q}), \  (w_t,\phi_t)\in  L^2([0,T];H^1),\\
&0\leq \psi\in L^2(\mathbb{R}^+\times S^2;C([0,T]; W^{1,q})),\ \psi_t\in L^2(\mathbb{R}^+\times S^2;C([0,T]; L^{q})),
\end{split}
\end{equation}
and  $\mathbb{V} \subset \mathbb{R}^3$  is bounded, simply connected, 
￼smooth domain with (\ref{fan2}) and
\begin{equation}
\label{py4}
\begin{split}
\psi|_{\partial \mathbb{V}}=&0,\ \Omega \cdot n\leq 0;\quad (w\cdot n, (\nabla \times w) \cdot n)|_{\partial \mathbb{V}}=(0,0);\quad \nabla \phi \cdot n |_{\partial \mathbb{V}}=0.
\end{split}
\end{equation}

\subsection{A priori estimate to the linearized problem away from vacuum} We immediately  have the global existence of a unique strong solution $(I,\rho,u,\theta)$ to  (\ref{eq:1.weeer})-(\ref{py4}) by the standard methods at least for the case that $\rho_0$ is away from vacuum.

 \begin{lemma}\label{lem1}
 Assume in addition to (\ref{tyuu}) that $\rho_0\geq \delta$ for some constant $\delta>0$ and (\ref{kkk})-(\ref{gogogo}).
Then there exists a unique strong solution $(I,\rho,u,\theta)$ in $\mathbb{R}^+\times S^2\times [0,T] \times \mathbb{V}$ to IBVP  (\ref{eq:1.weeer})-(\ref{py4}) with (\ref{fan2}) and 
\begin{equation*}\begin{split}
& I\in  C([0,T]; L^2(\mathbb{R}^+\times S^2;  L^{q}(\mathbb{V}))),\  \text{and} \ \rho \geq \underline{\delta}
\end{split}
\end{equation*}
for some constant $\underline{\delta}>0$.
\end{lemma}
\begin{proof}
First, 
the existence  of a unique solution $\rho$ to $ (\ref{eq:1.weeer})_1$ can be obtained by the standard theory of transport equation (see Lemma 6 in  \cite{CK}), and  $\rho$ can be written as
\begin{equation}
\label{weq:bb1}
\rho(t,x)=\rho_0(U(0;t,x))\exp\Big(-\int_{0}^{t}\textrm{div} w(s,U(s;t,x))\text{d}s\Big),
\end{equation}
where  $U\in C([0,T]\times[0,T]\times \mathbb{V})$ is the solution to the initial value problem
\begin{equation}
\label{eq:bb1}
\begin{cases}
\displaystyle
\frac{d}{ds}U(s;t,x)=w(s,U(s;t,x)),\quad 0\leq s\leq T,\\[8pt]
\displaystyle
U(t;t,x)=x, \quad x\in \mathbb{V},\quad 0\leq t\leq T.
\end{cases}
\end{equation}
So we can get the lower bound of $\rho$ easily.

Second, $(\ref{eq:1.weeer})_2$ can be written into
\begin{equation}\label{rst}
\frac{1}{c}I_t+\Omega\cdot\nabla I+\Big(\sigma_a+\int_0^\infty \int_{S^2} \sigma'_s(\rho)\text{d}\Omega' \text{d}v'\Big)I=F(v,\Omega,t,x),
\end{equation}
where
\begin{equation}\label{biao1}F=S+\int_0^\infty \int_{S^2} \frac{v}{v'}\sigma_s(\rho)\psi\text{d}\Omega' \text{d}v'\in L^2(\mathbb{R}^+\times S^2;C([0,T]; W^{1,q})),
\end{equation}
then we easily get the existence and regularity of a unique solution $I$ to (\ref{rst}) that
 $$
 I\in L^2(\mathbb{R}^+\times S^2;C([0,T]; W^{1,q})),\ I_t\in L^2(\mathbb{R}^+\times S^2;C([0,T]; L^{q})).
 $$
According to the classical imbedding theory for Sobolev spaces, we  have
$$ I\in C([0,T]; L^2(\mathbb{R}^+\times S^2;  L^{q}(\mathbb{V}))).$$
Next the radiation transfer equation $ (\ref{eq:1.weeer})_2$ can be written as 
\begin{equation}\label{trans}
\frac{1}{c}I_t+c\Omega \cdot \nabla I+\mathbb{H}I=F,\quad \mathbb{H}=\sigma_a+\int_0^\infty \int_{S^2} \sigma'_s(\rho)\text{d}\Omega' \text{d}v'.
\end{equation}
We denote by $ x(t;x_0)$ the photon path starting from $x_0$ when $t=\tau$, i.e.,
      $$   \frac{d}{dt} x(t;\tau,x_0)=c\Omega ,     \quad  x(\tau;\tau,x_0)= x_0 .                   $$
Along the photon path,
we have
     \begin{equation}\begin{split} \label{eq:**}
&I(v,\Omega,t,x(t;\tau,x_0 ))=\exp\Big(\int_\tau^t -c\mathbb{H}(v,\Omega,s,x(s;\tau,x_0 ),\rho,\theta)\text{d}s\Big)\Big(I(v,\Omega,\tau,x_0)\\
&\qquad+\int_\tau^t F(v,\Omega,s,x(s;\tau,x_0 ),\rho,\theta)\exp\Big(\int_\tau^s c\mathbb{H}(v,\Omega,l,x(l;\tau,x_0 ),\rho,\theta)\text{d}l\Big)\text{d}s \Big)\geq 0
\end{split}
\end{equation}
for  $x_0=x-c\Omega (t-\tau) $,  which implies  that $I$ is nonnegative.

Finally,   it is not difficult to solve $(\theta, u)$ from the linear parabolic equations
\begin{equation}\label{rst1}
\begin{split}
&\theta_t+
 w\cdot \nabla\theta +\frac{R}{c_v}\theta\text{div} w-\frac{\kappa}{c_v}\rho^{-1}\triangle \theta
=\frac{1}{c_v}\rho^{-1}(Q( w)+\overline{N}_r),\\
\displaystyle
&u_t+ w\cdot\nabla u
 +\rho^{-1}Lu= -\rho^{-1}\nabla P_m -\frac{1}{c}\rho^{-1}\int_0^\infty \int_{S^2}\overline{C}_r\Omega \text{d}\Omega \text{d}v,
\end{split}
\end{equation}
to complete the proof of this lemma.  Here we omit the details.
\end{proof}

Now we will give a priori estimates for the solution $(I,\rho,u,\theta)$ obtained in Lemma \ref{lem1}, which is independent of the lower bound of $\rho_0$.
We first fix a positive constant $c_0$ that
\begin{equation*}\begin{split}
&2+\|I_0\|_{L^2(\mathbb{R}^+\times S^2;  W^{1,q}) } +\|\rho_0\|_{ W^{1,q}}+\|(\theta_0,u_0)\|_{2}\\
\displaystyle
& +|(g_1,g_2)|_2+\|S(v,t,x)\|_{ L^2(\mathbb{R}^+;C^1([0,T]; W^{1,q}))\cap C^1([0,T]; L^1(\mathbb{R}^+; L^2))}\leq c_0,
\end{split}
\end{equation*}
and positive constants $c_i$ ($i=1,2,...,5$) that
\begin{equation}\label{mini}\begin{split}
\|\psi\|^2_{L^2(\mathbb{R}^+\times S^2;C([0,T^*]; W^{1,q}))}\leq c^2_1,\quad \|\psi_t\|^2_{L^2(\mathbb{R}^+\times S^2;C([0,T^*]; L^{q}))}\leq & Cc^2_2,\\
\sup_{0\leq t\leq T^*}\|w(t)\|^2_{1}+\int_{0}^{T^*}\Big(|w|^2_{D^{2,q}}+|w|^2_{D^{2}}+|w_t|^2_{D^1}\Big)\text{d}t\leq c^2_2,\quad \sup_{0\leq t\leq T^*}|w(t)|^2_{D^2}\leq& c^2_3,\\
 \sup_{0\leq t\leq T^*}\|\phi(t)\|^2_{1}+\int_{0}^{T^*}\Big(|\phi|^2_{D^{2,q}}+|\phi|^2_{D^{2}}+|\phi_t|^2_{D^1}\Big)\text{d}t\leq c^2_4,\quad \sup_{0\leq t\leq T^*}|\phi(t)|^2_{D^2}\leq& c^2_5,
\end{split}
\end{equation}
for some time $T^*\in (0,T)$ and constants $c_i$ ($i=1,2,3,4,5$) such that 
$$
1<c_0<c_1<c_2<c_3<c_4<c_5.
$$
The constants $c_i$ ($i=1,2,3,4,5$) and $T^*$ will be determined later and depend only on $c_0$ and the fixed constants $\alpha$, $\beta$, $\kappa$, q, $R$, $c_v$, $\mu$, $\lambda$,  $c$, $|\mathbb{V}| $ and $T$.
In subsections $4.1$-$4.3$, $M=M(\cdot): [0,+\infty) \rightarrow[1,+\infty)$ still denotes a strictly increasing continuous function, and $C\geq 1$  denotes  a generic  constant. Both $M(\cdot)$ and $C$    depend only on fixed constants $\alpha$, $\beta$, $\kappa$, q, $R$, $c_v$, $\mu$, $\lambda$,  $c$,  $|\mathbb{V}| $ and $T$.
 We start with the estimates for $\rho$.

\begin{lemma}\label{lem:2}
$$
\|\rho(t)\|_{ W^{1,q}}\leq Cc_0,  \quad  |\rho_t(t)|_{ q}\leq Cc_0c_3
$$
for  $0\leq t \leq T_1=\min(T^*,(1+c^2_3)^{-1})$.
 \end{lemma}
\begin{proof}
From the standard energy estimate for transport equation and (\ref{fan2}), we have
\begin{equation*}\begin{split}
\|\rho(t)\|_{W^{1,q}}\leq \|\rho_0\|_{W^{1,q}} \exp\Big(C\int_0^t \|\nabla w(s)\|_{W^{1,q}}\text{d}s\Big),
\end{split}
\end{equation*}
where we have used the fact that $w\cdot n|_{\partial \mathbb{V}}=0$.
Therefore, observing that
$$
\int_0^t \|\nabla w(s)\|_{W^{1,q}}\text{d}s\leq t^{\frac{1}{2}}\Big(\int_0^t \|\nabla w(s)\|^2_{W^{1,q}}\text{d}s\Big)^{\frac{1}{2}}\leq C(c_3t+c_3t^{\frac{1}{2}})\leq C,
$$
for $0\leq t\leq T_1=\min(T^*,(1+c^2_3)^{-1})$, then  the desired estimate for $\rho$ is available. 

For the term $\rho_t$,   from the continuity equation $ (\ref{eq:1.weeer})_1$,
we get
\begin{equation*}\begin{split}
|\rho_t|_q\leq C(|\rho|_\infty|\nabla w|_q+|w|_\infty|\nabla \rho|_q)\leq Cc_0|| w||_2\leq Cc_0c_3.
\end{split}
\end{equation*}

\end{proof}
Next we show the estimates for $I$.
\begin{lemma}\label{lem:3}
$$
\|I\|^2_{L^2(\mathbb{R}^+\times S^2;C([0,T_2]; W^{1,q}))}\leq Cc^2_0 ,\quad  \|I_t\|^2_{L^2(\mathbb{R}^+\times S^2;C([0,T_2]; L^{q}))}\leq Cc_0c_1,
$$
for  $T_2= \min(T^*,(1+M(c_5)c^4_5)^{-1})$.
 \end{lemma}
\begin{proof}First, multiplying $ (\ref{eq:1.weeer})_2$ by $q|I|^{q-2}I$ and integrating over $\mathbb{V}$,  from H\"older's inequality, we have
\begin{equation}\label{thyt}
\begin{split}
&\frac{d}{dt}|I|^q_{q}+\int_{\partial \mathbb{V}}|I|^q n \cdot \Omega \text{d}\nu
\leq   C|S|_{q}|I|^{q-1}_{q}
+C|\rho|_\infty|I|^{q-1}_{q}\int_0^\infty \int_{S^2} \frac{v}{v'}|\psi|_{q} |\overline{\sigma}_s|_\infty \text{d}\Omega' \text{d}v',
\end{split}
\end{equation}
where we  used the fact $\sigma_a\geq 0$ and $\sigma'_s\geq 0$.

Second, differentiating $(\ref{eq:1.weeer})_2$ $\zeta$-times ($|\zeta|=1$) with respect to $x$, multiplying the resulting equation by $q|\partial^\zeta_xI|^{q-2} \partial^\zeta_xI$ and  integrating over $\mathbb{V}$,  we have
\begin{equation}\label{thytffg}
\begin{split}
&\frac{1}{c}\frac{d}{dt}|\partial^\zeta_x I|^q_{q}+\int_{\partial \mathbb{V}}|\partial^\zeta_x I |^q n \cdot \Omega \text{d}\nu\\
\leq& C\Big(|\partial^\zeta_x S|_{q}+|\sigma_a|_{D^{1,q}}|I|_{\infty}\Big)|\partial^\zeta_x I|^{q-1}_{q}+C|\nabla\rho|_{q}|I|_{\infty}|\partial^\zeta_x I|^{q-1}_{q}\int_0^\infty \int_{S^2} |\overline{\sigma}'_s|_\infty\text{d}\Omega' \text{d}v'\\
&+C|\rho|_{\infty}|\partial^\zeta_x I|^{q-1}_{q}\Big(|I|_{q}\int_0^\infty \int_{S^2} |\partial^\zeta_x\overline{\sigma}'_s|_\infty\text{d}\Omega' \text{d}v'+\int_0^\infty \int_{S^2} \frac{v}{v'}|\psi|_{q}  |\partial^\zeta_x\overline{\sigma}_s|_\infty \text{d}\Omega' \text{d}v'\Big)\\
&+C|\partial^\zeta_x I|^{q-1}_{q}\int_0^\infty \int_{S^2} \frac{v}{v'}|\overline{\sigma}_s|_\infty\big(|\psi|_{D^{1,q}} |\rho|_{\infty} 
+|\psi|_{\infty} |\nabla \rho|_{q}\big) \text{d}\Omega' \text{d}v',
\end{split}
\end{equation}
where we also  used  $\sigma_a\geq 0$ and $\sigma'_s\geq 0$. Now considering the boundary term $I_{\partial \mathbb{V}}=\int_{\partial \mathbb{V} \cap \{n\cdot \Omega \leq 0\}}\|  I \|^q_{W^{1,q}} n \cdot \Omega \text{d}\nu$,
due to  $I=0$  on $\partial \mathbb{V}$ when $n\cdot \Omega \leq 0$, we  obtain that 
\begin{equation}\label{ppy1}
I_t=I= 0,\quad  \nabla I=(\nabla I \cdot n)n \ \text{on} \  \partial \mathbb{V},\  \text{when}\  n\cdot \Omega \leq 0,
\end{equation}
which, together with the assumption (\ref{zhen1}) and  $(\ref{eq:1.weeer})_2$, implies that
\begin{equation}\label{ppy2}
\Omega \cdot \nabla I=S+\int_0^\infty \int_{S^2} \frac{v}{v'}\overline{\sigma}_s \rho\psi\text{d}\Omega' \text{d}v'=0, \ \text{on} \  \partial \mathbb{V}\  \text{when}\  n\cdot \Omega \leq 0.
\end{equation}
Then, via (\ref{ppy1}), $I_{\partial \mathbb{V}}$ can be written as
\begin{equation}\label{ppy3}
\begin{split}
I_{\partial \mathbb{V}}=&\int_{\partial \mathbb{V} \cap \{n\cdot \Omega \leq 0\}}|\nabla I |^q n \cdot \Omega \text{d}\nu
=\int_{\partial \mathbb{V} \cap \{n\cdot \Omega \leq 0\}}|\nabla I |^{q-2} [\nabla I \cdot \Omega][\nabla I \cdot n] \text{d}\nu=0.
\end{split}
\end{equation}

Then from (\ref{thyt})-(\ref{ppy1}), (\ref{ppy3}) and  assumptions (\ref{zhen1})-(\ref{jia345}), we have
\begin{equation}\label{thytq}
\begin{split}\frac{d}{dt}\|I\|^2_{W^{1,q}}
\leq&  C\big(1+c_0\alpha+|\sigma_a|_{D^{1,q}}\big)\|I\|^2_{W^{1,q}}+C\| S\|^2_{W^{1,q}}\\
&+Cc^2_0\int_0^\infty \int_{S^2} \Big|\frac{v}{v'}\Big|^2 |\partial^\zeta_x\overline{\sigma}_s|^2_\infty\text{d}\Omega' \text{d}v'\cdot\int_0^\infty \int_{S^2} |\psi|^2_{q}  \text{d}\Omega' \text{d}v'\\
&+Cc^2_0\int_0^\infty \int_{S^2} \Big|\frac{v}{v'}\Big|^2|\overline{\sigma}_s|^2_\infty\text{d}\Omega' \text{d}v'\cdot\int_0^\infty \int_{S^2} \|\psi\|^2_{W^{1,q}} \text{d}\Omega' \text{d}v'\\
\leq&M(c_5)c_0c_5\|I\|^2_{W^{1,q}}+C\|S\|^2_{W^{1,q}}+Cc^4_1\int_0^\infty \int_{S^2} \Big|\frac{v}{v'}\Big|^2 \|\overline{\sigma}_s\|^2_{W^{1,\infty}} \text{d}\Omega' \text{d}v',
\end{split}
\end{equation}
where we have used the fact that 
$$
|\sigma_a|_{D^{1,q}}\leq \big( |\rho|_{\infty} |\nabla \sigma|_q +|\sigma|_{\infty} |\nabla \rho|_q\big)\leq M(c_5)c_0c_5.
$$
From Gronwall's inequality and (\ref{thytq}), we have
\begin{equation}\label{fude}
\begin{split}
&\|I(v,\Omega,t,x)\|^2_{C([0,T_2]; W^{1,q})}\leq \exp (M(c_5)c^2_5T_2)\|I_0\|^2_{W^{1,q}}\\
&+ \exp (M(c_5)c^2_5T_2)\Big(\int_0^{T_2}\|S\|^2_{W^{1,q}}\text{d}s+c^4_1T_2\int_0^\infty \int_{S^2} \Big|\frac{v}{v'}\Big|^2  \|\overline{\sigma}_s\|^2_{W^{1,\infty}} \text{d}\Omega' \text{d}v'\Big)
\end{split}
\end{equation}
for $ T_2= \text{min}(T^*,(1+M(c_5)c^4_5)^{-1})$.
Then integrating above inequality in $\mathbb{R}^+\times S^2$ with respect to $(v,\Omega)$, via (\ref{zhen1})-(\ref{jia345}), we have
$$
\|I\|^2_{L^2(\mathbb{R}^+\times S^2;C([0,T_2]; W^{1,q}))}\leq Cc^2_0,\quad 
\text{for}\quad 0\leq t\leq T_2.
$$

Finally,
due to $ I_t=-c\Omega\cdot\nabla I+c\overline{A}_r$ and Minkowski's inequality, we have
\begin{equation*}\begin{split}
& \|I_t\|_{L^2(\mathbb{R}^+\times S^2;C([0,T_2]; L^{q}))}\\
\leq & C\|\nabla I\|_{L^2(\mathbb{R}^+\times S^2;C([0,T_2]; L^{q}))}+C\|\overline{A}_r\|_{L^2(\mathbb{R}^+\times S^2;C([0,T_2]; L^{q}))}\\
\leq &Cc_0+C\|S\|_{L^2(\mathbb{R}^+;C([0,T_2]; L^{q}))}+Cc_0|I|_{L^2(\mathbb{R}^+\times S^2;C([0,T_2]; L^{q}))}\|\sigma\|_{ L^2(\mathbb{R}^+\times S^2;C([0,T]\times \mathbb{V}))}\\
&+Cc_0|\psi|_{L^2(\mathbb{R}^+\times S^2;C([0,T_2]; L^{q}))}\int_0^\infty \int_{S^{2}}\Big(\int_{0}^{\infty}\int_{S^{2}} \Big|\frac{v}{v'}\Big|^2|\overline{\sigma}_s|^2_\infty\text{d}\Omega' \text{d}v'\Big)^{\frac{1}{2}}\text{d}\Omega \text{d}v\\
&+Cc_0|I|_{L^2(\mathbb{R}^+\times S^2;C([0,T_2]; L^{q}))}\Big(\int_{0}^{\infty}\int_{S^{2}}\Big(\int_{0}^{\infty}\int_{S^{2}} |\overline{\sigma}'_s|_\infty\text{d}\Omega' \text{d}v'\Big)^2\text{d}\Omega \text{d}v\Big)^{\frac{1}{2}} \leq Cc_0c_1.
\end{split}
\end{equation*}
\end{proof}

Next we give the estimate for $\theta$.
\begin{lemma}\label{lem:3ejia}
\begin{equation*}
\begin{split}
\|\theta(t)\|^2_{1}+|\sqrt{\rho}\theta(t)|^2_2+\int_{0}^{t}|\sqrt{\rho}\theta_t(s)|^2_2\text{d}s\leq Cc^7_2c_3,
\end{split}
\end{equation*}
for $  0\leq t \leq T_3=\min(T^*,(1+M(c_5)c^8_5)^{-1})$.
 \end{lemma}
\begin{proof}

\underline{Step 1}. 
Multiplying $(\ref{eq:1.weeer})_3$ by $\theta$ and integrating over $\mathbb{V}$, we have
\begin{equation}\label{zhou111}
\begin{split}
&\frac{\kappa}{c_v} \int_{\mathbb{V}}|\nabla \theta |^2 \text{d}x+\frac{1}{2}\frac{d}{dt}\int_{\mathbb{V}}\rho |\theta|^2 \text{d}x\\
\leq& C \int_{\mathbb{V}} \Big(\rho|\theta|^2 |\text{div}w|   +|\nabla w|^2| \theta|+|\overline{N}_r| |\theta|\Big)\text{d}x\\
\leq& C\Big(|\rho|^{\frac{1}{2}}_\infty|\sqrt{\rho}\theta|_2|\nabla w|_3+|\nabla w|_2|\nabla w|_3+|\overline{N}_r|_{6/5}\Big)|\theta|_6\\
\leq &\frac{\kappa}{20c_vc^2_0}(|\sqrt{\rho}\theta|^2_2+c^2_0|\nabla \theta|^2_2)+Cc^5_3|\sqrt{\rho}\theta|^2_2+Cc^8_3,
\end{split}
\end{equation}
where we have used  the Poincar$\acute{\text{e}}$ type inequality for $\theta$ in  Lemma \ref{pang},
and the fact:

\begin{equation}\label{zhouzhan}
\begin{split}
|\overline{N}_r|_{6/5}\leq& C|(1+|w|)\overline{B}_r|_{6/5} 
\leq C(1+|w|_\infty) \Big(\int_0^\infty \int_{S^2} \Big(|S|_{6/5}+|\sigma|_\infty|\rho|_3|I|_2\Big) \text{d}\Omega \text{d}v\\
&+|\rho|_3\int_0^\infty \int_{S^2} \Big(\int_0^\infty \int_{S^2}\Big( \frac{v}{v'}|\overline{\sigma}_s|_\infty|I'|_2+|\overline{\sigma}'_s|_\infty |I|_2 \Big)\text{d}\Omega' \text{d}v\Big) \text{d}\Omega \text{d}v\Big)\\
\leq & C(1+|w|_\infty)\Big(|S|_{L^1(\mathbb{R}^+;L^{\frac{6}{5}})}+(\alpha+\beta)|\rho |_3|I|_{L^2(\mathbb{R}^+\times S^2;L^2)}\Big)\leq Cc^3_3.
\end{split}
\end{equation}

Then from Gronwall's inequality and  (\ref{zhou111}), we have
\begin{equation}\label{zhouzhan1}\begin{split}
\int_0^t |\nabla \theta|^2_2 \text{d}s+|\sqrt{\rho} \theta|^2_2\leq C(c^3_0+c^8_3t)\exp (Cc^5_3t)\leq Cc^3_0
\end{split}
\end{equation}
for $0\leq t\leq T_3= \min(T^*,(1+M(c_5)c^8_5)^{-1})$.

\underline{Step 2}. 
Multiplying $(\ref{eq:1.weeer})_3$ by $\theta_t$ and integrating over $\mathbb{V}$, we have
\begin{equation}\label{zhou122}
\begin{split}
&\frac{\kappa}{2c_v}\frac{d}{dt} \int_{\mathbb{V}}|\nabla \theta |^2 \text{d}x+\int_{\mathbb{V}}\rho |\theta_t|^2 \text{d}x\\
\leq&  \int_{\mathbb{V}} \Big(C\rho|\theta| |\text{div}w| |  \theta_t|  + C\rho |w|  | \nabla \theta| |  \theta_t|  +Q(w) \theta_t+\overline{N}_r\theta_t\Big)\text{d}x\\
\leq & C|\sqrt{\rho}\theta_t|_2|\rho|^{\frac{1}{2}}_\infty|\theta|_6|\nabla w|_3+C|\sqrt{\rho}\theta_t|_2|\rho|^{\frac{1}{2}}_\infty|\nabla \theta|_2| w|_\infty+ \int_{\mathbb{V}} \Big(Q(w) \theta_t+\overline{N}_r\theta_t\Big)\text{d}x.
\end{split}
\end{equation}
For the last term on the right-hand side of (\ref{zhou122}), we have
\begin{equation}\label{zhou133}
\begin{split}
\int_{\mathbb{V}} \Big(Q(w) \theta_t+\overline{N}_r\theta_t\Big)\text{d}x=&\frac{d}{dt}\int_{\mathbb{V}} Q(w) \theta\text{d}x-\int_{\mathbb{V}} Q(w)_t \theta\text{d}x+\int_{\mathbb{V}}\overline{N}_r\theta_t\text{d}x\\
\leq &\frac{d}{dt}\int_{\mathbb{V}} Q(w) \theta\text{d}x+C|\nabla w_t|_2|\nabla w|_3 |\theta|_6+\sum_{i=1}^4R_{i}.
\end{split}
\end{equation}
For  terms $R_{1}$-$R_4$,  from H\"older's inequality, we have
\begin{equation}\label{zhou155}
\begin{split}
R_{1}=&\frac{1}{c_v}\int_0^\infty \int_{S^2} \int_{\mathbb{V}} \Big(1-\frac{w\cdot \Omega}{c}\Big) S\theta_t \text{d}x \text{d}\Omega \text{d}v\\
=&\frac{1}{c_v}\frac{d}{dt}\int_0^\infty \int_{S^2} \int_{\mathbb{V}} \Big(1-\frac{w\cdot \Omega}{c}\Big) S\theta \text{d}x \text{d}\Omega \text{d}v\\
&-\frac{1}{c_v}\int_0^\infty \int_{S^2} \int_{\mathbb{V}} \Big(\Big(1-\frac{w\cdot \Omega}{c}\Big) S_t-\frac{w_t\cdot \Omega}{c}S\Big)\theta \text{d}x \text{d}\Omega \text{d}v\\
\leq & \frac{1}{c_v}\frac{d}{dt}\int_0^\infty \int_{S^2} \int_{\mathbb{V}} \Big(1-\frac{w\cdot \Omega}{c}\Big) S\theta \text{d}x \text{d}\Omega \text{d}v\\
&+C|\theta|_6(1+|w|_3)\|S_t\|_{L^1(\mathbb{R}^+\times S^2; L^2)}+C|\theta|_6|w_t|_3\|S\|_{L^1(\mathbb{R}^+; L^2)},\\
R_{2}=&-\frac{1}{c_v}\int_0^\infty \int_{S^2} \int_{\mathbb{V}} \Big(1-\frac{w\cdot \Omega}{c}\Big) \sigma_a I\theta_t \text{d}x \text{d}\Omega \text{d}v\\
\leq &  C(1+|w|_{3})|\rho|^{\frac{1}{2}}_\infty|\sqrt{\rho}\theta_t|_{2}\int_0^\infty \int_{S^2} |\sigma|_\infty|I|_6\text{d}\Omega \text{d}v\\
\leq & C\beta(1+|w|_{3})|\rho|^{\frac{1}{2}}_\infty|\sqrt{\rho}\theta_t|_{2}\|I\|_{L^2(\mathbb{R}^+\times S^2;L^6)}\leq Cc^{\frac{5}{2}}_3|\sqrt{\rho}\theta_t|_{2},\\
R_{3}=&\frac{1}{c_v}\int_\mathbb{I}\int_{\mathbb{V}}\frac{v}{v'}\Big(1-\frac{w\cdot \Omega}{c}\Big)\sigma_s I'
\theta_t\text{d}x \text{d}\mathbb{I}\\
 \leq &C(1+|w|_{3})|\rho|^{\frac{1}{2}}_\infty|\sqrt{\rho}\theta_t|_{2}\int_\mathbb{I}  \frac{v}{v'}
|\overline{\sigma}_s|_\infty |I'|_6\text{d}\mathbb{I}\\
\leq & C\alpha(1+|w|_{3})|\rho|^{\frac{1}{2}}_\infty|\sqrt{\rho}\theta_t|_{2}\|I\|_{L^2(\mathbb{R}^+\times S^2;L^6)}\leq Cc^{\frac{5}{2}}_3|\sqrt{\rho}\theta_t|_{2},\\
R_{4}=&-\frac{1}{c_v}\int_\mathbb{I} \int_{\mathbb{V}}\Big(1-\frac{w\cdot \Omega}{c}\Big)\sigma'_s I
\theta_t\text{d}x \text{d}\mathbb{I}\\
\leq &C(1+|w|_{3})|\rho|^{\frac{1}{2}}_\infty|\sqrt{\rho}\theta_t|_{2} \int_\mathbb{I}
|\overline{\sigma}'_s|_\infty |I|_6  \text{d}\mathbb{I}\\
\leq &C\alpha (1+|w|_{3})|\rho|^{\frac{1}{2}}_\infty|\sqrt{\rho}\theta_t|_{2}\|I\|_{L^2(\mathbb{R}^+\times S^2;L^6)}\leq Cc^{\frac{5}{2}}_3|\sqrt{\rho}\theta_t|_{2}.
\end{split}
\end{equation}
 
Combining (\ref{zhou122})-(\ref{zhou155}), from Lemma \ref{pang},  Young's inequality and (\ref{zhouzhan1}), we have
\begin{equation}\label{zhou188}
\begin{split}
&\frac{\kappa}{2c_v}\frac{d}{dt} \int_{\mathbb{V}}|\nabla \theta |^2 \text{d}x+\frac{1}{2}\int_{\mathbb{V}}\rho |\theta_t|^2 \text{d}x-\frac{d}{dt}\int_{\mathbb{V}} Q(w) \theta\text{d}x\\
\leq&  \frac{1}{c_v}\frac{d}{dt}\int_0^\infty \int_{S^2} \int_{\mathbb{V}} \Big(1-\frac{w\cdot \Omega}{c}\Big) S\theta \text{d}x \text{d}\Omega \text{d}v+C|\nabla w_t|^2_2+Cc^5_3|\nabla \theta |^2_2+Cc^6_3.
\end{split}
\end{equation}
It is not hard to see
\begin{equation*}\begin{split}
\int_{\mathbb{V}} Q(w) \theta\text{d}x\leq &C|\theta|_6 |\nabla w|_2|\nabla w|_3\leq C(|\sqrt{\rho}\theta|_2+c_0|\nabla \theta|_2)|\nabla w|^{\frac{3}{2}}_2|\nabla w|^{\frac{1}{2}}_6\\
\leq & C(c_0|\nabla \theta|_2+c^{\frac{3}{2}}_0)(c^{2}_2+c^{\frac{3}{2}}_2c^{\frac{1}{2}}_3)\leq \frac{\kappa}{4c_v}|\nabla \theta|^2_2+Cc^5_2c_3,\\
\frac{1}{c_v} \int_0^\infty \int_{S^2} \int_{\mathbb{V}} & \Big(1-\frac{w\cdot \Omega}{c}\Big) S\theta \text{d}x \text{d}\Omega \text{d}v\leq C(1+|w|_3)|\theta|_6|S|_{L^1(\mathbb{R}^+;L^2)}\\
\leq &Cc^2_2(c_0|\nabla \theta |_2+c^{\frac{3}{2}}_0)\leq \frac{\kappa}{4c_v}|\nabla \theta|^2_2+Cc^6_2,
\end{split}
\end{equation*}
which, along with (\ref{zhou188}) and  Gronwall's inequality, implies that 
$$
|\nabla \theta|^2_2+\int_0^t |\sqrt{\rho} \theta_t|^2_2\text{ds}\leq C(c^5_2c_3+c^6_3t)\exp\big(c^5_3 t\big)\leq Cc^5_2c_3,
$$
for $  0\leq t \leq T_3$, which implies that 
$$
|\theta|_2\leq C|\theta|_6\leq C(|\sqrt{\rho}\theta|_2+c_0|\nabla \theta|_2)\leq Cc^{\frac{7}{2}}_2c^{\frac{1}{2}}_3.
$$

\end{proof}

\begin{lemma}\label{lem:3e}
\begin{equation*}
\begin{split}
|\theta(t)|^2_{D^2}+|\sqrt{\rho}\theta_t(t)|^2_{2}+\int_{0}^{t}\big(|\theta(s)|^2_{D^{2,q}}+|\theta_t(s)|^2_{D^1}\big)\text{d}s\leq Cc^{12}_3,
\end{split}
\end{equation*}
for  $0\leq t \leq T_4=\min(T^*,(M(c_5)c^{21}_5)^{-1})$.
 \end{lemma}
\begin{proof}
First differentiating $(\ref{eq:1.weeer})_3$ with respect to $t$,  we have
\begin{equation}\label{pengyue1}
\begin{split}\rho \theta_{tt}-\frac{\kappa}{c_v}\triangle \theta_t
=-\rho_t \theta_{t}
-(\rho w\cdot  \nabla \theta)_t-\frac{1}{c_v}\Big((P_m\text{div} w)_t-Q( w)_t-\big(\overline{N}_r\big)_t\Big).\end{split}
\end{equation}
Multiplying (\ref{pengyue1}) by $\theta_t$ and integrating over $\mathbb{V}$, we have
\begin{equation}\label{zhou1}
\begin{split}
&\frac{1}{2}\frac{d}{dt} \int_{\mathbb{V}}\rho |\theta_t|^2 \text{d}x+\frac{\kappa}{c_v}\int_{\mathbb{V}}|\nabla \theta_t|^2 \text{d}x\\
\leq& C \int_{\mathbb{V}} \Big(|\rho_t| |w| | \nabla \theta|  + \rho |w_t|  | \nabla \theta|  +\rho |w|  | \nabla \theta_t| +|(P_m)_t|  |\text{div} w|
\Big) | \theta_t|   \text{d}x\\
&+\int_{\mathbb{V}} \Big(  \rho |\theta|  | \text{div} w_t| +|\nabla w||\nabla w_t|+|\overline{N}_r|_t\Big) | \theta_t|  \text{d}x
\equiv:\sum_{i=1}^{6}I_i+E_I,
\end{split}
\end{equation}
where the radiation source term:
\begin{equation*}
\begin{split}
E_I=\frac{1}{c_v}\int_0^\infty \int_{S^2} \int_{\mathbb{V}}\Big( \Big(1-\frac{w\cdot \Omega}{c}\Big) (\overline{B}_r)_t \theta_t+ \Big(-\frac{w_t\cdot \Omega}{c}\Big) \overline{B}_r\theta_t\Big) \text{d}x \text{d}\Omega \text{d}v=\sum_{j=1}^{8} J_j.
\end{split}
\end{equation*}

Second, we estimate $\sum_{i=1}^{6}I_i$.  According to Lemmas \ref{gaga1}-\ref{pang},  \ref{lem:2}-\ref{lem:3ejia}, H\"older's  and Young's inequality, we have
\begin{equation*}
\begin{split}
I_1\leq& C|\rho_t|_{3} |w|_{\infty} |\nabla \theta|_{2} |  \theta_t|_{6}\leq Cc^{16}_3+\frac{\kappa}{20c^2_0c_v}(|\sqrt{\rho}\theta_t|^2_2+c^2_0|\nabla \theta_t|^2_2),\\
I_2+I_5\leq& C|\rho|^{\frac{1}{2}}_{\infty} |\sqrt{\rho}\theta_t|^{\frac{1}{2}}_{2}|\sqrt{\rho}\theta_t|^{\frac{1}{2}}_{6}( | w_t|_{6} |\nabla \theta|_{2}+ |\nabla w_t|_{2} | \theta|_{6})\\
\leq& C c^{21}_3|\sqrt{\rho}\theta_t|^2_{2}+\frac{\kappa}{20c^2_0c_v}(|\sqrt{\rho}\theta_t|^2_2+c^2_0|\nabla \theta_t|^2_2)+C|\nabla w_t|^2_{2}, \\
I_3\leq& C|\rho|^{\frac{1}{2}}_{\infty} |w|_{\infty} |\nabla \theta_t|_{2} |\sqrt{\rho}\theta_t|_{2}\leq Cc^3_3|\sqrt{\rho}\theta_t|^2_{2}+\frac{\kappa}{20c_v}|\nabla \theta_t|^2_2,\\
I_4\leq& C|(P_m)_t|_{2} |\nabla w|_{3} |\theta_t|_{6}\leq C(|\rho|^{\frac{1}{2}}_{\infty}|\sqrt{\rho}\theta_t|_2 +|\rho_t|_3|\theta|_6)|\nabla w|_{3} |\theta_t|_{6}\\
\leq& Cc^5_3|\sqrt{\rho}\theta_t|^2_{2}+\frac{\kappa}{20c^2_0c_v}(|\sqrt{\rho}\theta_t|^2_2+c^2_0|\nabla \theta_t|^2_2)+Cc^{16}_3,\\
I_6\leq& C|\theta_t|_{6} |\nabla w_t|_{2} |\nabla w|_{3}\leq \frac{\kappa}{20c^2_0c_v}(|\sqrt{\rho}\theta_t|^2_2+c^2_0|\nabla \theta_t|^2_2)+Cc^4_3|\nabla w_t|^2_{2},
\end{split}
\end{equation*}

For the term $E_I$. From H\"older's inequality,  (\ref{zhen1})-(\ref{jia345}) and Young's inequality, we have
\begin{equation}\label{op11}
\begin{split}
J_1=&\frac{1}{c_v}\int_0^\infty \int_{S^2} \int_{\mathbb{V}} \Big(1-\frac{w\cdot \Omega}{c}\Big) S_t\theta_t \text{d}x \text{d}\Omega \text{d}v\\
\leq &C\big(1+|w|_{3}\big)|\theta_t|_{6}\int_0^\infty \int_{S^2} |S_t |_{2}\text{d}\Omega \text{d}v
\leq  \frac{\kappa}{20c^2_0c_v}(|\sqrt{\rho}\theta_t|^2_2+c^2_0|\nabla \theta_t|^2_2)+Cc^6_2,\\
J_2=&-\frac{1}{c_v}\int_0^\infty \int_{S^2} \int_{\mathbb{V}} \Big(1-\frac{w\cdot \Omega}{c}\Big) (\sigma_a I)_t\theta_t \text{d}x \text{d}\Omega \text{d}v\\
\leq &  C(1+|w|_{\infty})\int_0^\infty \int_{S^2} \big(|\rho|^{\frac{1}{2}}_\infty|\sqrt{\rho}\theta_t|_{2}|\sigma_t|_2|I|_\infty+|\sigma|_\infty|I|_3|\rho_t|_2|\theta_t|_6\big)\text{d}\Omega \text{d}v\\
&+ C(1+|w|_{\infty})|\rho|^{\frac{1}{2}}_\infty|\sqrt{\rho}\theta_t|_{2}\int_0^\infty \int_{S^2} |\sigma|_\infty|I_t|_{2}\text{d}\Omega\\
\leq &\frac{\kappa}{20c^2_0c_v}(|\sqrt{\rho}\theta_t|^2_2+c^2_0|\nabla \theta_t|^2_2)+C(\epsilon)c^{10}_3(|\sqrt{\rho}\theta_t|^2_{2}+1)+\epsilon \int_0^\infty \int_{S^2} |\sigma_t|^2_{2}\text{d}\Omega \text{d}v,
\end{split}
\end{equation}
where we  used $(\sigma_a)_t=\rho \sigma_t+\rho_t \sigma$ and $\epsilon$ is a sufficiently small constant.
Similarly, 
\begin{equation}\label{op2}
\begin{split}
J_3=&\frac{1}{c_v}\int_\mathbb{I}\int_{\mathbb{V}}\frac{v}{v'}\Big(1-\frac{w\cdot \Omega}{c}\Big)(\sigma_s I')_t
\theta_t\text{d}x \text{d}\mathbb{I}\\
 \leq &C(1+|w|_{\infty})\int_\mathbb{I}  \frac{v}{v'}\big(|\theta_t|_{6}
|(\sigma_s)_t|_2 |I'|_3+|\rho|^{\frac{1}{2}}_\infty|\sqrt{\rho}\theta_t|_{2}
|\overline{\sigma}_s|_\infty |I'_t|_2\big)\text{d}\mathbb{I}\\
\leq &C\alpha^2 c^{10}_3\|I'\|^2_{L^2(\mathbb{R}^+\times S^2; H^1)}
+C\alpha^2 c^3_3\|I'_t\|^2_{L^2(\mathbb{R}^+\times S^2; L^2)}\\
&+C|\sqrt{\rho}\theta_t|^2_{2}+\frac{\kappa}{20c^2_0c_v}(|\sqrt{\rho}\theta_t|^2_2+c^2_0|\nabla \theta_t|^2_2)\leq \frac{\kappa}{20c_v}|\nabla \theta_t|^2_2+C(|\sqrt{\rho}\theta_t|^2_{2}+c^{12}_3),\\
\end{split}
\end{equation}
\begin{equation}\label{op3}
\begin{split}
J_4=&\frac{1}{c_v}\int_\mathbb{I} \int_{\mathbb{V}}-\Big(1-\frac{w\cdot \Omega}{c}\Big)(\sigma'_s I)_t
\theta_t\text{d}x \text{d}\mathbb{I}\\
\leq &C(1+|w|_{\infty}) \int_\mathbb{I}\big(|\theta_t|_{6}
|(\sigma'_s)_t|_2|I|_3 +|\rho|^{\frac{1}{2}}_\infty|\sqrt{\rho}\theta_t|_{2}
|\overline{\sigma}'_s|_\infty |I_t|_2\big) \text{d}\mathbb{I}\\
\displaystyle
\leq &C\alpha^2 c^{10}_3\|I\|^2_{L^2(\mathbb{R}^+\times S^2; H^1)}
+C\alpha^2 c^3_3\|I_t\|^2_{L^2(\mathbb{R}^+\times S^2; L^2)}
\\
&+C|\sqrt{\rho}\theta_t|^2_{2}+\frac{\kappa}{20c^2_0c_v}(|\sqrt{\rho}\theta_t|^2_2+c^2_0|\nabla \theta_t|^2_2)
\leq  \frac{\kappa}{20c_v}|\theta_t|^2_{D^1}+C|\sqrt{\rho}\theta_t|^2_{2}+Cc^{12}_3,\\
J_5=&\frac{1}{c_v}\int_0^\infty \int_{S^2} \int_{\mathbb{V}} \frac{-w_t \cdot \Omega}{c} S\theta_t \text{d}x \text{d}\Omega \text{d}v\\
\leq&  C|w_t|_{3}|\theta_t|_{6}\int_0^\infty \int_{S^2} | S|_{2}\text{d}\Omega \text{d}v
 \leq \frac{\kappa}{20c^2_0c_v}(|\sqrt{\rho}\theta_t|^2_2+c^2_0|\nabla \theta_t|^2_2)+Cc^{4}_0|w_t|^2_{D^1},\\
J_6=&\frac{1}{c_v}\int_\mathbb{I} \int_{\mathbb{V}} \frac{w_t \cdot \Omega}{c}
\sigma'_s I\theta_t\text{d}x \text{d}\mathbb{I}
\leq C|w_t|_{6}|\rho|^{\frac{1}{2}}_\infty|\sqrt{\rho}\theta_t|_{2}
\int_\mathbb{I}
|\overline{\sigma}'_s|_\infty |I|_3\text{d}\mathbb{I}\\
\leq &C|\sqrt{\rho}\theta_t|^2_{2}+C\alpha^2 |w_t|^2_{D^1}|\rho|_\infty\|I\|^2_{L^2(\mathbb{R}^+\times S^2; H^1)}
\leq  C|\sqrt{\rho}\theta_t|^2_{2}+Cc^{3}_0|w_t|^2_{D^1},\\
J_7=&\frac{1}{c_v}\int_0^\infty \int_{S^2} \int_{\mathbb{V}} \frac{w_t \cdot \Omega}{c}  \sigma_a I\theta_t \text{d}x \text{d}\Omega \text{d}v\\
\leq& C|w_t|_{6}|\rho|^{\frac{1}{2}}_\infty|\sqrt{\rho}\theta_t|_{2}\int_0^\infty \int_{S^2} |\sigma|_\infty\|I\|_{3}\text{d}\Omega \text{d}v\\
\leq &C|\sqrt{\rho}\theta_t|^2_{2}+Cc_0|w_t|^2_{D^1}\|\sigma\|^2_{L^2(\mathbb{R}^+\times S^2; L^\infty)}\|I\|^2_{L^2(\mathbb{R}^+\times S^2; H^1)}\\
\leq& C|\sqrt{\rho}\theta_t|^2_{2}+Cc^{3}_0|w_t|^2_{D^1},\\
J_8=&\frac{1}{c_v}\int_\mathbb{I} \int_{\mathbb{V}} -\frac{v}{v'}\frac{w_t \cdot \Omega}{c}
\sigma_s I'\theta_t\text{d}x \text{d}\mathbb{I}
\leq C|w_t|_{6}|\rho|^{\frac{1}{2}}_\infty|\sqrt{\rho}\theta_t|_{2}
\int_\mathbb{I} \frac{v}{v'}
|\overline{\sigma}_s|_\infty |I'|_3\text{d}\mathbb{I}\\
\leq &C\alpha^2 |w_t|^2_{D^1}|\rho|_\infty\|I'\|^2_{L^2(\mathbb{R}^+\times S^2; H^1)}+C|\sqrt{\rho}\theta_t|^2_{2}
\leq  C|\sqrt{\rho}\theta_t|^2_{2}+Cc^{3}_0|w_t|^2_{D^1},
\end{split}
\end{equation}
Then combining the above estimates for $I_i$ and $J_j$,  from (\ref{zhou1}) we quickly have
\begin{equation}\label{thyter}
\begin{split}
&\frac{1}{2}\frac{d}{dt} \int_{\mathbb{V}}\rho |\theta_t|^2 \text{d}x+ \int_{\mathbb{V}}|\nabla \theta_t|^2 dx\\
\leq&  C(\epsilon)c^{21}_3|\sqrt{\rho}\theta_t|^2_2+\epsilon\int_0^\infty \int_{S^2} |\sigma_t|^2_{2}\text{d}\Omega \text{d}v+C(\epsilon)(c^{16}_3+c^4_3| w_t|^2_{D^1}).
\end{split}
\end{equation}
Notice that, via the assumption (\ref{jia345}) and Lemma \ref{pang}, we have
$$\epsilon \int_0^\infty \int_{S^2} |\sigma_t|^2_{2}\text{d}\Omega \text{d}v\leq \epsilon M(c_5)(|\rho_t|^2_2+|\theta_t|^2_2) \leq \epsilon M(c_5)(c^4_2+|\sqrt{\rho}\theta_t|^2_2+c^2_1|\nabla \theta_t|_2),$$
then integrating (\ref{thyter}) over $(\tau,t)$ with $\tau\in (0,t)$, letting $\epsilon$ be sufficiently small, we  have
 \begin{equation}\label{nv4}
\begin{split}
&|\sqrt{\rho}\theta_t(t)|^2_{2}+\int_{\tau}^{t}|\theta_t|^2_{D^1}\text{d}s
\leq |\sqrt{\rho}\theta_t(\tau)|^2_{2}+Cc^{21}_3\int_{\tau}^{t}|\sqrt{\rho}\theta_t|^2_2\text{d}s+Cc^{16}_3t+Cc^6_3.
\end{split}
\end{equation}
From $(\ref{eq:1.weeer})_3$, via  letting $
\Psi=\kappa\triangle \theta+Q( w)+ \overline{N}_r
$, we have
\begin{equation}\label{nv5}
\begin{split}
&|\sqrt{\rho}\theta_t|^2_{2}\leq |\rho|_{\infty} \|\nabla w \|^2_1 |\nabla \theta|^2_2+ \int_{\mathbb{V}}|\Psi|^2/ \rho \text{d}x,
\end{split}
\end{equation}

Via assumptions (\ref{zhen1})-(\ref{jia345}),  Lemma \ref{lem1},  the regularity of $S(v,t,x)$ and Minkowski's inequality,  we easily have
\begin{equation*}
\begin{split}
&\lim_{t \mapsto 0} \int_{\mathbb{V}}\Big(\frac{|\Psi(t)|^2}{\rho}-\frac{|\Psi(0)|^2}{\rho_0}\Big)\text{d}x\\
\leq&
\lim_{t \mapsto 0}\Big(\frac{1}{\underline{\delta}} \int_{\mathbb{V}}|\Psi(t)-\Psi(0)|^2\text{d}x+\frac{1}{\delta\underline{\delta}}|\rho(t)-\rho_0|_\infty \int_{\mathbb{V}}|\Psi(0)|^2\text{d}x\Big)
=0.
\end{split}
\end{equation*}
According to the compatibility condition (\ref{kkk}) and equation $(\ref{eq:1.weeer})_3$, we have
\begin{equation}\label{nv6}
\begin{split}
\limsup_{\tau \rightarrow 0}|\sqrt{\rho}\theta_t(\tau)|^2_{2}\leq |\rho_0|_{\infty} \|\nabla w_0 \|^2_1 |\nabla \theta_0|^2_2+|g_1|^2_2\leq Cc^5_0.
\end{split}
\end{equation}
Therefore, letting $\tau\rightarrow 0$ in (\ref{nv4}), we have
\begin{equation}\label{ghj}
\begin{split}
&|\sqrt{\rho}\theta_t(t)|^2_{2}+\int_{0}^{t}|\theta_t|^2_{D^1}\text{d}s
\leq Cc^{21}_3\int_{0}^{t}|\sqrt{\rho}\theta_t|^2_2\text{d}s+Cc^6_3
\end{split}
\end{equation}
for $0\leq t \leq T_3$.
Then according to Gronwall's inequality, we have
\begin{equation*}
\begin{split}
&|\sqrt{\rho}\theta_t(t)|^2_{2}+\int_{0}^{t}|\theta_t|^2_{D^1}\text{d}s
\leq Cc^6_3\exp\big(c^{21}_3 t\big)\leq Cc^6_3,
\end{split}
\end{equation*}
for  $0\leq t \leq T_4=\min(T^*,(M(c_5)c^{21}_5)^{-1})$.

 The further estimates can be obtained by Lemma \ref{tvd1}. From
\begin{equation}\label{tuo1}
\begin{split}
-\frac{\kappa}{c_v}\triangle \theta=-\rho \theta_t-\rho w\cdot \nabla \theta+\frac{1}{c_v}\Big(P_m\text{div} w+Q( w)+\overline{N}_r\Big),
\end{split}
\end{equation}
and Minkowski's inequality, we have
\begin{equation}\label{nv8}
\begin{split}
| \theta|_{D^2}\leq& C\big(|\rho \theta_t|_{2}+|\rho w\cdot \nabla \theta|_{2}+|\rho \theta\text{div} w|_{2}+|Q( w)|_{2}+ |\overline{N}_r|_{2})\leq Cc^6_3.
\end{split}
\end{equation}

Similarly,  for $0\leq t\leq T_4$,
we  also have
\begin{equation}\label{nv11}
\begin{split}
\int_0^t | \theta|^2_{D^{2,q}}\text{d}s\leq&C\int_0^t \Big(|\rho \theta_t|_{q}+|\rho w\cdot \nabla \theta|_{q}+|P_m\text{div}w|_{q}+|Q(w)|_{q}+|\overline{N}_r|_{q}\Big)^2\text{d}s
\leq Cc^{12}_3.
\end{split}
\end{equation}

According to $P_m=R\rho \theta$, for $0\leq t\leq T_4$, we easily obtain that
\begin{equation}\label{nv15}
\begin{split}
| \nabla P_m|_{2}\leq Cc^{\frac{9}{2}}_2c^{\frac{1}{2}}_3, \quad | \nabla P_m|_{q}\leq Cc^{7}_3, \quad |  (P_m)_t|_{2}\leq Cc^{8}_3.
\end{split}
\end{equation}
\end{proof}

Next we give the estimate for the velocity $u$.
\begin{lemma}\label{lem:4}
\begin{equation*}
\begin{split}
\|u(t)\|^2_{1}+|\sqrt{\rho}u_t(t)|^2_{2}+\int_{0}^{t}\big(|u|^2_{D^{2,q}}+|u_t|^2_{D^1}\big)\text{d}s\leq Cc^{7}_0, \ |u(t)|^2_{D^2}\leq Cc^{9}_2c_3,
\end{split}
\end{equation*}
for  $0\leq t \leq T_5=\min(T^*,(M(c_5)c^{25}_5)^{-1})$.
\end{lemma}
\begin{proof}
Differentiating $ (\ref{eq:1.weeer})_4$ with respect to $t$, we have
\begin{equation}\label{pengyue3}
\rho u_{tt}+Lu_t=-\rho_tu_t-(\rho w\cdot\nabla u)_t-(\nabla P_m)_t-\frac{1}{c}\int_0^\infty \int_{S^2}(\overline{C}_r)_t\Omega \text{d}\Omega \text{d}v,
\end{equation}
multiplying (\ref{pengyue3}) by $u_t$ and integrating over $\mathbb{V}$,  via 
 $\triangle u=\nabla\text{div}u-\nabla\times \text{curl}u$, 
 boundary condition (\ref{fan2}) and Lemma \ref{gag22},
we have
\begin{equation}\label{zhou2}
\begin{split}
&\frac{d}{dt} \int_{\mathbb{V}}\rho |u_t|^2 \text{d}x+ \int_{\mathbb{V}}|\nabla  u_t|^2 \text{d}x\\
\leq& C \int_{\mathbb{V}} \Big(|\rho_t w \cdot \nabla u \cdot u_t|+|\rho w_t \cdot \nabla u \cdot u_t|+|\rho w \cdot \nabla u_t \cdot u_t| 
\\
&+|(P_m)_t \text{div} u_t|\Big)\text{d}x-\frac{1}{c}\int_{\mathbb{V}} \Big(\int_0^\infty \int_{S^2}  (\overline{C}_r)_t u_t\cdot \Omega\text{d}\Omega \text{d}v\Big)\text{d}x\equiv:\sum_{i=7}^{10}I_i+E_{II},
\end{split}
\end{equation}
where the radiation source term:
\begin{equation*}
\begin{split}
E_{II}=&-\frac{1}{c}\int_{\mathbb{V}} \Big(\int_0^\infty \int_{S^2}  (\overline{C}_r)_t u_t\cdot \Omega\text{d}\Omega \text{d}v\Big)\text{d}x=\sum_{j=9}^{12}\frac{}{} J_j.
\end{split}
\end{equation*}

First, we estimate $\sum_{i=7}^{10}I_i$.  According to Lemma \ref{gaga1},  \ref{lem:2}-\ref{lem:3e}, H\"older's inequality, Gagliardo-Nirenberg inequality and Young's inequality, we have
\begin{equation*}
\begin{split}
I_7\leq& C|\rho_t|_{3} |w|_{\infty} |\nabla u|_{2} |  u_t|_{6}\leq C(\epsilon)c^6_3|\nabla u|^2_{2}+\epsilon|  \nabla u_t|^2_{2},\\
I_8\leq& C|\rho|^{\frac{1}{2}}_{\infty} | w_t|_{6} |\nabla u|_{3} |\sqrt{\rho}u_t|_{2}
\leq C\frac{1}{\eta} c_0\|\nabla u\|^2_{1}+C\eta|\nabla w_t|^2_{2}|\sqrt{\rho}u_t|^2_{2},\\
I_9\leq& C|\rho|^{\frac{1}{2}}_{\infty} |w|_{\infty} |\nabla u_t|_{2} |\sqrt{\rho}u_t|_{2}\leq C(\epsilon)c^3_3|\sqrt{\rho}u_t|^2_{2}+\epsilon|  \nabla u_t|^2_{2},\\
I_{10}\leq& C|(P_m)_t|_{2} |\nabla u_t|_{2}\leq C(\epsilon)c^{16}_3+\epsilon|  \nabla u_t|^2_{2},
\end{split}
\end{equation*}
where  $\epsilon$ and $\eta$ are both positive constants.
For radiation term $ E_{II}$, we have
\begin{equation*}
\begin{split}
J_9= &-\frac{1}{c} \int_0^\infty \int_{S^2} \int_{\mathbb{V}}  S_t u_t\cdot \Omega \text{d}x \text{d}\Omega \text{d}v
\leq C|u_t|_{2}\int_0^\infty \int_{S^2} |S_t |_{2}\text{d}\Omega \text{d}v
\leq \epsilon|  \nabla u_t|^2_{2}+C(\epsilon)c^2_0,\\
J_{10}=&\frac{1}{c} \int_0^\infty \int_{S^2} \int_{\mathbb{V}}  (\sigma_a I)_t u_t\cdot \Omega \text{d}x \text{d}\Omega \text{d}v\\
\leq& C\int_0^\infty \int_{S^2} \Big(|\sqrt{\rho}u_t |_{2}|\sigma_t|_2|I|_\infty|\rho|^{\frac{1}{2}}_\infty+|\rho_t|_2|\sigma|_\infty|I|_3|u_t|_6\Big)\text{d}\Omega \text{d}v\\
&+C|\rho|^{\frac{1}{2}}_\infty|\sqrt{\rho}u_t|_{2}\int_0^\infty \int_{S^2} |\sigma |_{\infty}|I_t|_2\text{d}\Omega \text{d}v\\
 \leq& \epsilon|  \nabla u_t|^2_{2} +C\Big(\frac{1}{\eta}c^3_3+1\Big)|\sqrt{\rho}u_t|^2_{2}+\eta \int_0^\infty \int_{S^2} |\sigma_t|^2_{2}\text{d}\Omega \text{d}v+C(\epsilon)c^{6}_3,
\end{split}
\end{equation*}
where we   used   the fact $(\sigma_a)_t=\rho \sigma_t+\rho_t \sigma$. Similarly
\begin{equation*}
\begin{split}
J_{11}= &-\frac{1}{c}\int_\mathbb{I} \int_{\mathbb{V}}\frac{v}{v'}(\sigma_s I')_t
u_t\cdot \Omega\text{d}x \text{d}\mathbb{I}\\
\leq &C|u_t|_{6} \int_\mathbb{I}
|(\sigma_s)_t|_2 |I'|_3\text{d}\mathbb{I}+C|\rho|^{\frac{1}{2}}_\infty|\sqrt{\rho}u_t|_{2}\int_\mathbb{I} \frac{v}{v'}
|\overline{\sigma}_s|_\infty |I'_t|_2\text{d}\mathbb{I}\\
\leq &C(\epsilon)\alpha^2 c^4_3\|I'\|^2_{L^2(\mathbb{R}^+\times S^2; H^1)} +C(\epsilon)\alpha^2 c_0\|I'_t\|^2_{L^2(\mathbb{R}^+\times S^2; L^2)}+\epsilon|u_t|^2_{D^1}+C|\sqrt{\rho}u_t|^2_{2}\\
\leq&\epsilon |u_t|^2_{D^1}+C|\sqrt{\rho}u_t|^2_{2}+C(\epsilon)c^{6}_3,\\
\end{split}
\end{equation*}
\begin{equation*}
\begin{split}
J_{12}= &\frac{1}{c}\int_\mathbb{I} \int_{\mathbb{V}}(\sigma'_s I)_t
u_t \cdot \Omega\text{d}x \text{d}\mathbb{I}
\leq C|u_t|_{6} \int_\mathbb{I}
|(\sigma'_s)_t|_2 |I|_3\text{d}\mathbb{I}+C|\rho|^{\frac{1}{2}}_\infty|\sqrt{\rho}u_t|_{2}\int_\mathbb{I}
|\overline{\sigma}'_s|_\infty |I_t|_2\text{d}\mathbb{I}\\
\leq &C(\epsilon)\alpha c^4_3\|I\|^2_{L^2(\mathbb{R}^+\times S^2; H^1)}
+C\alpha c_0\|I_t\|^2_{L^2(\mathbb{R}^+\times S^2; L^2)}+\epsilon|u_t|^2_{D^1}+C|\sqrt{\rho}u_t|^2_{2}\\
\leq& \epsilon|u_t|^2_{D^1}+C|\sqrt{\rho}u_t|^2_{2}+C(\epsilon)c^{6}_3.
\end{split}
\end{equation*}
Then combining  estimates for $I_i$ and $J_j$, letting $\epsilon$ be sufficiently small,  from (\ref{zhou2}) we  have
\begin{equation*}
\begin{split}
&\frac{1}{2}\frac{d}{dt} \int_{\mathbb{V}}\rho |u_t|^2 \text{d}x+\frac{1}{2} \int_{\mathbb{V}}|\nabla u_t|^2\text{d}x\\
\leq& Cc^{16}_3+C\Big(c^3_3\frac{1}{\eta}+\eta|\nabla w_t|^2_{2}+1\Big)|\sqrt{\rho}u_t|^2_{2}+\eta \int_0^\infty \int_{S^2} |\sigma_t|^2_{2}\text{d}\Omega \text{d}v+C\frac{1}{\eta} c_0\|\nabla u\|^2_{1}.
\end{split}
\end{equation*}
Similarly to prove (\ref{ghj}), via the compatibility condition (\ref{kkk}),  we have
\begin{equation}\label{nv7hj}
\begin{split}
&|\nabla u(t)|^2+|\sqrt{\rho}u_t (t)|^2_{2}+\frac{1}{2}\int_{0}^{t}| u_t|^2_{D^1}\text{d}s\\
\leq & C\frac{1}{\eta} c_0\int_{0}^{t}| u|^2_{D^2}\text{d}s+C\int_{0}^{t}\Big(c^3_3\frac{1}{\eta}+\eta|\nabla w_t|^2_{2}+1\Big)(|\sqrt{\rho}u_t|^2_{2}+|\nabla u|^2_2))\text{d}s\\
&+M(c_5)\eta \int_0^t (|\theta_t|^2_2+|\rho_t|^2_2)\text{d}s +Cc^{16}_3t+Cc^5_0.
\end{split}
\end{equation}
For $| u|_{D^2}$, due to Lemma \ref{tvd1} and  Minkowski's inequality, we have
\begin{equation}\label{nv980}
\begin{split}
|u|_{D^2}\leq & C\Big(|\rho u_t|_2+|\rho w\cdot\nabla u|_2+|\nabla P_m|_2+\int_0^\infty \int_{S^2}|\overline{C}_r|_2 \text{d}\Omega \text{d}v\Big)\\
\leq & C(|\rho|^{\frac{1}{2}}_\infty|\sqrt{\rho}u_t|_{2}+|\rho|_\infty|w|_6|\nabla u|_{3}+|\rho|_\infty|\nabla \theta|_2+Cc^{2}_0)\\
\leq & Cc^4_2(|\sqrt{\rho}u_t|_{2}+|\nabla u|_2)+\frac{1}{2}|u|_{D^2}+Cc^{\frac{9}{2}}_2c^{\frac{1}{2}}_3.
\end{split}
\end{equation}
According to Lemma \ref{pang}, we also have
\begin{equation}\label{zao}
|\theta_t|_2\leq C|\theta_t|_6\leq C(|\sqrt{\rho}\theta_t|_2+c_0|\nabla \theta_t|_2)\leq C(c^6_3+c_0|\nabla \theta_t|_2).
\end{equation}

Therefore, from (\ref{nv7hj})-(\ref{zao}), we have
\begin{equation}\label{nv7mm}
\begin{split}
&|\nabla u(t)|^2+|\sqrt{\rho}u_t(t)|^2_{2}+\frac{1}{2}\int_{0}^{t}| u_t|^2_{D^1}\text{d}s\leq Cc^{16}_3t+Cc^5_0+C\frac{1}{\eta}c^{11}_3t\\
&  +C\int_{0}^{t}\Big(c^9_3\frac{1}{\eta}+\eta|\nabla w_t|^2_{2}+1\Big)(|\sqrt{\rho}u_t|^2_{2}+|\nabla u|^2_2))\text{d}s+\eta M(c_5)(c^{12}_3t+c^{14}_3)
\end{split}
\end{equation}
for $0\leq t \leq T_4$.
From Gronwall's inequality and (\ref{nv7mm}), we have
\begin{equation}\label{nv98nn}
\begin{split}
&| u(t)|^2_{D^1}+|\sqrt{\rho}u_t (t)|^2_{2}+\mu\int_{0}^{t}| u_t|^2_{D^1}\text{d}s \\
\leq &C\Big(\eta M(c_5)(c^{12}_3t+c^{14}_3)+c^{16}_3t+c^5_0+\frac{1}{\eta}c^{11}_3t\Big)\exp\Big(\int_{0}^{t}\Big(c^5_3\frac{1}{\eta}+\eta|\nabla w_t|^2_{2}+1\Big)\text{d}s\Big)\\
\leq &C\Big(\eta M(c_5)(c^{12}_3t+c^{14}_3)+c^{16}_3t+c^5_0+\frac{1}{\eta}c^{11}_3t+t\Big)\exp\Big( c^5_3t\frac{1}{\eta}+\eta c^2_{2}+t\Big)\leq Cc^5_0
\end{split}
\end{equation}
for $0\leq t \leq T_5=\min(T^*,(M(c_5)c^{25}_5)^{-1})$ and   $\eta=\frac{1}{c^{14}_3}$. Then from  (\ref{nv980}), we have
$$
|u|_{D^2}\leq  Cc^4_2(|\sqrt{\rho}u_t|_{2}+|\nabla u|_2)+Cc^{\frac{9}{2}}_2c^{\frac{1}{2}}_3\leq Cc^{6}_2c^{\frac{1}{2}}_3.
$$

 Then similarly, we have
\begin{equation*}
\begin{split}
\int_{0}^{t}|u|^2_{D^{2,q}}\text{d}s
\leq& \int_{0}^{t}\Big( |\rho u_t+\rho w\cdot\nabla u+\nabla P_m|^2_q+\Big(\int_0^\infty \int_{S^2}|\overline{C}_r|_q\text{d}\Omega \text{d}v\Big)^2 \Big)\text{d}s\leq Cc^{7}_0.
\end{split}
\end{equation*}
\end{proof}

Based on the a priori estimates of Lemmas \ref{lem:2}-\ref{lem:4}, we conclude that
\begin{equation}\label{jkkll}
\begin{split}
\|\rho(t)\|_{ W^{1,q}}\leq Cc_0, \quad  |\rho_t(t)|_{ q}\leq& Cc_0c_3,\\
\|I\|^2_{L^2(\mathbb{R}^+\times S^2;C([0,T_2]; W^{1,q}))}\leq Cc^2_0 ,\quad  \|I_t\|^2_{L^2(\mathbb{R}^+\times S^2;C([0,T_2]; L^{q}))}\leq& Cc_0c_1,\\
\|\theta(t)\|^2_{2}+|\sqrt{\rho}\theta_t(t)|^2_{2}+\int_{0}^{t}\big(|\theta(s)|^2_{D^{2,q}}+|\theta_t(s)|^2_{D^1}\big)\text{d}s\leq & Cc^{12}_3,\\
\|u(t)\|^2_{1}+|\sqrt{\rho}u_t(t)|^2_{2}+\int_{0}^{t}\Big(|u(s)|^2_{D^{2,q}}+|u_t(s)|^2_{D^1}\Big)\text{d}s\leq Cc^{7}_0,\quad |u(t)|^2_{D^2}\leq & Cc^{12}_2c_3
\end{split}
\end{equation}
for $0 \leq t \leq \min(T^*,(M(c_5)c^{25}_5)^{-1})$. Therefore, if we define the constants $c_i$ ($i=1,2,3,4,5$) and $T^*$ by
\begin{equation}\label{dingyi}
\begin{split}
&c_1=C^{\frac{1}{2}}c_0, \quad  c_2=C^{\frac{1}{2}}c^{\frac{7}{2}}_0,\quad c_3=Cc^{12}_2=C^{7}c^{42}_0,\\
& c_4= c_5=C^{\frac{1}{2}}c^{6}_3=C^{\frac{85}{2}}c^{252}_0, \quad \text{and} \quad T^*=\min (T, (M(c_5)c^{25}_5)^{-1}),
\end{split}
\end{equation}
then we deduce that
\begin{equation}\label{pri}
\begin{split}
\|I\|^2_{L^2(\mathbb{R}^+\times S^2;C([0,T^*]; W^{1,q}))}\leq c^2_1,\quad \|I_t\|^2_{L^2(\mathbb{R}^+\times S^2;C([0,T^*]; L^{q}))}\leq& c^2_2,\\
\sup_{0\leq t\leq T^*}\|u(t)\|^2_{1}+\int_{0}^{T^*}\Big(|u|^2_{D^{2,q}}+|u|^2_{D^{2}}+|u_t|^2_{D^1}\Big)\text{d}t\leq c^2_2,\quad \sup_{0\leq t\leq T^*}|u(t)|^2_{D^2}\leq& c^2_3,\\
 \sup_{0\leq t\leq T^*}\|\theta(t)\|^2_{1}+\int_{0}^{T^*}\Big(|\theta|^2_{D^{2,q}}+|\theta|^2_{D^{2}}+|\theta_t|^2_{D^1}\Big)\text{d}t\leq& c^2_4,\\
 \sup_{0\leq t\leq T^*}|\theta(t)|^2_{D^2}\leq c^2_5,\quad  \sup_{0\leq t\leq T^*}(\|\rho(t)\|_{ W^{1,q}}+ |\rho_t(t)|_{ q})\leq& c^2_3.
\end{split}
\end{equation}

\subsection{The unique solvability of  the linearized problem with vacuum}Now we give the key lemma to prove our main result.

\begin{lemma}\label{lemk1}
Let (\ref{tyuu}) and (\ref{kkk})-(\ref{gogogo}) hold, 
then $\exists$ unique strong solution $(I,\rho,u,\theta)$ on $\mathbb{R}^+\times S^2\times [0,T^*] \times \mathbb{V}$ to  (\ref{eq:1.weeer})-(\ref{py4}).
Moreover, $(I,\rho,u,\theta)$ satisfies   estimates (\ref{pri}).
\end{lemma}

\begin{proof}\underline{Step 1}: Existence of strong solution.
We define $\rho_0=\rho_0+\delta$ for each $\delta\in (0,1)$. Then from the compatibility conditions (\ref{kkk}), we have
\begin{equation*}
\begin{split}
&Lu_0+R\nabla (\rho^\delta_0 \theta_0)+\frac{1}{c}\int_0^\infty \int_{S^2}A^{0,\delta}_r\Omega \text{d}\Omega \text{d}v=(\rho^\delta_0)^{\frac{1}{2}}_0 g^\delta_1,\\
&-\frac{1}{c_v}(\kappa\triangle \theta_0+Q( u_0))-\int_0^\infty \int_{S^{2}}\frac{1}{c_v} \Big(1-\frac{u_0\cdot \Omega}{c}\Big)A^{0,\delta}_r \text{d}\Omega \text{d}v=(\rho^\delta_0)^{\frac{1}{2}} g^\delta_2,
\end{split}
\end{equation*}
where
\begin{equation*}\begin{split}
g^\delta_1=&\Big(\frac{\rho_0}{\rho^\delta_0}\Big)^{\frac{1}{2}}g_1+R\delta\frac{\nabla \theta_0}{(\rho^\delta_0)^{\frac{1}{2}}}-\frac{1}{c}\int_0^\infty \int_{S^{2}}\frac{(A^{0}_r-A^{0,\delta}_r)}{(\rho^\delta_0)^{\frac{1}{2}}}\Omega\text{d}\Omega \text{d}v,\\
g^\delta_2=&\Big(\frac{\rho_0}{\rho^\delta_0}\Big)^{\frac{1}{2}}g_2+\int_0^\infty \int_{S^{2}}\frac{1}{c_v} \Big(1-\frac{u_0\cdot \Omega}{c}\Big)\frac{(A^{0}_r-A^{0,\delta}_r)}{(\rho^\delta_0)^{\frac{1}{2}}}\text{d}\Omega \text{d}v.
\end{split}
\end{equation*}
Then according to the assumption (\ref{jia345}), for all  $\delta> 0$ small enough,
\begin{equation*}\begin{split}
&1+\|\rho^\delta_0\|_{ W^{1,q}}+\|(\theta_0,u_0)\|_{2}+|(g^\delta_1,g^\delta_2)|_2+\|I_0\|_{L^2(\mathbb{R}^+\times S^2;  W^{1,q})}\\
& +\|S(v,t,x)\|_{ L^2(\mathbb{R}^+;C^1([0,T]; W^{1,q}))\cap C^1([0,T]; L^1(\mathbb{R}^+; L^2))}\leq c_0.
\end{split}
\end{equation*}
Therefore, corresponding to the initial data $(I_0,\rho^\delta_0,\theta_0,u_0)$, there exists a unique strong solution $(I^\delta,\rho^\delta,u^\delta,\theta^\delta)$ satisfying  (\ref{pri}). Then there exists a subsequence of solutions $(I^\delta,\rho^\delta,u^\delta,\theta^\delta)$ converges to a limit $(I,\rho,u,\theta)$ in weak or weak* sense. Due to the compact property in \cite{jm}, there exists a subsequence of solutions $(I^\delta,\rho^\delta,u^\delta,\theta^\delta)$ satisfying:
\begin{equation}\label{ert}\begin{split}
& I^\delta\rightarrow  I\ \text{weakly}\  \text{in}\  L^2(\mathbb{R}^+\times S^2\times [0,T^*]\times \mathbb{V}),\\
&\rho^\delta\rightarrow \rho \  \text{in }\ C([0,T^*];L^2(K)),\ (u^\delta,\theta^\delta)\rightarrow (u,\theta)\ \text{in } \ C([0,T^*];H^1(K)),
\end{split}
\end{equation}
where $K$ is any compact subset of $\mathbb{V}$. 
Combining the lower semi-continuity of norms and (\ref{ert}), we know that $(I,\rho,u,\theta)$ also satisfies the local estimates (\ref{pri}).
So it is easy to show that $(I,\rho,u,\theta)$ is a weak solution in the sense of distribution  and  satisfies
\begin{equation}\label{yuo}
\begin{split}
&I\in L^2(\mathbb{R}^+\times S^2;L^\infty([0,T^*]; W^{1,q})),\ I_t\in L^2(\mathbb{R}^+\times S^2;L^\infty([0,T^*]; L^{q})),\\
& \rho\in L^\infty([0,T^*]; W^{1,q})),\quad \rho_t\in L^2\in L^\infty([0,T^*]; L^{q}),\\
&(\theta,u)\in L^\infty([0,T^*];H^2)\cap  L^2([0,T^*];D^{2,q}),\\
&(\theta_{t},u_{t})\in  L^2([0,T^*];H^1),\ (\sqrt{\rho}\theta_{t},\sqrt{\rho}u_{t})\in  L^\infty([0,T^*];L^2).
\end{split}
\end{equation}

\underline{Step 2}: Uniqueness  can be obtained by the same method used in  Lemma \ref{lem1}.

\underline{Step 3}: Time-continuity. The continuity of $\rho$ can be obtained by the same argument as in Lemma \ref{lem1}. For $I$,
due to (\ref{yuo}), for $\forall$ $(v,\Omega)\in R^+\times S^2$, we have
$$
I(v,\Omega,\cdot,\cdot)\in C([0,T^*];  L^q) \cap C([0,T^*]; W^{1,q}-\text{weak}).
$$
According to ({\ref{fude}}), we have
\begin{equation}\label{thymm}
\begin{split}
\limsup_{t\rightarrow 0}\|I(v,\Omega,\cdot,\cdot)\|^2_{W^{1,q}}\leq \|I_0\|^2_{W^{1,q}},
\end{split}
\end{equation}
which implies that $I(v,\Omega,t,x)$ is right-continuous at $t=0$ (see \cite{teman}). So we easily get the desired  conclusion for $I$ from the reversibility on the time to equation $(\ref{eq:1.weeer})_1$. Similarly, from (\ref{yuo}), we have
$$
(u,\theta)\in C([0,T^*]; H^1) \cap C([0,T^*]; D^2-\text{weak}).
$$
From equations (\ref{eq:1.weeer}) and (\ref{yuo}), we know that
$$
(\rho \theta_t, \rho u_t)\in L^2([0,T^*]; L^2), \ \text{and}\ ((\rho \theta_t)_t, (\rho u_t)_t)\in L^2([0,T^*]; H^{-1}),
$$
via Aubin-Lions lemma, we have  $(\rho \theta_t, \rho u_t)\in C([0,T^*]; L^2)$. 
Due to (\ref{tuo1}), 
\begin{equation}\label{thghj}
Lu=-\rho u_t-\rho (w\cdot \nabla) u
  -\nabla P_m -\frac{1}{c}\int_0^\infty \int_{S^2}\overline{C}_r\Omega \text{d}\Omega \text{d}v,
\end{equation}
and the elliptic regularity estimate in Lemma \ref{tvd1}, we have $
(u,\theta)\in C([0,T^*]; D^2)
$.
\end{proof}
\subsection{Proof of Theorem \ref{th1}} Our proof  is based on the classical iteration scheme and the existence results for the linearized problem with vacuum in  Section 4.2. Let us denote as in Section 4.1 that
\begin{equation*}\begin{split}
&2+\|I_0\|_{L^2(\mathbb{R}^+\times S^2;  W^{1,q})}+\|\rho_0\|_{ W^{1,q}}+\|(\theta_0,u_0)\|_{2}+|(g_1,g_2)|_2\\
& +\|S(v,t,x)\|_{ L^2(\mathbb{R}^+;C^1([0,T]; W^{1,q}))\cap C^1([0,T]; L^1(\mathbb{R}^+; L^2))}\leq c_0.
\end{split}
\end{equation*}

Let $u^0\in C([0,T];H^2)\cap  L^2([0,T];D^{2,q}) $, $\theta^0\in C([0,T];H^2)\cap  L^2([0,T];D^{2,q}) $ and $I^0\in L^2(\mathbb{R}^+\times S^2;C([0,T]; W^{1,q}))$ be the solutions to the following linear  problems
\begin{equation*}\begin{split}
&u^0_t-\triangle u^0=0 ;  \ u^0(0)=u_0 \ \text{in} \ \mathbb{V};\ (u\cdot n, (\nabla \times u) \cdot n)|_{\partial \mathbb{V}}=(0,0),\\
&\theta^0_t-\triangle \theta^0=0 ;\ \theta^0(0)=\theta_0 \ \text{in} \ \mathbb{V};\ \nabla \theta \cdot n |_{\partial \mathbb{V}}=0.\\[4pt]
&I^0_t+c\Omega\cdot\nabla I^0=0;\ I^0(0)=I_0 \ \text{in} \ \mathbb{R}^+\times S^2\times \mathbb{V};\ I^0|_{\partial \mathbb{V}}=0,\ \Omega \cdot n\leq 0.
\end{split}
\end{equation*}
Then we can choose a  time $T^{**}\in (0,T^*)$ such that (\ref{mini}) still  holds with $(T^*, \psi, w,\phi)$ replaced by $(T^{**},I^0,u^0,\theta^0)$.

\begin{proof}
\underline{Step 1}. Existence.
Let $(w,\phi,\psi)=(u^0,\theta^0,I'^0)$,  we can get $(I^1, \rho^1, u^1,\theta^1)$ as a strong solution to  (\ref{eq:1.weeer})-(\ref{py4}). Then we construct approximate solutions $(I^{k+1}, \rho^{k+1}, u^{k+1},\theta^{k+1})$ inductively as follows: assuming  $(I^{k}, u^{k},\theta^{k})$ was defined for $k\geq 1$, let $(I^{k+1}, \rho^{k+1}, u^{k+1},\theta^{k+1})$  be the  solution to (\ref{eq:1.weeer})-(\ref{py4}) with $(\psi,w,\phi)$ replaced by $(I'^{k}, u^{k},\theta^{k})$ as following:
\begin{equation}
\label{eq:1.wgo}
\begin{cases}
\displaystyle
\rho^{k+1}_t+\text{div}(\rho^{k+1} u^k)=0,\\[4pt]
\displaystyle
\frac{1}{c}I^{k+1}_t+\Omega\cdot\nabla I^{k+1}=\overline{A}^{k}_r,\\[4pt]
\displaystyle
\rho^{k+1} \theta^{k+1}_t+
\rho^{k+1} u^{k}\cdot \nabla \theta^{k+1} +\frac{1}{c_v}(P^{k+1}_m\text{div} u^{k}-k\triangle \theta^{k+1})
=\frac{1}{c_v}(Q(\nabla u^{k})+\overline{N}^{k}_r),\\[6pt]
\displaystyle
\rho^{k+1} u^{k+1}_t+\rho^{k+1} u^{k}\cdot \nabla  u^{k+1}
  +\nabla P^{k+1}_m +Lu^{k+1}=-\frac{1}{c}\int_0^\infty \int_{S^2}\overline{C}^{k}_r\Omega \text{d}\Omega \text{d}v,
\end{cases}
\end{equation}
where
\begin{equation*}
\begin{split}
&\overline{A}^{k}_r=S-\sigma^{k+1,k}_aI^{k+1}+\int_0^\infty \int_{S^2}\Big( \frac{v}{v'}\sigma^{k+1}_sI'^{k} -(\sigma'_s)^{k+1}I^{k+1} \Big)\text{d}\Omega' \text{d}v',\\
&\overline{B}^{k}_r=S-\sigma^{k+1,k}_aI^{k+1}+\int_0^\infty \int_{S^2}\Big( \frac{v}{v'}\sigma^{k+1}_sI'^{k+1} -(\sigma'_s)^{k+1}I^{k+1} \Big)\text{d}\Omega' \text{d}v',\\
&\overline{C}^{k}_r=S-\sigma^{k+1}_aI^{k+1}+\int_0^\infty \int_{S^2}\Big( \frac{v}{v'}\sigma^{k+1}_sI'^{k+1} -(\sigma'_s)^{k+1}I^{k+1} \Big)\text{d}\Omega' \text{d}v',\\
&P^{k+1}_m=R\rho^{k+1}\theta^{k+1},\ \overline{N}^k_r=\int_0^\infty \int_{S^2} \Big(1-\frac{u^k\cdot \Omega}{c}\Big) \overline{B}^k_r \text{d}\Omega \text{d}v,\\
& \sigma^{k+1,k}_a=\sigma(v,\Omega,t,x,\rho^{k+1},\theta^k)\rho^{k+1}=\sigma^{k+1,k}\rho^{k+1},\ \sigma^{k+1}_s=\overline{\sigma}_s\rho^{k+1},\\
&\sigma^{k+1}_a=\sigma(v,\Omega,t,x,\rho^{k+1},\theta^{k+1})\rho^{k+1}=\sigma^{k+1}\rho^{k+1},  \ (\sigma'_s)^{k+1}=\overline{\sigma}'_s\rho^{k+1},
\end{split}
\end{equation*}
and 
$
(I^{k+1}, \rho^{k+1}, u^{k+1},\theta^{k+1})|_{t=0}=(I_0, \rho_0, u_0,\theta_0)
$. Via Section $4.2$,  we know $(I^k,\rho^k,u^k,\theta^k)$ also satisfy the estimate (\ref{pri}).

Next, we show that $(I^k, \rho^k, u^k,\theta^k)$ converges to a limit in a strong sense which will be used to prove the existence of the strong solution. Letting
$$
\overline{I}^{k+1}=I^{k+1}-I^k,\ \overline{\rho}^{k+1}=\rho^{k+1}-\rho^k,\ \overline{u}^{k+1}=u^{k+1}-u^k,\ \overline{\theta}^{k+1}=\theta^{k+1}-\theta^k,
$$
we have
\begin{equation}
\label{eq:1.2w}
\begin{split}
\displaystyle
&\overline{\rho}^{k+1}_t+\text{div}(\overline{\rho}^{k+1} u^k)+\text{div}(\rho^{k} \overline{u}^k)=0,\\[6pt]
\displaystyle
&\frac{1}{c}\overline{I}^{k+1}_t+\Omega\cdot\nabla \overline{I}^{k+1}+\Big(\sigma^{k+1,k}_a+\int_0^\infty \int_{S^2}(\sigma'_s)^{k+1}\text{d}\Omega' \text{d}v'\Big)\overline{I}^{k+1}
=-I^k(\sigma^{k+1,k}_a-\sigma^{k,k-1}_a)\\
&-\int_0^\infty \int_{S^2}
\Big(    \big((\sigma'_s)^{k+1}-(\sigma'_s)^{k}\big)I^{k}
- \Big(\frac{v}{v'}
\big(\sigma^{k}_s\overline{I}'^{k}+I'^k(\sigma^{k+1}_s-\sigma^{k}_s)\big)\Big) \Big)\text{d}\Omega' \text{d}v',\\
&\rho^{k+1}\overline{\theta}^{k+1}_t+
\rho^{k+1}u^k\cdot \nabla\overline{\theta}^{k+1}-\frac{\kappa}{c_v}\triangle \overline{\theta}^{k+1}\\
=&\frac{1}{c_v}(Q(\nabla u^k)-Q(\nabla u^{k-1}))-\overline{\rho}^{k+1}\theta^k_t-\overline{\rho}^{k+1}(u^{k-1}\cdot \nabla \theta^k+\frac{1}{c_v}R\theta^k \text{div}u^{k-1})\\
&-\rho^{k+1}(\overline{u}^k\cdot \nabla \theta^k+\frac{1}{c_v}R\overline{\theta}^{k+1}\text{div}u^{k}+\frac{1}{c_v}R\theta^k \text{div}\overline{u}^{k})+\frac{1}{c_v}L_1,\\
&\rho^{k+1} \overline{u}^{k+1}_t+\rho^{k+1} u^k\cdot\nabla \overline{u}^{k+1}+L\overline{u}^{k+1}\\
=&\overline{\rho}^{k+1}(-u^k_t-u^{k-1}\cdot\nabla u^k)
-\rho^{k+1}\overline{u}^{k}\cdot\nabla u^k-R\nabla (\rho^{k+1}\overline{\theta}^{k+1}+\overline{\rho}^{k+1}\theta^{k})-\frac{1}{c}L_2,
\end{split}
\end{equation}
where $L_1$ and  $L_2$ are given via
\begin{equation*}\begin{split}
L_1&=\int_0^\infty \int_{S^2}\Big(1-\frac{u^k\cdot \Omega}{c}\Big)\Big(-\sigma^{k+1,k}_a\overline{I}^{k+1}-I^k(\sigma^{k+1,k}_a-\sigma^{k,k-1}_a)\Big) \text{d}\Omega \text{d}v\\
&+\int_\mathbb{I}\frac{v}{v'}\Big(1-\frac{u^k\cdot \Omega}{c}\Big)\Big(I'^k(\sigma^{k+1}_s-\sigma^{k}_s)+\sigma^{k+1}_s\overline{I}'^{k+1}\Big)\text{d}\mathbb{I}\\
&-\int_\mathbb{I}\Big(1-\frac{u^k\cdot \Omega}{c}\Big)\Big(I^{k}((\sigma'_s)^{k+1}-(\sigma'_s)^{k})+(\sigma'_s)^{k+1}\overline{I}^{k+1}\Big)\text{d}\mathbb{I}\\
&-\int_0^\infty \int_{S^2}\Big(\frac{\overline{u}^k\cdot \Omega}{c}\Big)\Big(S-\sigma^{k,k-1}_aI^{k}
+\int_0^\infty \int_{S^2}\Big(\frac{v}{v'}\sigma^{k}_sI'^{k-1}-(\sigma'_s)^kI^k\Big)\text{d}\Omega' \text{d}v'\Big) \text{d}\Omega \text{d}v,\\
L_2&=\int_0^\infty \int_{S^2}\Omega\Big(-\sigma^{k+1}_a\overline{I}^{k+1}-I^k(\sigma^{k+1}_a-\sigma^{k}_a)\Big) \text{d}\Omega \text{d}v\\
&+\int_\mathbb{I}\frac{v}{v'}\Big(\sigma^{k+1}_s\overline{I}'^{k+1}+I'^k(\sigma^{k+1}_s-\sigma^{k}_s)\Big)\text{d}\mathbb{I}-\int_\mathbb{I}\Big(I^{k}\big((\sigma'_s)^{k+1}-(\sigma'_s)^{k}\big)+(\sigma'_s)^{k+1}\overline{I}^{k+1}\Big)\text{d}\mathbb{I}.
\end{split}
\end{equation*}

First, we consider the mass density $\rho$. Multiplying $ (\ref{eq:1.2w})_1$ by $\overline{\rho}^{k+1}$ and integrating over $\mathbb{V}$, from (\ref{pri}) we have ($0<\eta \leq \frac{1}{10}$ is a constant which will be determined later)
\begin{equation}\label{go64}\begin{cases}
\displaystyle
\frac{d}{dt}|\overline{\rho}^{k+1}|^2_2\leq A^k_\eta(t)|\overline{\rho}^{k+1}|^2_2+\eta |\nabla\overline{u}^k|^2_2,\\[8pt]
\displaystyle
A^k_\eta(t)=C\Big(|\nabla u^k|^2_{W^{1,q}}+1/\eta|\nabla\rho^{k}|^2_{3} +1/\eta|\rho^{k}|^2_{\infty}\Big),\ \text{and} \ \int_0^t A^k_\eta(s)\text{d}s\leq C+C_{\eta}t
\end{cases}
\end{equation}
for $t\in [0,T^{**}]$, where $C_{\eta}>0$ is a  constant only  depending on $1/\eta$ and   parameters of  $C$.

Next, we consider the specific radiation intensity $I$. Multiplying $ (\ref{eq:1.2w})_2 $  by $\overline{I}^{k+1}$ and integrating over $\mathbb{R}^+\times S^2\times\mathbb{V}$, from  assumptions (\ref{zhen1})-(\ref{jia345}) and Lemma \ref{pang}, we have
\begin{equation}\label{go66}\begin{split}
&\frac{d}{dt}\|\overline{I}^{k+1}\|^2_{L^2(\mathbb{R}^+\times S^2 ;L^2(\mathbb{V}))}\\
\leq&  \int_0^\infty \int_{S^2} \big(|\sigma^{k,k-1}|_{\infty}+|\overline{\sigma}|_{\infty}| \rho^{k+1}|_\infty\big)\|I^k\|_{W^{1,q}}| \overline{\rho}^{k+1}|_2|\overline{I}^{k+1}|_2\text{d}\Omega \text{d}v\\
& +\int_0^\infty \int_{S^2} \big( |\overline{\sigma}|_{\infty}| \rho^{k+1}|_\infty\|I^k\|_{ W^{1,q}}| \overline{\theta}^k|_6|\overline{I}^{k+1}|_2\big)\text{d}\Omega \text{d}v\\
& +\int_\mathbb{I} |\overline{I}^{k+1}|_2\Big(\frac{v}{v'}|\overline{\sigma}_s|_\infty \big(|\rho^k|_\infty|\overline{I}'^k|_2+\|I'^k|_{W^{1,q}}|\overline{\rho}^{k+1}|_2\big) + |\overline{\sigma}'_s|_\infty |I^k|_{W^{1,q}}|\overline{\rho}^{k+1}|_2\Big) \text{d}\mathbb{I}\\
 \leq& D^k_\eta(t)\|\overline{I}^{k+1}\|^2_{L^2(\mathbb{R}^+\times S^2;L^2(\mathbb{V}))}+C|\overline{\rho}^{k+1}|^2_2\\
& +\eta\big(|\nabla \overline{\theta}^k|^2_2+|\sqrt{\rho^k} \overline{\theta}^k|^2_2+\|\overline{I}^{k}\|^2_{L^2(\mathbb{R}^+\times S^2;L^2(\mathbb{V}))}\big),
\end{split}
\end{equation}
where we have used the fact
\begin{equation}\label{opq}
\begin{split}
&\sigma^{k+1,k}_a-\sigma^{k,k-1}_a=\rho^{k+1}\sigma^{k+1,k}-\rho^k\sigma^{k,k-1}\\
=&\rho^{k+1}(\sigma^{k+1,k}-\sigma^{k,k})+\rho^{k+1}(\sigma^{k,k}-\sigma^{k,k-1})+\overline{\rho}^{k+1}\sigma^{k,k-1},
\end{split}
\end{equation}
and 
$D^k_\eta(t)$ is defined via
\begin{equation*}\begin{split}
D^k_\eta(t)=&C\big(1+1/\eta\big)|\rho^{k+1}|^2_\infty\|I^{k}\|^2_{L^2(\mathbb{R}^+\times S^2;  W^{1,q}(\mathbb{V})}\|\overline{\sigma}\|^2_{L^\infty(\mathbb{R}^+\times S^2;L^\infty(\mathbb{V}))}\\
&+C\alpha^2/\eta|\rho^k|^2_{\infty}+C(\alpha^2+\beta^2)\|I^{k}\|^2_{L^2(\mathbb{R}^+\times S^2;W^{1,q}(\mathbb{V})}.
\end{split}
\end{equation*}
From  (\ref{pri}), we also have  $\int_0^t D^k_\eta(s)\text{d}s\leq C+C_{\eta}t$, for $t\in[0,T^{**}]$.

Then considering   $\theta$. Multiplying $ (\ref{eq:1.2w})_3$ by $\overline{\theta}^{k+1}$ and integrating over $\mathbb{V}$, we have
\begin{equation*}\begin{split}
&\frac{1}{2}\frac{d}{dt}|\sqrt{\rho}^{k+1}\overline{\theta}^{k+1}|^2_2+\frac{\kappa}{c_v}|\nabla\overline{\theta}^{k+1} |^2_2\\
=& \int_{\mathbb{V}}\Big(\frac{1}{c_v}(Q(\nabla u^k)-Q(\nabla u^{k-1}))\overline{\theta}^{k+1}-\overline{\rho}^{k+1}\theta^k_t\overline{\theta}^{k+1}-\overline{\rho}^{k+1}u^{k-1}\cdot \nabla \theta^k\overline{\theta}^{k+1}\Big)\text{d}x\\
&+ \int_{\mathbb{V}}\Big(-\frac{1}{c_v}R\overline{\rho}^{k+1}\theta^k \text{div}u^{k-1}\overline{\theta}^{k+1}-\frac{1}{c_v}R\overline{\theta}^{k+1}\rho^{k+1}\overline{\theta}^{k+1}\text{div}u^{k}-\rho^{k+1}\overline{u}^k\cdot \nabla \theta^k\overline{\theta}^{k+1}\Big)\text{d}x\\
&+ \int_{\mathbb{V}}\Big(-\frac{1}{c_v}R\rho^{k+1}\theta^k \text{div}\overline{u}^{k}\overline{\theta}^{k+1}+L_1\overline{\theta}^{k+1}\Big)\text{d}x
\equiv:\sum_{i=11}^{26} I_i.
\end{split}
\end{equation*}
According to Gagliardo-Nirenberg inequality and  Minkowski's inequality we have
\begin{equation*}\begin{split}
& I_{11}\leq C|\nabla \overline{u}^k|_2|\overline{\theta}^{k+1}|_6(|\nabla u^k|_3+|\nabla u^{k-1}|_3), \ I_{12}\leq C |\overline{\rho}^{k+1}|_2|\theta^k_t|_3|\overline{\theta}^{k+1}|_6,\\
&I_{13}+I_{14}\leq C|\overline{\rho}^{k+1}|_{2}|\overline{\theta}^{k+1}|_6\|\theta^{k}\|_2\| u^{k-1}\|_2,\\ & I_{15}\leq C|\sqrt{\rho}^{k+1}\overline{\theta}^{k+1}|_2|\rho^{k+1}|^{\frac{1}{2}}_{\infty}|\overline{\theta}^{k+1}|_6\|\nabla u^{k}\|_1,\\
& I_{16}+I_{17}\leq C|\sqrt{\rho}^{k+1}\overline{\theta}^{k+1}|_2|\rho^{k+1}|^{\frac{1}{2}}_{\infty}| \nabla \overline{u}^k|_2\|\theta^{k}\|_2,
\end{split}
\end{equation*}
and from the assumptions (\ref{zhen1})-(\ref{jia345}) and the definition of $L_1$,
\begin{equation*}\begin{split}
I_{18}=&\frac{1}{c_v} \int_{\mathbb{V}}\int_0^\infty\int_{S^2}\Big(1-\frac{u^k\cdot\Omega}{c}\Big)\overline{\theta}^{k+1}\Big(-\sigma^{k+1,k}_a\overline{I}^{k+1}\Big)\text{d}\Omega\text{d}v\text{d}x\\
\leq & C(1+||\nabla u^k||_1)|\sqrt{\rho}^{k+1}\overline{\theta}^{k+1}|_2|\rho^{k+1}|^{\frac{1}{2}}_\infty\|\overline{I}^{k+1}\|_{L^2(\mathbb{R}^+\times S^2;L^2)}\|\sigma^{k+1,k}\|_{L^2(\mathbb{R}^+\times S^2;L^\infty)},\\
 I_{19}=&\frac{1}{c_v} \int_{\mathbb{V}}\int_0^\infty \int_{S^2}\Big(1-\frac{u^k\cdot\Omega}{c}\Big)\overline{\theta}^{k+1}\Big(-I^k(\sigma^{k+1,k}_a-\sigma^{k,k-1}_a)\Big) \text{d}\Omega \text{d}v\text{d}x\\
\leq &C(1+||\nabla u^k||_1)|\sqrt{\rho}^{k+1}\overline{\theta}^{k+1}|_2|\rho^{k+1}|^{1/2}_\infty\|\overline{\sigma}\|_{L^2(\mathbb{R}^+\times S^2;L^\infty)}\\
&\cdot \big(| \overline{\rho}^{k+1}|_{2}\|I^{k}\|_{L^2(\mathbb{R}^+\times S^2;W^{1,q})}
+| \overline{\theta}^{k}|_{6}\|I^{k}\|_{L^2(\mathbb{R}^+\times S^2;H^1)}\big)\\
&+C(1+||\nabla u^k||_1)|\overline{\rho}^{k+1}|_2 |  \overline{\theta}^{k+1}|_{6} \| I^{k}\|_{L^2(\mathbb{R}^+\times S^2;H^1)}\|\sigma^{k,k-1}\|_{L^2(\mathbb{R}^+\times S^2;L^\infty)},\\
I_{20}=&\frac{1}{c_v} \int_{\mathbb{V}}\int_\mathbb{I}\frac{v}{v'}\Big(1-\frac{u^k\cdot\Omega}{c}\Big)\overline{\theta}^{k+1}I'^k(\sigma^{k+1}_s-\sigma^{k}_s)\text{d}\mathbb{I}\text{d}x\\
 \leq& C\alpha(1+||\nabla u^k||_1)|\overline{\theta}^{k+1}|_6|\overline{\rho}^{k+1}|_{2}\|I^k\|_{L^2(\mathbb{R}^+\times S^2;H^1)},\\
 I_{21}=&\frac{1}{c_v} \int_{\mathbb{V}}\int_\mathbb{I}\frac{v}{v'}\Big(1-\frac{u^k\cdot\Omega}{c}\Big)\overline{\theta}^{k+1}\sigma^{k+1}_s\overline{I}'^{k+1}\text{d}\mathbb{I}\text{d}x\\
 \leq&  C\alpha(1+||\nabla u^k||_1)|\sqrt{\rho}^{k+1}\overline{\theta}^{k+1}|_2|\rho^{k+1}|^{\frac{1}{2}}_\infty\|\overline{I}^{k+1}\|_{L^2(\mathbb{R}^+\times S^2;L^2)},\\
 I_{22}=&\frac{1}{c_v} \int_{\mathbb{V}}\int_\mathbb{I}\overline{\theta}^{k+1}\Big(1-\frac{u^k\cdot \Omega}{c}\Big)I^{k}\Big((\sigma'_s)^{k+1}-(\sigma'_s)^{k}\Big)\text{d}\mathbb{I}\text{d}x\\
 \leq& C\alpha^{\frac{1}{2}}_2(1+||\nabla u^k||_1)| \overline{\theta}^{k+1}|_6|\overline{\rho}^{k+1}|_{2}\|I^{k}\|_{L^2(\mathbb{R}^+\times S^2;H^1)},\\
 I_{23}=&\frac{1}{c_v} \int_{\mathbb{V}}\int_0^\infty \int_{S^2}\Big(-\frac{\overline{u}^k\cdot \Omega}{c}\Big)S\overline{\theta}^{k+1}\text{d}\Omega \text{d}v\text{d}x\leq C|\overline{\theta}^{k+1}|_6|\nabla\overline{u}^{k}|_2\|S\|_{L^1(\mathbb{R}^+;L^{\frac{3}{2}}(\mathbb{V}))},\\
\end{split}
\end{equation*}
where in $I_{19}$ we have used the fact (\ref{opq}).  Similarly we have
\begin{equation*}
\begin{split}
 I_{24}=&\frac{1}{c_v} \int_{\mathbb{V}}\int_0^\infty \int_{S^2}\overline{\theta}^{k+1}\Big(-\frac{\overline{u}^k\cdot \Omega}{c}\Big)\Big(-\sigma^{k,k-1}_aI^{k}
\Big) \text{d}\Omega \text{d}v\text{d}x\\
\leq &C|\overline{\theta}^{k+1}|_6|\sqrt{\rho}^k\overline{u}^{k}|_2|\rho^k|^{1/2}_\infty\|I^{k}\|_{L^2(\mathbb{R}^+\times S^2;H^1)}\|\sigma^{k+1,k}\|^2_{L^2(\mathbb{R}^+\times S^2;L^\infty)},\\
 I_{25}=&\frac{1}{c_v} \int_{\mathbb{V}}\int_0^\infty \int_{S^2}\overline{\theta}^{k+1}\Big(-\frac{\overline{u}^k\cdot \Omega}{c}\Big)\Big(\int_0^\infty \int_{S^2}\frac{v}{v'}\sigma^{k}_sI'^{k}\text{d}\Omega' \text{d}v'\Big) \text{d}\Omega \text{d}v\text{d}x\\
\leq & C\alpha|\rho^k|^{\frac{1}{2}}_{\infty}|\overline{\theta}^{k+1}|_6|\sqrt{\rho}^k\overline{u}^{k}|_2\|I^{k}\|_{L^2(\mathbb{R}^+\times S^2;H^1)},\\
%
 I_{26}=&\frac{1}{c_v} \int_{\mathbb{V}}\int_0^\infty \int_{S^2}\Big(-\frac{\overline{u}^k\cdot \Omega}{c}\Big)\overline{\theta}^{k+1}\Big(\int_0^\infty \int_{S^2}(\sigma'_s)^{k}I^{k}\text{d}\Omega' \text{d}v'\Big) \text{d}\Omega \text{d}v\text{d}x\\
\leq & C\alpha^{\frac{1}{2}}_2|\rho^k|^{\frac{1}{2}}_{\infty}|\overline{\theta}^{k+1}|_6|\sqrt{\rho}^k\overline{u}^{k}|_2\|I^{k}\|_{L^2(\mathbb{R}^+\times S^2;H^1)}.
\end{split}
\end{equation*}
Then combining the above estimates,  from  Lemma \ref{pang} and (\ref{pri}),  we have
\begin{equation}\label{go100}\begin{split}
&\frac{d}{dt}|\sqrt{\rho}^{k+1}\overline{\theta}^{k+1}|^2_2+|\nabla\overline{\theta}^{k+1} |^2_2
\leq E^k_\eta(t)|\sqrt{\rho}^{k+1}\overline{\theta}^{k+1}|^2_2+E^k_2(t)|\overline{\rho}^{k+1}|^2_{2}\\
&\ \ +E^k_3(t)\|\overline{I}^{k+1}\|^2_{L^2(\mathbb{R}^+\times S^2;L^2(\mathbb{V}))}+C(|\nabla\overline{u}^{k}|^2_2+|\sqrt{\rho}^k\overline{u}^{k}|^2_2)+\eta(|\nabla\overline{\theta}^{k}|^2_2+|\sqrt{\rho}^k\overline{\theta}^{k} |^2_2),
\end{split}
\end{equation}
where $E^k_{\eta}(t)$, $E^k_2(t)$ and $E^k_3(t)$ satisfying:
\begin{equation*}
\begin{cases}
\displaystyle
\ E^k_\eta(t)=C+C\|\rho^{k+1}|_{\infty}(\|\theta^{k}\|^2_2+\| u^{k}\|^2_2)\\[6pt]
\displaystyle
 \qquad \ \ \ +1/\eta(1+\|\nabla u^k\|^2_1)|\rho^{k+1}|_\infty\|I^k\|^2_{L^2(\mathbb{R}^+\times S^2; W^{1,q})}\|\overline{\sigma}\|^2_{L^2(\mathbb{R}^+\times S^2;L^\infty)},\\[6pt]
\displaystyle
\ \int_0^t E^k_\eta(s)\leq C+C_\eta t, \quad \int_0^t (E^k_2(s)+E^k_3(s))\text{d}s\leq C, \quad \text{for} \quad  t\in[0,T^{**}].
\end{cases}
\end{equation*}

Finally, multiplying $ (\ref{eq:1.2w})_4$ by $\overline{u}^{k+1}$ and integrating over $\mathbb{V}$, we have
\begin{equation*}\begin{split}
&\frac{1}{2}\frac{d}{dt}|\sqrt{\rho}^{k+1}\overline{u}^{k+1}|^2_2+\mu|\nabla\times \overline{u}^{k+1} |^2_2+(\lambda+\mu)|\text{div}\overline{u}^{k+1} |^2_2\\
=& \int_{\mathbb{V}}\Big(-\overline{\rho}^{k+1}u^k_t\cdot\overline{u}^{k+1}-\overline{\rho}^{k+1}(u^{k-1}\cdot\nabla u^k)\cdot\overline{u}^{k+1}
-\rho^{k+1}(\overline{u}^{k}\cdot\nabla u^k)\cdot\overline{u}^{k+1}\\
&-R\nabla(\rho^{k+1}\overline{\theta}^{k+1})\cdot\overline{u}^{k+1}-R\nabla(\overline{\rho}^{k+1}\theta^{k})\cdot\overline{u}^{k+1}+L_2\overline{u}^{k+1}\Big)\text{d}x
\equiv:\sum_{i=27}^{37} I_i.
\end{split}
\end{equation*}

According to the Gagliardo-Nirenberg inequality, Minkowski's inequality and H\"older's inequality, we easily have
\begin{equation*}\begin{split}
I_{27}\leq& C |\overline{\rho}^{k+1}|_2|u^k_t|_3|\overline{u}^{k+1}|_6, \ I_{28}\leq C|\overline{\rho}^{k+1}|_{2}|\nabla\overline{u}^{k+1}|_2\|\nabla u^k\|_1\|\nabla u^{k-1}\|_1,\\
I_{29}\leq& C|\rho^{k+1}|^{\frac{1}{2}}_{\infty}|\sqrt{\rho}^{k+1}\overline{u}^{k+1}|_2\|\nabla u^k\|_1|\nabla \overline{u}^k|_2,\\
 I_{30}\leq& C|\rho^{k+1}|^{\frac{1}{2}}_{\infty}|\sqrt{\rho}^{k+1}\overline{\theta}^{k+1}|_2|\nabla\overline{u}^{k+1}|_2,\  I_{31}\leq C|\overline{\rho}^{k+1}|_{2}|\nabla\overline{u}^{k+1}|_2| \theta^k|_\infty,\\
I_{32}=&-\frac{1}{c} \int_{\mathbb{V}}\int_0^\infty\int_{S^2}\Omega\cdot\overline{u}^{k+1}\Big(-\sigma^{k+1}_a\overline{I}^{k+1}\Big)\text{d}\Omega\text{d}v\text{d}x\\
\leq & C|\sqrt{\rho}^{k+1}\overline{u}^{k+1}|_2|\rho^{k+1}|^{\frac{1}{2}}_\infty\|\overline{I}^{k+1}\|^2_{L^2(\mathbb{R}^+\times S^2;L^2(\mathbb{V}))}\|\sigma^{k+1,k}\|^2_{L^2(\mathbb{R}^+\times S^2;L^\infty)},\\
 I_{33}=&-\frac{1}{c} \int_{\mathbb{V}}\int_0^\infty \int_{S^2}\Omega\cdot\overline{u}^{k+1}\Big(-I^k(\sigma^{k+1}_a-\sigma^{k}_a)\Big) \text{d}\Omega \text{d}v\text{d}x\\
\leq &C|\rho^{k+1}|^{\frac{1}{2}}_{\infty}|\overline{\rho}^{k+1}|_{2}|\sqrt{\rho}^{k+1}\overline{u}^{k+1}|_2\|I^{k}\|_{L^2(\mathbb{R}^+\times S^2;W^{1,q}(\mathbb{V})}\|\overline{\sigma}\|_{L^2(\mathbb{R}^+\times S^2;L^\infty(\mathbb{V}))}\\
&+C|\sqrt{\rho}^{k+1}\overline{u}^{k+1}|_2|\sqrt{\rho}^{k+1}\overline{\theta}^{k+1}|_2\|I^{k}\|_{L^2(\mathbb{R}^+\times S^2;W^{1,q}(\mathbb{V})}\|\overline{\sigma}\|_{L^2(\mathbb{R}^+\times S^2;L^\infty(\mathbb{V}))}\\
%
&+C|\overline{\rho}^{k+1}|_2|\nabla\overline{u}^{k+1}|_2\|I^{k}\|_{L^2(\mathbb{R}^+\times S^2;H^1(\mathbb{V}))}\|\sigma^{k}\|_{L^2(\mathbb{R}^+\times S^2;L^\infty(\mathbb{V}))},\\
\end{split}
\end{equation*}
where in $I_{33}$ we have used the fact
\begin{equation*}
\begin{split}
&\sigma^{k+1}_a-\sigma^{k}_a=\rho^{k+1}\sigma^{k+1}-\rho^k\sigma^{k}
=\rho^{k+1}[(\sigma^{k+1}-\sigma^{k,k+1})+(\sigma^{k,k+1}-\sigma^{k,k})]+\overline{\rho}^{k+1}\sigma^{k,k},
\end{split}
\end{equation*}
and similarly, we have
\begin{equation*}\begin{split}
 I_{34}=&-\frac{1}{c} \int_{\mathbb{V}}\int_\mathbb{I}\frac{v}{v'}\Omega\cdot\overline{u}^{k+1}\sigma^{k+1}_s\overline{I}'^{k+1}\text{d}\mathbb{I}\text{d}x\\
 \leq &  C\alpha|\sqrt{\rho}^{k+1}\overline{u}^{k+1}|_2|\rho^{k+1}|^{\frac{1}{2}}_\infty\|\overline{I}^{k+1}\|_{L^2(\mathbb{R}^+\times S^2;L^2(\mathbb{V}))},\\
 I_{35}=&-\frac{1}{c} \int_{\mathbb{V}}\int_\mathbb{I}\frac{v}{v'}\Omega\cdot\overline{u}^{k+1}I'^k(\sigma^{k+1}_s-\sigma^{k}_s)\text{d}\mathbb{I}\text{d}x
 \leq C\alpha|\nabla\overline{u}^{k+1}|_2|\overline{\rho}^{k+1}|_{2}\|I^k\|_{L^2(\mathbb{R}^+\times S^2;H^1(\mathbb{V}))},\\
\end{split}
\end{equation*}

\begin{equation*}\begin{split}
 I_{36}=&\frac{1}{c} \int_{\mathbb{V}}\int_\mathbb{I}\overline{u}^{k+1}I^{k}\Big((\sigma'_s)^{k+1}-(\sigma'_s)^{k}\Big)\text{d}\mathbb{I}\text{d}x
 \leq C\alpha^{\frac{1}{2}}_2|\nabla \overline{u}^{k+1}|_2|\overline{\rho}^{k+1}|_{2}\|I^{k}\|_{L^2(\mathbb{R}^+\times S^2;H^1)},\\
%
 I_{37}=&\frac{1}{c} \int_{\mathbb{V}}\int_\mathbb{I}\overline{u}^{k+1} (\sigma'_s)^{k+1}\overline{I}^{k+1}\text{d}\mathbb{I}\text{d}x
 \leq C\alpha^{\frac{1}{2}}_2|\sqrt{\rho}^{k+1}\overline{u}^{k+1}|_2|\rho^{k+1}|^{\frac{1}{2}}_\infty\|\overline{I}^{k+1}\|_{L^2(\mathbb{R}^+\times S^2;L^2)}.
\end{split}
\end{equation*}
Then combining   the above estimates for $I_i$ $(i=27,...,37)$, from Lemma \ref{gag22} and (\ref{pri}),  we have
\begin{equation}\label{gogo1}\begin{split}
&\frac{d}{dt}|\sqrt{\rho}^{k+1}\overline{u}^{k+1}|^2_2+|\nabla\overline{u}^{k+1} |^2_2
\leq F^k_\eta(t)|\sqrt{\rho}^{k+1}\overline{u}^{k+1}|^2_2+F^k_2(t)|\overline{\rho}^{k+1}|^2_{2}\\
&\qquad+F^k_3(t)\|\overline{I}^{k+1}\|^2_{L^2(\mathbb{R}^+\times S^2;L^2(\mathbb{V}))}+F^k_4(t)|\sqrt{\rho}^{k+1}\overline{\theta}^{k+1}|^2_2+\eta|\nabla\overline{u}^{k}|^2_2,
\end{split}
\end{equation}
where $F^k_{\eta}(t)$, $F^k_2(t)$, $F^k_3(t)$ and $F^k_4(t)$ satisfying:
\begin{equation*}
\begin{cases}
\displaystyle
F^k_\eta(t)=C\Big(1+1/\eta|\rho^{k+1}|_{\infty}\|\nabla u^{k}\|^2_1+|\rho^{k+1}|_{\infty}\|I^{k}\|^2_{L^2(\mathbb{R}^+\times S^2;W^{1,q}(\mathbb{V})}\|\overline{\sigma}\|^2_{L^2(\mathbb{R}^+\times S^2;L^\infty(\mathbb{V}))}\Big),\\[6pt]
\displaystyle
\int_0^t F^k_\eta(s)\text{d}s\leq C+C_\eta t,\quad \int_0^t (F^k_2(s)+F^k_3(s)+F^k_4(s)\big)\text{d}s\leq C,\quad \text{for}\quad t\in [0,T^{**}].
\end{cases}
\end{equation*}

Next, we denote that 
\begin{equation*}\begin{split}
\Lambda^{k+1}(T^{**},\epsilon)=&\sup_{0\leq t \leq T^{**}}\|\overline{I}^{k+1}(t)\|^2_{L^2(\mathbb{R}^+\times S^2;L^2(\mathbb{V}))}+\sup_{0\leq
t \leq T^{**}}|\overline{\rho}^{k+1}(t)|^2_{
 2}\\
&+\epsilon \sup_{0\leq t \leq T^{**}}|\sqrt{\rho}^{k+1}\overline{\theta}^{k+1}(t)|^2_2+\sup_{0\leq t \leq T^{**}}|\sqrt{\rho}^{k+1}\overline{u}^{k+1}(t)|^2_2,
\end{split}
\end{equation*}
then from (\ref{go64})-(\ref{gogo1}), we have
\begin{equation*}\begin{split}
&\Lambda^{k+1}(T^{**},\epsilon)+\int_0^{T^{**}} (\epsilon  |\nabla\overline{\theta}^{k+1} |^2_2+|\nabla\overline{u}^{k+1} |^2_2)\text{d}t\\
\leq& \int_0^{T^{**}} G^k_{\epsilon,\eta} \Lambda^{k+1}(t,\epsilon)\text{d}t+\int_0^{T^{**}} C\Big((\eta+\epsilon \eta) |\nabla\overline{\theta}^{k} |^2_2+(\epsilon+\eta)|\nabla\overline{u}^{k} |^2_2\Big)\text{d}t\\
&+\epsilon T^{**}\sup_{0\leq t \leq T^{**}}|\sqrt{\rho}^{k}\overline{u}^{k}(t)|^2_2+ \eta T^{**} \sup_{0\leq t \leq T^{**}}\|\overline{I}^{k}(t)\|^2_{L^2(\mathbb{R}^+\times S^2;L^2(\mathbb{V}))}
\end{split}
\end{equation*}
for some $G^k_{\epsilon,\eta}$ such that  $\int_{0}^{t}G^k_{\epsilon,\eta}(s)\text{d}s\leq (1+\epsilon)(C+C_{\eta} t)$ for $0\leq t \leq T^{**}$.
According to Gronwall's inequality, we have
\begin{equation*}\begin{split}
&\Lambda^{k+1}(T^{**},\epsilon)+\int_{0}^{T^{**}}\Big(\epsilon|\nabla\overline{\theta}^{k+1} |^2_2+|\nabla\overline{u}^{k+1}|^2_2\Big)\text{d}s\\
\leq & \Big( \int_0^{T^{**}} C\big((\eta+\epsilon \eta) |\nabla\overline{\theta}^{k} |^2_2
+(\epsilon+\eta)|\nabla\overline{u}^{k} |^2_2\big)\text{d}t
+\epsilon T^{**}\sup_{0\leq t \leq T^{**}}|\sqrt{\rho}^{k}\overline{u}^{k}(t)|^2_2\\
&+\epsilon \eta T^{**} \sup_{0\leq t \leq T^{**}}|\sqrt{\rho}^{k}\overline{\theta}^{k}(t)|^2_2+\eta T^{**} \sup_{0\leq t \leq T^{**}}\|\overline{I}^{k}(t)\|^2_{L^2(\mathbb{R}^+\times S^2;L^2(\mathbb{V}))}\Big)\exp{\big((1+\epsilon)(C+C_{\eta} t)\big)}.
\end{split}
\end{equation*}
Due to $0< T^{**}\leq 1$, first, we  can choose $0<\epsilon=\epsilon_0<1$ small enough such that
$$
(1+C) \epsilon_0\exp\big((1+\epsilon_0)C\big)\leq \frac{1}{8};
$$
second, we can choose $\eta=\eta_0$ small enough such that
$$
(C+1)(\epsilon_0\eta_0+\eta_0)  \exp\big((1+\epsilon_0)C\big)\leq \frac{1}{8};
$$
then finally, we can choose $T^{**}=T_*$ small enough such that
$$ \text{exp}\big((1+\epsilon_0)C_{\eta_0}T_*\big)\leq 2.$$
Then when $\Lambda^{k+1}=\Lambda^{k+1}(T_*,\epsilon_0,\eta_0)$, we have
\begin{equation*}\begin{split}
\sum_{k=1}^{\infty}\Big(  \Lambda^{k+1}+\int_{0}^{T_*}\Big(|\nabla\overline{\theta}^{k+1} |^2_2+|\nabla\overline{u}^{k+1}|^2_2\Big)\text{d}t\Big)\leq C<+\infty.
\end{split}
\end{equation*}
Thus we know that the full consequence $(I^k,\rho^k,u^k,\theta^k)$ converges to a limit $(I,\rho,u,\theta)$ in the following strong sense:
\begin{equation}\label{stronga}
\begin{split}
&I^k\rightarrow I \ \text{in}\ L^\infty([0,T_*];L^2(R^+\times S^2;L^2(\mathbb{V}))),\\
&\rho^k\rightarrow\rho \ \text{in}\ L^\infty([0,T_*];L^2(\mathbb{V})),\ (\theta^k,u^k)\rightarrow (\theta,u)\ \text{in}\ L^2([0,T_*];D^1(\mathbb{V}),
\end{split}
\end{equation}
wihch,   along with  (\ref{pri}), show that $(I,\rho,u,\theta)$ satisfies the regularities $(\ref{yuo})$ and estimates (\ref{pri}).
Then it is easy to see $(I,\rho,u,\theta)$ is a weak solution in the sense of distribution. 

\underline{Step 2}. The uniqueness.   Let $(I_1,\rho_1,u_1,\theta_1)$ and $(I_2,\rho_2,u_2,\theta_2)$ be two strong solutions to IBVP (\ref{eq:1.2})-(\ref{eq:2.2hh}) with (\ref{fan2}) satisfying the regularity (\ref{yuo}). We denote that
$$
\overline{I}=I_1-I_2,\quad  \overline{\rho}=\rho_1-\rho_2,\quad  \overline{u}=u_1-u_2, \quad \overline{\theta}=\theta_1-\theta_2.
$$
Via the same method as in the derivations of (\ref{go64})-(\ref{gogo1}), letting
$$
\Lambda(t)=\|\overline{I}\|^2_{L^2(\mathbb{R}^+\times S^2;L^2(\mathbb{V}))}+|\overline{\rho}|^2_{ 2}+|\sqrt{\rho}_1\overline{\theta}|^2_2+|\sqrt{\rho}_1\overline{u}|^2_2,
$$
we similarly have
\begin{equation}\label{gonm}\begin{split}
\frac{d}{dt}\Lambda(t)+|\nabla \overline{u}|^2_2+|\nabla \overline{\theta}|^2_2\leq H(t)\Lambda (t),
\end{split}
\end{equation}
where $ \int_{0}^{t}H(s)ds\leq C$, for $t\in [0,T_*]$.  Then from the Gronwall's inequality and $u\cdot n|_{\partial \mathbb{V}}=0$, we conclude that
$\overline{I}=\overline{\rho}=\overline{u}=\overline{\theta}=0$, then the uniqueness is obtained.

\underline{Step 3}. The time-continuity can be obtained by the same method as in  Lemma \ref{lem1}.
\end{proof}

\section{ necessity and sufficiency of the compatibility condition}

\textbf{Now we show the proof for Theorem \ref{th2}.}

\begin{proof}\underline{Step}:1.
Necessity. Let $(I,\rho,u, \theta)$ be a strong solution to  (\ref{eq:1.2})-(\ref{eq:2.2hh}) with (\ref{fan1}) or (\ref{fan2}) and the regularity shown in Definition \ref{strong}. Then due to $(\ref{eq:1.2})$, we have
\begin{equation}\label{da9}\begin{split}
Lu(t)+\nabla P_{m}(t)+\frac{1}{c}\int_0^\infty \int_{S^2}A_r(t)\Omega \text{d}\Omega \text{d}v=&\sqrt{\rho}(t) G^1(t),\\
-\frac{1}{c_v}(\kappa\triangle \theta+Q( u))-\int_0^\infty \int_{S^{2}}\frac{1}{c_v} \Big(1-\frac{u\cdot \Omega}{c}\Big)A_r \text{d}\Omega \text{d}v=&\sqrt{\rho}(t) G^2(t),
\end{split}
\end{equation}
for  $0\leq t \leq T_*$, where 
$$G^1(t)=\sqrt{\rho} (-u_t-u\cdot\nabla u), \ G^2(t)=\sqrt{\rho} (-\theta_t-u\cdot\nabla \theta-R\theta \text{div}u). $$
 Since
$$
(\sqrt{\rho}u_t,\sqrt{\rho}\theta_t, \sqrt{\rho}u\cdot\nabla u,\sqrt{\rho}u\cdot\nabla \theta, \sqrt{\rho}\theta \text{div}u) \in  L^\infty([0,T_*];L^2),
$$
we have $(G^1,G^2)\in  L^\infty([0,T];L^2)$. So there exists a sequence $\{t_k\}$ ($t_k\rightarrow 0$) such that
$$
(G^1(t_k),G^2(t_k))\rightharpoonup (f,g) \quad  \text{weak-* \ in} \quad L^2 \quad \text{for some} \quad (f,g) \in L^2.
$$
So, let $t=t_k\rightarrow 0$ in (\ref{da9}), we obtain
\begin{equation}\label{jojo}\begin{split}
Lu(0)+\nabla P_{m}(\rho(0))+\frac{1}{c}\int_0^\infty \int_{S^2}A_r(0)\Omega \text{d}\Omega \text{d}v=&\sqrt{\rho}(0) f,\\
-\frac{1}{c_v}(\kappa\triangle \theta(0)+Q(u(0))-\int_0^\infty \int_{S^{2}}\frac{1}{c_v} \Big(1-\frac{u(0)\cdot \Omega}{c}\Big)A_r(0) \text{d}\Omega \text{d}v=&\sqrt{\rho}(0) g.
\end{split}
\end{equation}
Combining with the strong convergence (\ref{coco}) and (\ref{jojo}), we know that the necessity of the compatibility condition is obtained. Moreover, from the construction of our strong solutions in Section $4$, we easily deduce that $f=g_1$ and $g=g_2$.

\underline{Step}:2.
Sufficiency.  Let $(I_0,\rho_0,u_0,\theta_0)$ be the initial data satisfying (\ref{gogo})-(\ref{kkk}). Then there exists a unique strong  solution $(I,\rho,u,\theta)$ on $\mathbb{R}^+\times S^2\times [0,T_*]\times\mathbb{V} $ to (\ref{eq:1.2})-(\ref{eq:2.2hh}) with (\ref{fan1}) or (\ref{fan2}).
Then we only need to make sure that
$$
I(v,\Omega,0,x)=I_0,\ \rho(0,x)=\rho_0,\ u(0,x)=u_0(x), \   \theta(0,x)=\theta_0, \ x\in \mathbb{V}.
$$
From Remark \ref{initial},
it remains to prove that $u(0,x)=u_0(x)$ and $ \theta(0,x)=\theta_0(x)$ when $ x\in V$. Let $\overline{u}_0=u_0-u(0,x)$ and $\overline{\theta}_0=\theta_0-\theta(0,x)$. According to the proof of the necessity, we know that $(I(v,\Omega,0,x),\rho(0,x),u(0,x),\theta(0,x))$ also satisfies the relation (\ref{kkk}) for $(g_1,g_2)\in L^2$. Then we quickly know that
$(\overline{\theta}_0,\overline{u}_0)\in D^1_0(V)\cap D^2(V)$ is the unique solution of the elliptic problem (\ref{zhen101}) in $V$, and thus $\overline{u}_0=0$ and $\overline{\theta}_0=0$ in $V$, which implies that
$u(0,x)=u_0(x),\ \theta(0,x)=\theta_0(x),  \ x\in V$.
\end{proof}

\section{Beal-Kato-Majda blow-up criterion.}

Now we prove (\ref{eq:2.91}). Let $(I, \rho, u,\theta)$ be the strong solution obtained in Theorem \ref{th1} to   (\ref{eq:1.2})-(\ref{eq:2.2hh}) with (\ref{fan1}) in $\mathbb{R}^+S^2 \times [0,\overline{T})\times \mathbb{V}$. We assume the opposite of (\ref{eq:2.91}) holds, i.e.,
\begin{equation}\label{we11*}
\begin{split}
\lim \sup_{T\mapsto \overline{T}}\big(\|\nabla I\|_{L^2(\mathbb{R}^+\times S^2; L^\infty([0,T]; L^2(\mathbb{V})))}+ |\rho|_{L^\infty([0,T]\times \mathbb{V})}+ |\theta|_{L^\infty([0,T]\times \mathbb{V})}\big)=C_0<\infty.
\end{split}
\end{equation}
In this section, $C\geq 1$  denotes  a generic  constant   depending  only on $(I_0,\rho_0,u_0,\theta_0)$, $C_0$,  $\alpha$, $\beta$, $\kappa$, q, $R$, $c_v$, $\mu$, $\lambda$,  $c$,  $|\mathbb{V}| $ and $T$.

\subsection{The lower order  estimate for $|u|_{L^\infty([0,\overline{T}]; D^1(\mathbb{V})}$}
We first have the $L^\infty$ bound of $I$ and some classical energy estimates.
 \begin{lemma}\label{s2}
\begin{equation*}
\begin{split}
\|I\|_{L^2(\mathbb{R}^+\times S^2; L^\infty([0,T]\times \Omega))} +\|I_t\|_{L^2(\mathbb{R}^+\times S^2; L^\infty([0,T];L^2( \mathbb{V})))}  \leq& C,\quad 0\leq t<  T,\\
\int_0^\infty \int_{S^2}\|A_r(t)\|_{ L^\infty(\mathbb{V})} \text{d}\Omega \text{d}v+|\sqrt\rho (u, \theta)(t)|_2+\|(\nabla u,\nabla \theta)\|_{L^2([0,T]\times \mathbb{V})}\leq& C,\quad 0\leq t<  T.
\end{split}
\end{equation*}
 \end{lemma}

\begin{proof}First, let $2\leq r$, multiplying $ (\ref{eq:1.2})_1$ by $r|I|^{r-2}I$ and integrating over $\mathbb{V}$, we have
\begin{equation}\label{kthyt}
\begin{split}
\frac{d}{dt}|I|^r_{r}
\leq  Cr|I|^{r-1}_r\Big(|S|_{r}
+|\rho|_\infty\int_0^\infty \int_{S^2} \frac{v}{v'}|I'|_{r} | \overline{\sigma}_s|_\infty \text{d}\Omega' \text{d}v'\Big).
\end{split}
\end{equation}
According to  assumptions (\ref{zhen1})-(\ref{jia345})  and  (\ref{kthyt}), we deduce
\begin{equation}\label{kthyt1}
\begin{split}
\frac{d}{dt}|I|^2_{r}
\leq& C\Big(|I|^2_{r}+|S|^2_{r}+ |I|^2_{L^2(\mathbb{R}^+\times S^2; L^r))}  \int_0^\infty \int_{S^2}\Big|\frac{v}{v'}\Big|^2|\overline{\sigma}_s|^2_\infty\text{d}\Omega' \text{d}v'\Big).
\end{split}
\end{equation}
From Gronwall's inequality, we have
\begin{equation*}
\begin{split}
&\|I(v,\Omega,t,x)\|^2_{C([0,T];L^r)}
\leq  \exp (CT)\Big(|I_0|^2_{r}+\int_0^{T}|S|^2_{r}\text{d}s+T\int_0^\infty \int_{S^2} \Big|\frac{v}{v'}\Big|^2 |\overline{\sigma}_s|^2_\infty \text{d}\Omega' \text{d}v'\Big).
\end{split}
\end{equation*}
Integrating the above inequality in $\mathbb{R}^+\times S^2$ with respect to $(v,\Omega)$,  we have
$$
\|I\|^2_{L^2(\mathbb{R}^+\times S^2;C([0,T];L^r))}\leq C,\quad 0\leq t<  T,
$$
where  $C$ is independent of $r$. Letting $r\rightarrow \infty$, we obtain the  bound of $\|I\|^2_{L^2(\mathbb{R}^+\times S^2;C([0,T];L^\infty))}$.
Moreover, the estimate for $I_t$ follows quickly from $ I_t=-c\Omega\cdot\nabla I+cA_r$. 

Secondly,   multiplying  $(\ref{eq:1.2})_3$ by $u$,     $(\ref{eq:1.2})_4$ by $\theta$ respectively and integrating the resulting equations over $\mathbb{V}$ by parts, we have

\begin{equation}\label{zhu2s}
\begin{split}
&\frac{1}{2}\frac{d}{dt}\int_{\mathbb{V}} \rho |u|^2 \text{d}x+\int_{\mathbb{V}} \big(\mu |\nabla u|^2+(\lambda+\mu)(\text{div}u)^2\big)\text{d}x\\
= &\int_{\mathbb{V}}  P_m \text{div}u\text{d}x-\frac{1}{c}\int_{\mathbb{V}}\int_0^\infty \int_{S^2}A_r\Omega\cdot u \text{d}\Omega \text{d}v\text{d}x\leq \frac{\mu}{2}|\nabla u|^2_2+C,  \\
&\frac{1}{2}\frac{d}{dt}\int_{\mathbb{V}} \rho |\theta|^2\text{d}x+\int_{\mathbb{V}} \kappa |\nabla \theta|^2\text{d}x\\
\leq &\int_{\mathbb{V}}C \big(  \rho \theta^2 |\text{div}u|+ |\nabla u|^2 |\theta|+N_r \theta\big)\text{d}x\leq C|\nabla u|^2_2+C,
\end{split}
\end{equation}
which immediately implies the desired conclusions.
\end{proof}


Now we improve the energy estimate obtained in Lemma \ref{s2}.
  \begin{lemma}\label{abs3}
If $\lambda <3 \mu$, it holds that 
\begin{equation}\label{keyq}
\begin{split}
\int_{\mathbb{V} }\rho|u(t)|^4 \text{dx}+\int_0^T \int_{\mathbb{V}} |u|^2 |\nabla u|^2\text{d}x\text{d}t\leq C,\quad 0\leq t< T.
\end{split}
\end{equation}
 \end{lemma}

\begin{proof} For any $\lambda$ satisfying that $\lambda <3\mu$, there must exist a sufficiently small constant $\alpha_{\lambda\mu} >0$ such that:
\begin{equation}\label{zheng9}
\lambda <(3-\alpha_{\lambda\mu})\mu<3\mu.
\end{equation}
So we only need to show that (\ref{abs3}) holds under the assumption (\ref{zheng9}).

 First, multiplying $ (\ref{eq:1.2})_3$ by $r|u|^{r-2}u$ $(r\geq 3)$ and integrating the resulting equation over $\mathbb{V}$ by parts, then we have
\begin{equation}\label{lz1}
\begin{split}
 \frac{d}{dt}\int_{\mathbb{V}}& \rho |u|^r\text{d}x+\int_{\mathbb{V} }H_r\text{d}x
=-r(r-2)(\mu+\lambda)\int_{\mathbb{V}} \text{div}u |u|^{r-3}u\cdot \nabla |u|\text{d}x\\
&+\int_{\mathbb{V}} r P_m\text{div }(|u|^{r-2}u)\text{d}x-\frac{1}{c}\int_{\mathbb{V} }\int_0^\infty \int_{S^2}r|u|^{r-2}A_ru\cdot \Omega \text{d}\Omega \text{d}v\text{d}x,
\end{split}
\end{equation}
where
$$
H_r=r|u|^{r-2}\big(\mu|\nabla u|^2+(\mu+\lambda)|\text{div}u|^2+\mu(r-2)|\nabla |u||^2\big).
$$

For any given $\epsilon_1\in (0,1)$, we define a nonnegative function which will be determined in $\textbf{Step}\ 2$ as follows
$$
\phi(\epsilon_0,\epsilon_1,r)=\left\{ \begin{array}{llll}
\frac{\mu \epsilon_1(r-1)}{3\big(-\frac{\mu(4-\epsilon_0)}{3}-\lambda+\frac{r^2(\lambda+\mu)}{4(r-1)}\big)}, \quad \text{if}\quad \frac{r^2(\mu+\lambda)}{4(r-1)} -\frac{\mu(4-\epsilon_0)}{3}-\lambda>0, \\[12pt]
\displaystyle
0,\quad \text{otherwise}. \end{array}\right.
$$

$\textbf{Step}\ 1$: We assume that 
\begin{equation}\label{ghu}
\begin{split}
&\int_{\mathbb{V} \cap |u|>0}  |u|^r \Big| \nabla \Big(\frac{u}{|u|}\Big)\Big|^2\text{d}x> \phi(\epsilon_0,\epsilon_1,r)\int_{\mathbb{V} \cap |u|>0}  |u|^{r-2} \big| \nabla  |u|\big|^2\text{d}x.
\end{split}
\end{equation}
 A direct calculation gives for $|u|>0$ that (\ref{gpkk}) holds.
By (\ref{lz1}) and  the Cauchy's inequality, we have

\begin{equation}\label{lz3}
\begin{split}
& \frac{d}{dt}\int_{\mathbb{V}} \rho |u|^r\text{d}x+\int_{\mathbb{V} \cap |u|>0} H_r\text{d}x\\
=&-r(r-2)(\mu+\lambda)\int_{\mathbb{V} \cap |u|>0}  \text{div}u |u|^{\frac{r-2}{2}} |u|^{\frac{r-4}{2}}u\cdot \nabla |u|\text{d}x\\
&+\int_{\mathbb{V}} r P_m\text{div }(|u|^{r-2}u)\text{d}x-\frac{1}{c}\int_{\mathbb{V} }\int_0^\infty \int_{S^2}r|u|^{r-2}A_ru\cdot \Omega \text{d}\Omega \text{d}v\text{d}x\\
\leq & r(\mu+\lambda)\int_{\mathbb{V} \cap |u|>0}   |u|^{r-2} |\text{div}u|^2\text{d}x
+\frac{r(r-2)^2(\mu+\lambda)}{4}\int_{\mathbb{V} \cap |u|>0}   |u|^{r-2} |\nabla |u||^2\text{d}x\\
&+\int_{\mathbb{V}} r P_m\text{div }(|u|^{r-2}u)\text{d}x-\frac{1}{c}\int_{\mathbb{V} }\int_0^\infty \int_{S^2}r|u|^{r-2}A_ru\cdot \Omega \text{d}\Omega \text{d}v\text{d}x.
\end{split}
\end{equation}
Via Holder's inequaity, Gagliardo-Nirenberg inequality and Young's inequality, we have
\begin{equation}\label{zhu2s}
\begin{split}
M_1=& r\int_{\mathbb{V}} P_m|u|^{r-2}|\nabla u|  \text{d}x
\leq C \Big( \int_{\mathbb{V} }|u|^{r-2}|\nabla u|^2\text{d}x\Big)^{\frac{1}{2}}\Big(\int_{\mathbb{V} }|u|^{r-2}P_m\text{d}x\Big)^{\frac{1}{2}}\\
\leq&C\Big( \int_{\mathbb{V} }|u|^{r-2}|\nabla u|^2\text{d}x\Big)^{\frac{1}{2}}||u|^{\frac{r}{2}}|^{1-\frac{2}{r}}_6|P_m|^{\frac{1}{2}}_{\frac{6r}{2r+2}}\\
\leq &
\frac{1}{4}\mu r \epsilon_0 \int_{\mathbb{V} }|u|^{r-2}|\nabla u|^2\text{d}x+C(\mu, r,\epsilon_0),\\
M_2=&-\frac{1}{c}\int_{\mathbb{V} }\int_0^\infty \int_{S^2}r|u|^{r-2}A_ru\cdot \Omega \text{d}\Omega \text{d}v\text{d}x\\
\leq & C||u|^{\frac{r}{2}}|^{2-\frac{2}{r}}_6\int_0^\infty \int_{S^2}|A_r|_{\frac{3r}{2r+1}} \text{d}\Omega \text{d}v
\leq  \frac{1}{4}\mu r \epsilon_0 \int_{\mathbb{V} }|u|^{r-2}|\nabla u|^2\text{d}x+C(\mu, r,\epsilon_0),
\end{split}
\end{equation}
where $\epsilon_0\in (0,\frac{1}{4})$  is independent of $r$.
Then combining (\ref{ghu})-(\ref{zhu2s}), we quickly have
\begin{equation}\label{lz4}
\begin{split}
& \frac{d}{dt}\int_{\mathbb{V}} \rho |u|^r\text{d}x+rf(\epsilon_0,\epsilon_1,\epsilon_2,r)\int_{\mathbb{V} \cap |u|>0} |u|^{r-2}|\nabla |u||^2\text{d}x\\
&+\int_{\mathbb{R}^3 \cap \{|u|>0\}}  \mu r(1-\epsilon_0)\epsilon_2|u|^{r}\Big| \nabla \Big(\frac{u}{|u|}\Big)\Big|^2\text{d}x
\leq C(\mu, r,\epsilon_0),
\end{split}
\end{equation}
where
\begin{equation}\label{xuchen}
\begin{split}
f(\epsilon_0,\epsilon_1,\epsilon_2, r)=\mu (1-\epsilon_0)(1-\epsilon_2)\phi(\epsilon_0,\epsilon_1,r)+\mu(r-1-\epsilon_0)-\frac{(r-2)^2(\mu+\lambda)}{4}.
\end{split}
\end{equation}

\textbf{Subcase} 1: If $4\in \Big\{ r\Big| \frac{r^2(\mu+\lambda)}{4(r-1)}-\frac{(4-\epsilon_0)\mu}{3}-\lambda>0\Big\}$, i,e, $\lambda+\epsilon_0 \mu>0$, it is  easy to get 
$$[4,+\infty)\in \Big\{ r\big| \frac{r^2(\mu+\lambda)}{4(r-1)}-\frac{(4-\epsilon_0)\mu}{3}-\lambda>0\Big\}.$$ Therefore, we have
\begin{equation}\label{lzyue}
\begin{split}
\phi(\epsilon_0, \epsilon_1, r)=
\frac{\mu \epsilon_1(r-1)}{3\big(-\frac{(4-\epsilon_0)\mu}{3}-\lambda+\frac{r^2(\lambda+\mu)}{4(r-1)}\big)},\quad \text{for}\quad  r\in [4,\infty).
\end{split}
\end{equation}
Substituting (\ref{lzyue}) into (\ref{xuchen}), for $r\in [4,\infty)$, we have
\begin{equation}\label{xuchendd}
\begin{split}
&f(\epsilon_0,\epsilon_1,\epsilon_2, r)\\
=&\frac{\mu^2 \epsilon_1(1-\epsilon_0)(1-\epsilon_2)(r-1)}{3\big(-\frac{(4-\epsilon_0)\mu}{3}-\lambda+\frac{r^2(\lambda+\mu)}{4(r-1)}\big)}+\mu(r-1-\epsilon_0)-\frac{(r-2)^2(\mu+\lambda)}{4}.
\end{split}
\end{equation}
For $(\epsilon_1,\epsilon_2,r)=(1,0,4)$, we have
\begin{equation}\label{dulan}
\begin{split}
f(\epsilon_0,1,0,4)=&\frac{3(1-\epsilon_0)\mu^2}{\lambda+\epsilon_0\mu}+2\mu-\lambda-\epsilon_0 \mu
=-C_1(\lambda-a_1 \mu)(\lambda-a_2\mu).
\end{split}
\end{equation}
Then according to $\lambda+\epsilon_0 \mu>0$, we have $C_1=\frac{1}{\lambda+\epsilon_0\mu}>0$ and 
\begin{equation}\label{dulann}
\begin{split}
a_1(\epsilon_0)=&1-\epsilon_0+\sqrt{4-3\epsilon_0},\quad a_2(\epsilon_0)=-1-\epsilon_0-\sqrt{4-3\epsilon_0}.
\end{split}
\end{equation}
So if we want to make sure that   $f(\epsilon_0,1,0,4)>0$, we have to asuume that 
\begin{equation}\label{dulann1}
\begin{split}
-\epsilon_0 \mu<\lambda< a_1(\epsilon_0) \mu.
\end{split}
\end{equation}

Due to $a_1(0)=3$ and $a'_1(\epsilon_0)<0$ for $\epsilon_0\in [0,1/4]$, then 
we can  choose  $\epsilon_0 \in (0,1/4)$ such that $a_1(\epsilon_0)\leq 3-\alpha_{\lambda\mu}$. 

Since $f(\epsilon_0,\epsilon_1,\epsilon_2, 4)$ is continuous w.r.t.  $(\epsilon_1,\epsilon_2)$ over $[0,1]\times [0,1]$, there exists $(\epsilon_1,\epsilon_2)\in (0,1) \times (0,1)$  such that 
$$
f(\epsilon_0,\epsilon_1,\epsilon_2,  4)>0,
$$
which, together with (\ref{lz4}), implies that 
\begin{equation}\label{lz411}
\begin{split}
& \frac{d}{dt}\int_{\mathbb{R}^3} \rho |u|^4\text{d}x+C\int_{\mathbb{R}^3 \cap \{|u|>0\}}|u|^{2}|\nabla u|^2\text{d}x
\leq C.
\end{split}
\end{equation}

\textbf{Subcase} 2: If $4\notin \{ r\big| \frac{r^2(\mu+\lambda)}{4(r-1)}-\frac{(4-\epsilon_0)\mu}{3}-\lambda>0\}$, i.e., $\lambda <-\epsilon_0 \mu$. In this case, for $r=4$,  it is easy to get
\begin{equation}\label{peng}
\begin{split}
&r\Big[\mu (1-\epsilon_0)(1-\epsilon_2)\phi(\epsilon_0,\epsilon_1,r)+\mu(r-1-\epsilon_0)-\frac{(r-2)^2(\mu+\lambda)}{4}\Big]\\
>&4\Big(\frac{11}{4}\mu-(\mu+\lambda)\Big)=4\Big(\frac{7\mu}{4}-\lambda\Big)
\geq 4\Big(\frac{7\mu}{4}+\epsilon_0\mu\Big)>7\mu,
\end{split}
\end{equation}
which, together with (\ref{lz4})-(\ref{xuchen}),  implies that
\begin{equation}\label{lz422}
\begin{split}& \frac{d}{dt}\int_{\mathbb{R}^3} \rho |u|^4\text{d}x+C\int_{\mathbb{R}^3 \cap \{|u|>0\}}|u|^{2}|\nabla u|^2\text{d}x
\leq C.
\end{split}
\end{equation}

$\textbf{Step}$ 2 : we assume that
\begin{equation}\label{ghu11}
\begin{split}
&\int_{\mathbb{R}^3 \cap |u|>0}  |u|^r \Big| \nabla \Big(\frac{u}{|u|}\Big)\Big|^2\text{d}x\leq  \phi(\epsilon_0,\epsilon_1,r)\int_{\mathbb{R}^3 \cap |u|>0}  |u|^{r-2} \big| \nabla  |u|\big|^2\text{d}x.
\end{split}
\end{equation}
A direct calculation gives for $|u|>0$,
\begin{equation}\label{ghu22}
\begin{split}
\text{div}u=|u|\text{div}\Big(\frac{u}{|u|}\Big)+\frac{u\cdot \nabla |u|}{|u|}.
\end{split}
\end{equation}
Then combining (\ref{ghu22}) and (\ref{lz3})-(\ref{zhu2s}), we quickly have
\begin{equation}\label{lz77}
\begin{split}
& \frac{d}{dt}\int_{\mathbb{R}^3} \rho |u|^r\text{d}x+\int_{\mathbb{R}^3 \cap \{|u|>0\}}\mu r(1-\epsilon_0)|u|^{r-2}|\nabla u|^2\text{d}x\\
&+\int_{\mathbb{R}^3 \cap \{|u|>0\}}r(\lambda+\mu) |u|^{r-2}|\text{div}u|^2\text{d}x+\int_{\mathbb{R}^3 \cap \{|u|>0\}}\mu r(r-2)|u|^{r-2}\big| \nabla  |u| \big|^2\text{d}x\\
=&-r(r-2)(\mu+\lambda)\int_{\mathbb{R}^3 \cap \{|u|>0\}} \Big( |u|^{r-2} u\cdot \nabla |u| \text{div}\Big(\frac{u}{|u|}\Big)+ |u|^{r-4} |u\cdot \nabla |u| |^2\Big)\text{d}x.
\end{split}
\end{equation}
This gives
\begin{equation}\label{lz88}
\begin{split}
& \frac{d}{dt}\int_{\mathbb{R}^3} \rho |u|^r\text{d}x+\int_{\mathbb{R}^3 \cap \{|u|>0\}}r|u|^{r-4}G\text{d}x\leq C(\mu,r,\epsilon_0),
\end{split}
\end{equation}
where
\begin{equation}\label{wang1}
\begin{split}
G=&\mu (1-\epsilon_0)|u|^{2} |\nabla u|^2+(\mu+\lambda)|u|^{2}|\text{div}u|^2+\mu(r-2)|u|^{2}\big| \nabla  |u| \big|^2\\
&+(r-2)(\mu+\lambda)|u|^{2} u\cdot \nabla |u|\text{div}\Big(\frac{u}{|u|}\Big)+(r-2)(\mu+\lambda)|u \cdot \nabla |u||^2.
\end{split}
\end{equation}
Now we consider how to make sure that $G\geq 0$.
\begin{equation}\label{wang2mm}
\begin{split}
G=&\mu (1-\epsilon_0)|u|^{2} \Big( |u|^2\Big| \nabla \Big(\frac{u}{|u|}\Big)\Big|^2+\big| \nabla  |u|\big|^2\Big)\\
&+(\mu+\lambda)|u|^{2}|\Big(|u|\text{div}\Big(\frac{u}{|u|}\Big)+\frac{u\cdot \nabla |u|}{|u|}\Big)^2+\mu(r-2)|u|^{2}\big| \nabla  |u| \big|^2\\
&+(r-2)(\mu+\lambda)|u|^{2} u\cdot \nabla |u|\text{div}\Big(\frac{u}{|u|}\Big)+(r-2)(\mu+\lambda)|u \cdot \nabla |u||^2\\
=&\mu(1-\epsilon_0) |u|^{4} \Big| \nabla \Big(\frac{u}{|u|}\Big)\Big|^2+\mu(r-1-\epsilon_0)|u|^{2}\big| \nabla  |u| \big|^2\\
&+(r-1)(\mu+\lambda)\Big(u \cdot \nabla |u|+\frac{r}{2(r-1)}|u|^{2}\Big(\text{div}\frac{u}{|u|}\Big)\Big)^2\\
&+(\mu+\lambda)|u|^{4}\Big(\text{div}\frac{u}{|u|}\Big)^2-\frac{r^2(\mu+\lambda)}{4(r-1)}|u|^{4}\Big(\text{div}\Big(\frac{u}{|u|}\Big)\Big)^2,
\end{split}
\end{equation}
which, combining with the fact
$
\Big|\text{div}\Big(\frac{u}{|u|}\Big)\Big|^2\leq 3\Big|\nabla \Big(\frac{u}{|u|}\Big)\Big|^2
$,
implies that
\begin{equation}\label{wang3}
\begin{split}
G
\geq &\mu(r-1-\epsilon_0)|u|^{2}\big| \nabla  |u| \big|^2+\Big(\frac{(4-\epsilon_0)\mu}{3}+\lambda-\frac{r^2(\mu+\lambda)}{4(r-1)}\Big)|u|^{4}\Big(\text{div}\Big(\frac{u}{|u|}\Big)\Big)^2.
\end{split}
\end{equation}
Thus
\begin{equation}\label{wang4}
\begin{split}
&\int_{\mathbb{R}^3 \cap \{|u|>0\}} r|u|^{r-4} G\text{d}x\geq \mu r(r-1-\epsilon_0)\int_{\mathbb{R}^3 \cap \{|u|>0\}} |u|^{r-2} \big| \nabla  |u| \big|^2\text{d}x\\
&\qquad + r \Big(\frac{(4-\epsilon_0)\mu}{3}+\lambda-\frac{r^2(\mu+\lambda)}{4(r-1)}\Big) \int_{\mathbb{R}^3 \cap \{|u|>0\}} |u|^{r} \Big(\text{div}\Big(\frac{u}{|u|}\Big)\Big)^2\text{d}x\\
&\qquad \geq  g(\epsilon_0,\epsilon_1,r)\int_{\mathbb{R}^3 \cap \{|u|>0\}}  |u|^{r-2} \big| \nabla  |u| \big|^2\text{d}x,
\end{split}
\end{equation}
where
\begin{equation}\label{xuchendd}
\begin{split}
g(\epsilon_0,\epsilon_1,r)=\Big[3r \Big(\frac{(4-\epsilon_0)\mu}{3}+\lambda-\frac{r^2(\mu+\lambda)}{4(r-1)}\Big)\phi(\epsilon_0,\epsilon_1,r)+\mu r(r-1-\epsilon_0)\Big)\Big].
\end{split}
\end{equation}
Here we need that $\epsilon_0$ is sufficiently small such that $\epsilon_0<(r-1)(1-\epsilon_1)$.
Then combining (\ref{gpkk}),  (\ref{lz88}) and  (\ref{wang4})-(\ref{xuchendd}), when $r=4$, we quickly have
\begin{equation}\label{lz4mm}
\begin{split}
& \frac{d}{dt}\int_{\mathbb{R}^3} \rho |u|^4\text{d}x+C\int_{\mathbb{R}^3 \cap \{|u|>0\}}|u|^{2}|\nabla u|^2\text{d}x \leq C.
\end{split}
\end{equation}

So combining (\ref{lz411})-(\ref{lz422}) and (\ref{lz4mm}) for $\textbf{Step}$:1 and $\textbf{Step}$:2, we conclude that if $\lambda <(3-\alpha_{\lambda\mu})\mu$, there exits some constants $C>0$ such that (\ref{abs3}) holds.

\end{proof}

The next lemma will give a key estimate on $\nabla u$.
  \begin{lemma}\label{abs4}
\begin{equation*}
\begin{split}
&|\nabla u(t)|_{ 2}++\int_0^T |\sqrt{\rho} u_t|^2_2\text{d}t
 \leq C,\quad 0\leq t< T.
\end{split}
\end{equation*}
 \end{lemma}

\begin{proof}
Via the momentum equations $(\ref{eq:1.2})_2$, we have
\begin{equation*}
\begin{split}
\triangle G=\text{div}\Big(\rho \dot{u}+\frac{1}{c}\int_0^\infty \int_{S^2}A_r\Omega \text{d}\Omega \text{d}v\Big),\ \mu \triangle \omega=\nabla \times \Big(\rho \dot{u}+\frac{1}{c}\int_0^\infty \int_{S^2}A_r\Omega \text{d}\Omega \text{d}v\Big),
\end{split}
\end{equation*}
where 
\begin{equation*}
\dot{f}=f_t+u\cdot \nabla f=f_t+\text{div}(fu)-f\text{div}u, \ G=(2\mu+\lambda)\text{div}u-P_m, \ \text{and}\ \omega=\nabla \times u,
\end{equation*}
are the material derivative of $f$, the effective viscous flux, and the vorticity, respectively.
It follows from Lemmas \ref{tvd1} and \ref{s2} that
\begin{equation}\label{yue11}
\begin{split}
|\nabla G|_2+|\nabla \omega|_2\leq & C (|\rho u_t|_2+|\rho u \cdot \nabla u|_2+1)\leq C (|\sqrt{\rho} u_t|_2+|\sqrt{\rho} |u| |\nabla u| |_2+1).
\end{split}
\end{equation}

Multiplying   $(\ref{eq:1.2})_2$ by $u_t$ and integrating  over $\Omega$ gives 
\begin{equation}\label{yue1}
\begin{split}
&\frac{1}{2}\frac{d}{dt}\int_{\mathbb{V}} (\mu |\nabla u|^2+(\lambda+\mu)(\text{div}u)^2) \text{d}x+\int_{\mathbb{V} }\rho |u_t|^2\text{d}x\\
=&\int_{\mathbb{V}} \Big(P_m\text{div}u_t-\rho u \cdot \nabla u \cdot u_t-\frac{1}{c}\int_0^\infty \int_{S^2}A_r\Omega\cdot u_t \text{d}\Omega \text{d}v\Big) \text{d}x=A+B+C.
\end{split}
\end{equation}
For the first term on the right-hand side of (\ref{yue1}), one has
\begin{equation}\label{yue3}
\begin{split}
A=&\int_{\mathbb{V}} P_m \text{div}u_t \text{d}x=\frac{d}{dt} \int_{\mathbb{V}}  P_m \text{div}u \text{d}x-\int_{\mathbb{V}}  (P_m)_t \text{div}u \text{d}x\\
=&\frac{d}{dt} \int_{\mathbb{V}}  P_m \text{div}u \text{d}x-\frac{1}{2\mu+\lambda}\int_{\mathbb{V}}  (P_m)_t G \text{d}x-\frac{1}{2(2\mu+\lambda)}\frac{d}{dt} \int_{\mathbb{V}}  P^2_m  \text{d}x\\
=&A_1+A_2+A_3.
\end{split}
\end{equation}
We first consider the second term on the right-hand side of (\ref{yue3}) that
\begin{equation}\label{yuej4}
\begin{split}
A_2=&-\frac{1}{2\mu+\lambda}\int_{\mathbb{V}}  (P_m)_t G \text{d}x=-\frac{\gamma-1}{2\mu+\lambda}\int_{\mathbb{V}}  \Big( (\rho E_m)_t-\big(\frac{1}{2}\rho |u|^2\big)_t\Big) G \text{d}x\\
=&A_{21}+A_{22},\\
\end{split}
\end{equation}
\begin{equation} \label{yuej4qq}\begin{split}
A_{21}=&-\frac{\gamma-1}{2\mu+\lambda}\int_{\mathbb{V}}\Big(  \frac{1}{2}\rho |u|^2 u\cdot \nabla G + P_mu\cdot \nabla G + \frac{P_m}{\gamma-1} u\cdot \nabla G\Big) \text{d}x\\
&-\frac{\gamma-1}{2\mu+\lambda}\int_{\mathbb{V}}  (N_r G-\kappa \nabla \theta \cdot \nabla G) \text{d}x+\frac{\gamma-1}{2\mu+\lambda}\int_{\mathbb{V}} (u\mathbb{T})\cdot \nabla G \text{d}x\\
\leq & -\frac{\gamma-1}{2\mu+\lambda}\int_{\mathbb{V}}  \frac{1}{2}\rho |u|^2 u\cdot \nabla G \text{d}x+C|G|_2|N_r|_2\\
&+C|\nabla G|_2(|uP_m|_2+||u||\nabla u||_2+|\nabla \theta|_2).
\end{split}
\end{equation}
For the second term on the right-hand side of (\ref{yue1}),  Cauchy's inequality yields
\begin{equation}\label{yue4}
\begin{split}
B=-\int_{\mathbb{V} }\rho u \cdot \nabla u \cdot u_t \text{d}x  \leq \frac{1}{6}| \sqrt{\rho} u_t|^2_2+\int_{\mathbb{V}} \rho |u \cdot \nabla u|^2 \text{d}x.
\end{split}
\end{equation}
Then according to (\ref{yue11}), we obtain that 
\begin{equation}\label{yue5}
\begin{split}
|\sqrt{\rho}|u| |\nabla u| |_2\leq& C |\rho^{\frac{1}{4}}u |_4 |\nabla u|_{4}
\leq  C\big( |G|_{4}+|\omega|_{4}+1\big)\\
\leq & C( |G|^{\frac{1}{4}}_{2} | G|^{\frac{3}{4}}_{6}+|\omega|^{\frac{1}{4}}_{2} | \omega|^{\frac{3}{4}}_{6}+1)\\
\leq & \epsilon ( |\nabla G|_{2}+ |\nabla \omega|_{2})+C(\epsilon) (  G|_{2}+ |\omega|_{2})+C\\
\leq &C \epsilon(|\sqrt{\rho} u_t|_2+|\sqrt{\rho} |u| |\nabla u||_2)+C(\epsilon)(|\nabla u|_2+1),
\end{split}
\end{equation}
which immediately means that 
\begin{equation}\label{yue6}
\begin{split}
|\sqrt{\rho}u \cdot \nabla u|_2
\leq &C \epsilon|\sqrt{\rho} u_t|_2+C(\epsilon)(|\nabla u|_2+1).
\end{split}
\end{equation}
Then subsitituing (\ref{yue6}) into (\ref{yue11}), we have
\begin{equation}\label{yuekk}
\begin{split}
|\nabla G|_2+|\nabla \omega|_2\leq  C (|\sqrt{\rho} u_t|_2+|\nabla u |_2+1),
\end{split}
\end{equation}
which, together with (\ref{yuej4}), implies that
\begin{equation}\label{yuekk11}
\begin{split}
A_{21} \leq & -\frac{1}{4\mu+2\lambda}\int_{\mathbb{V}}  \rho |u|^2 u\cdot \nabla G \text{d}x+C(|\nabla u|^2_2+|\nabla \theta|^2_2+||u||\nabla u||^2_2+1)+\epsilon |\sqrt{\rho} u_t|^2_2,
\end{split}
\end{equation}
where we  used the fact that $|N_r|_2\leq C(1+|\nabla u|_2)$ via Lemma \ref{s2}.
Next we consider $A_{22}$,
\begin{equation}\label{yuej88}
\begin{split}
A_{22}=&\frac{\gamma-1}{2\mu+\lambda}\int_{\mathbb{V}}  \frac{1}{2}\rho_t |u|^2  G \text{d}x+\frac{\gamma-1}{2\mu+\lambda}\int_{\mathbb{V}} \rho u\cdot u_t G \text{d}x\\
\leq & -\frac{\gamma-1}{2\mu+\lambda}\int_{\mathbb{V}}  \frac{1}{2}\text{div}(\rho u) |u|^2  G \text{d}x+\epsilon |\sqrt{\rho} u_t|^2_2+C(\epsilon)\int_{\mathbb{V}} \rho |u|^2|G|^2 \text{d}x\\
\leq & \frac{\gamma-1}{2\mu+\lambda}\int_{\mathbb{V}}  \rho u\cdot \nabla u\cdot u  G \text{d}x+\frac{\gamma-1}{2\mu+\lambda}\int_{\mathbb{V}} \frac{1}{2} \rho |u|^2 u\cdot \nabla  G \text{d}x\\
&+\epsilon |\sqrt{\rho} u_t|^2_2+C\int_{\mathbb{V}} \rho |u|^2|\nabla u|^2 \text{d}x+C\\
\leq & \epsilon |\sqrt{\rho} u_t|_2+C|\sqrt{\rho} |u| |\nabla u||_2
+\frac{\gamma-1}{2\mu+\lambda}\int_{\mathbb{V}} \frac{1}{2} \rho |u|^2 u\cdot \nabla  G \text{d}x+C| |u| |\nabla u||^2_2+C,
\end{split}
\end{equation}
which, together with (\ref{yue6}), implies that 
\begin{equation}\label{yuej8ll}
\begin{split}
A_{22}
\leq & \epsilon |\sqrt{\rho} u_t|_2
+\frac{\gamma-1}{2\mu+\lambda}\int_{\mathbb{V}} \frac{1}{2} \rho |u|^2 u\cdot \nabla  G \text{d}x+C|\nabla u|^2_2+C| |u| |\nabla u||^2_2+C.
\end{split}
\end{equation}
Then combining (\ref{yuekk11}) and (\ref{yuej8ll}), we deduce that 
\begin{equation}\label{yllp1}
\begin{split}
A_2
\leq & \epsilon |\sqrt{\rho} u_t|_2
+C|\nabla u|^2_2+C|\nabla \theta|^2_2+C| |u| |\nabla u||^2_2+C.
\end{split}
\end{equation}
Next we consider the term $B$ and $C$. From (\ref{yue4}) and (\ref{yue6}), we have
\begin{equation}\label{tie1}
\begin{split}
B=&\int_{\mathbb{V}} -\rho u \cdot \nabla u \cdot u_t \text{d}x\leq C|\sqrt{\rho} |u| |\nabla u||_2+\epsilon |\sqrt{\rho} u_t|^2_2
\leq \epsilon |\sqrt{\rho} u_t|^2_2+C|\nabla u|^2_2+C,\\
\end{split}
\end{equation}
\begin{equation}\label{shanna1}
\begin{split}
C=&
-\frac{1}{c}\int_{\mathbb{V}} \int_0^\infty \int_{S^2}A_r\Omega\cdot u_t \text{d}\Omega \text{d}v \text{d}x\\
\leq & -\frac{1}{c}\frac{d}{dt}\int_{\mathbb{V}} \int_0^\infty \int_{S^2}S\Omega\cdot u\text{d}\Omega \text{d}v \text{d}x+\frac{1}{c}\int_{\mathbb{V}} \int_0^\infty \int_{S^2}S_t\Omega\cdot u\text{d}\Omega \text{d}v \text{d}x\\
&+C|\rho|^{\frac{1}{2}}_\infty|\sqrt{\rho} u_t|_2|I|_{L^2(\mathbb{R}^+\times S^2;L^2(\mathbb{V}))}\big(|\sigma|_{L^2(\mathbb{R}^+\times S^2;L^\infty(\mathbb{V}))}+\alpha+\alpha\big)\\
\leq & -\frac{1}{c}\frac{d}{dt}\int_{\mathbb{V}} \int_0^\infty \int_{S^2}S\Omega\cdot u\text{d}\Omega \text{d}v\text{d}x+\epsilon |\sqrt{\rho} u_t|^2_2+C|\nabla u|^2_2+C.
\end{split}
\end{equation}
Then from (\ref{yue1})-(\ref{yue3}) and  (\ref{yllp1})-(\ref{shanna1}), letting $\epsilon$ be sufficiently small, we have
\begin{equation}\label{yueop1}
\begin{split}
&\frac{1}{2}\frac{d}{dt}\int_{\mathbb{V}} (\mu |\nabla u|^2+(\lambda+\mu)(\text{div}u)^2) \text{d}x+\int_{\mathbb{V} }\rho |u_t|^2\text{d}x\\
\leq & C(|\nabla u|^2_2+|\nabla \theta|^2_2)+C| |u| |\nabla u||^2_2+C.
\end{split}
\end{equation}
From Gronwall's inequality and Lemma \ref{abs3}, we obtain the desired conclusions.

\end{proof}

\subsection{The lower order estimate for $|\theta|_{L^\infty([0,\overline{T}]; D^1(\mathbb{V})}$}
\begin{lemma}\label{ablem:4-1} 
\begin{equation*}
\begin{split}
| \theta(t)|_{D^1}+ \int_{0}^{T}(|\theta|^2_{D^2}+ |\dot{u}|^2_{D^1}+|\sqrt{\rho} \dot{\theta}|^2_2+|\text{div}u|_\infty)\text{d}t\leq& C, \quad 0\leq t \leq T,\\
|u(t)|_\infty+|\nabla u(t)|_6+|(G,\omega)(t)|_{D^1}+|(\sqrt{\rho} \dot{u},\sqrt{\rho}u_t)(t)|_{2} \leq& C, \quad 0\leq t \leq T.
\end{split}
\end{equation*}
 \end{lemma}

 \begin{proof}
\underline{Step 1}. Applying $\dot{u}[\partial / \partial t+\text{div}(u\cdot)]$ to $(\ref{eq:1.2})_3$  and integrating by parts give
\begin{equation}\label{bzhen4}
\begin{split}
&\frac{1}{2}\frac{d}{dt}\int_{\mathbb{V} }\rho |\dot{u}|^2 \text{d}x
=  -\int_{\mathbb{V} }   \dot{u}\cdot \big(\nabla (P_m)_t+\text{div}(\nabla P_m\otimes u)\big)\text{d}x\\
&-\int_{\mathbb{V} } \dot{u}\cdot \big(\triangle u_t+\text{div}(\triangle u\otimes u)\big)  \text{d}x
+ (\lambda+\mu)\int_{\mathbb{V}} \dot{u}\cdot  \big(\nabla \text{div}u_t+\text{div}(\nabla  \text{div}u\otimes u)\big) \text{d}x\\
&
-\frac{1}{c}\int_{\mathbb{V}} \int_0^\infty \int_{S^2}\dot{u}\cdot \big((A_r)_t \Omega+\text{div}(A_r \Omega\otimes u)\big) \text{d}\Omega \text{d}v \text{d}x
\equiv:\sum_{i=3}^{6}M_i.
\end{split}
\end{equation}

According to the continuity equation $(\ref{eq:1.2})_2$,  Lemmas \ref{s2}-\ref{abs4}, H\"older's inequality, Gagliardo-Nirenberg inequality and Young's inequality,  we deduce that
\begin{equation}\label{zhou6}
\begin{split}
M_3
=& \int_{\mathbb{V} }\big(R\text{div}\dot{u} (\rho \theta)_t-P_m(\nabla u)^\top:\nabla \dot{u}-R\rho \theta u \cdot \nabla \text{div}\dot{u}\big)\text{d}x\\
=& \int_{\mathbb{V} }\big(R[(\rho \theta)_t+\text{div}(\rho \theta u)]\text{div}\dot{u}-P_m(\nabla u)^\top:\nabla \dot{u} \big)\text{d}x\\
\leq&  C|\rho|^{\frac{1}{2}}_\infty|\sqrt{\rho}\dot{\theta}|_2|\text{div}\dot{u}|_2+C|P_m|_4|\nabla u|_4|\nabla \dot{u}|_2
\leq  \frac{\mu}{20} |\nabla \dot{u}|^2_2+C|\sqrt{\rho}\dot{\theta}|^2_2+C|\nabla u|^2_4,\\
\end{split}
\end{equation}
\begin{equation}\label{zhou6pl}
\begin{split}
M_4
=& -\int_{\mathbb{V} }\mu\big(\partial_i\dot{u}^j\partial_i u^j_t+\triangle u^j u\cdot \nabla \dot{u}^j\big)\text{d}x\\
=& -\int_{\mathbb{V} }\mu\big(|\nabla \dot{u}|^2- \partial_i \dot{u}^j \partial_ku^k \partial_i u^j- \partial_i \dot{u}^j \partial_i u^k\partial_k u^j-\partial_iu^j \partial_i u^k \partial_k \dot{u}^j\big)\text{d}x\\
\leq &-\frac{\mu}{2}|\nabla \dot{u}|^2_{2}+C|\nabla u|^4_4,
\end{split}
\end{equation}
and similarly, we have
\begin{equation}\label{ppy6}
\begin{split}
M_5=&  (\lambda+\mu)\int_{\mathbb{V}} \big(\dot{u}\cdot  \big(\nabla \text{div}u_t+\text{div}(\nabla  \text{div}u\otimes u\big)\big)\text{d}x
\leq  -\frac{\mu+\lambda}{2}|\nabla \dot{u}|^2_2 +C |\nabla u|^4_4.
\end{split}
\end{equation}
Next we consider the radiation term $M_6$:
\begin{equation}\label{ppy6}
\begin{split}
M_6=& -\frac{1}{c}\int_{\mathbb{V}} \int_0^\infty \int_{S^2}\dot{u}\cdot \big((A_r)_t \Omega+\text{div}(A_r \Omega\otimes u)\big) \text{d}\Omega \text{d}v \text{d}x\equiv:\sum_{i=1}^{5}-\frac{1}{c}M_{6i}.
\end{split}
\end{equation}
Then via (\ref{kll}) and $\sigma_a=\sigma(v,\Omega,\theta) \rho$, we have
\begin{equation}\label{ppyy1}
\begin{split}
M_{61}=&\int_0^\infty \int_{S^2} \int_{\mathbb{V}} \dot{u}\cdot \Omega S_t \text{d}x \text{d}\Omega \text{d}v
\leq |\dot{u}|_{2}\int_0^\infty \int_{S^2} |S_t |_{2}\text{d}\Omega \text{d}v\leq \frac{\mu}{20}|\nabla \dot{u}|^2_{2}+C,\\
M_{62}
=&\int_0^\infty \int_{S^2} \int_{\mathbb{V}} \dot{u}\cdot \Omega (\sigma_\theta \theta_t \rho I-\sigma \text{div}(\rho u) I+\sigma \rho I_t) \text{d}x \text{d}\Omega \text{d}v\\
\leq & CC_r\big( |\sqrt{\rho}\dot{\theta}|_{2}|\sqrt{\rho}\dot{u}|_{2}+|\nabla \dot{u}|_2|u|_6|\rho|_3+(|\nabla \theta|_2+1)|\dot{u}|_6|u|_6|\rho|_6+|\rho|^{\frac{1}{2}}_\infty|\sqrt{\rho}\dot{u}|_{2}\big)\\
\leq &C(|\sqrt{\rho}\dot{u}|^2_{2}+|\sqrt{\rho}\dot{\theta}|^2_{2})+\frac{\mu}{20}|\nabla \dot{u}|^2_2+C|\nabla \theta|^2_2+C,\\
M_{63}
=&\int_0^\infty \int_{S^2} \int_{\mathbb{V}}  \int_0^\infty \int_{S^2}\frac{v}{v'}\dot{u}\cdot \Omega(\overline{\sigma}_s \rho I'_t+(\overline{\sigma}_s)_t \rho I'-\overline{\sigma}_s \text{div}(\rho u) I')\text{d}\Omega \text{d}v\text{d}x \text{d}\Omega \text{d}v\\
\leq& CD_r(|\rho|^{\frac{1}{2}}_\infty|\sqrt{\rho}\dot{u}|_{2}+ |\nabla \dot{u}|_2|\nabla u|_2|\rho|_3)\leq C|\sqrt{\rho}\dot{u}|^2_{2}+\frac{\mu}{20}|\nabla \dot{u}|^2_2+C,
\end{split}
\end{equation}
where we  used the fact  $(\sigma_a)_t =\sigma_\theta \theta_t \rho +\sigma \rho_t$ and 
\begin{equation*}
\begin{split}
\displaystyle
C_r=&\int_0^\infty \int_{S^2}( |\sigma_\theta|_{\infty}|I|_{\infty}+|\sigma|_{\infty}|I|_{\infty}+|\sigma|_{\infty}|I_t|_{2}+|\sigma|_{\infty}|\nabla I|_2)\text{d}\Omega \text{d}v\leq C,\\
\displaystyle
D_r=&\int_\mathbb{I}\frac{v}{v'}(|\overline{\sigma}_s|_\infty( | I'|_\infty+ | I'_t|_2+ |\nabla I'|_2)+|\nabla \overline{\sigma}_s|_\infty |I'|_6+|(\overline{\sigma}_s)_t|_\infty | I'|_2) \text{d}\mathbb{I}\leq C.
\end{split}
\end{equation*}
Similarly, we have
\begin{equation}\label{ppyy2}
\begin{split}
M_{64}=&-\frac{1}{c}\int_0^\infty \int_{S^2} \int_{\mathbb{V}} \Big( \int_0^\infty \int_{S^2}\dot{u}\cdot \Omega(\overline{\sigma}'_s \rho I_t+(\overline{\sigma}'_s)_t \rho I+\overline{\sigma}'_s \rho_t I)\text{d}\Omega' \text{d}v'\Big)\text{d}x \text{d}\Omega \text{d}v\\
\leq& CE_r(|\rho|^{\frac{1}{2}}_\infty|\sqrt{\rho}\dot{u}|_{2}+ |\nabla \dot{u}|_2|\nabla u|_2|\rho|_3)\leq C|\sqrt{\rho}\dot{u}|^2_{2}+\frac{\mu}{20}|\nabla \dot{u}|^2_2+C,\\
M_{65}\leq & C|\nabla \dot{u}|_2 |u|_6 \int_0^\infty \int_{S^2}|A_r|_3\text{d}\Omega \text{d}v\leq C|\nabla u|^2_2+\frac{\mu}{20}|\nabla \dot{u}|^2_2,
\end{split}
\end{equation}
where
\begin{equation*}
\begin{split}
E_r=&\int_\mathbb{I}\big(|\overline{\sigma}'_s|_\infty( | I|_\infty+| I_t|_2+|\nabla I|_2)+|\nabla \overline{\sigma}'_s|_\infty |I|_6 +|(\overline{\sigma}'_s)_t|_\infty | I|_2 \big) \text{d}\mathbb{I}\leq C.
\end{split}
\end{equation*}
Together with (\ref{bzhen4})-(\ref{ppyy2}), we deduce that
\begin{equation}\label{mou5}
\begin{split}
\frac{1}{2}\frac{d}{dt}  & \int_{\mathbb{V} }\rho |\dot{u}|^2 \text{d}x+|\dot{u}|^2_{D^1}
\leq  C|\nabla u|^4_4+C|\sqrt{\rho}\dot{u}|^2_{2}+C|\sqrt{\rho}\dot{\theta}|^2_{2}+C|\nabla \theta|^2_2+C.
\end{split}
\end{equation}

\underline{Step 2}.  Multiplying $(\ref{eq:1.2})_4$ by $\dot{\theta}$, and integrating over $\mathbb{V}$, we have
\begin{equation}\label{ppyue1jkk}
\begin{split}
&\frac{\kappa}{2c_v}\frac{d}{dt}\int_{\mathbb{V}} |\nabla \theta|^2\text{d}x+\int_{\mathbb{V} }\rho |\dot{\theta}|^2\text{d}x\\
=&\frac{1}{c_v}\int_{\mathbb{V}} \Big(-P_m\text{div}u \dot{\theta}+Q(u)\theta_t+Q(u)u\cdot \nabla \theta+\kappa \triangle \theta u \cdot \nabla \theta+N_r \dot{\theta}\Big) \text{d}x=\sum_{i=7}^{11} M_i.
\end{split}
\end{equation}
Then from H\"older's inequality and Young's inequality, we have
\begin{equation}\label{ppyue1}
\begin{split}
M_7\leq& C|\rho|^{\frac{1}{2}}_\infty|\theta|_\infty|\sqrt{\rho}\dot{\theta}|_2|\nabla u|_2\leq \frac{1}{20}|\sqrt{\rho}\dot{\theta}|^2_2+C,\\
M_8=&\frac{1}{c_v}\frac{d}{dt}\int_{\mathbb{V}} Q(u)\theta \text{d}x-\frac{4\mu}{c_v}\int_{\mathbb{V}} D(u):D(u_t) \theta \text{d}x-\frac{2\lambda}{c_v} \int_{\mathbb{V}} \text{div}u \text{div}u_t \theta \text{d}x\\
\leq &\frac{1}{c_v}\frac{d}{dt}\int_{\mathbb{V}} Q(u)\theta \text{d}x+C|\nabla u|_2|\nabla \dot{u}|_2\\
&+\frac{1}{c_v}\int_{\mathbb{V}} \Big(4\mu D(u):D(u\cdot \nabla u) \theta+2\lambda \text{div}u \text{div}(u\cdot \nabla u\big) \theta\Big) \text{d}x\\
\leq & \frac{1}{c_v}\frac{d}{dt}\int_{\mathbb{V}} Q(u)\theta \text{d}x+C|\nabla \dot{u}|_2+\frac{4\mu}{c_v}\int_{\mathbb{V}} \Big( D(u):(\nabla u \cdot \nabla u+(\nabla u \cdot \nabla u)^\top)\Big)\theta\text{d}x\\
&+\frac{4\mu}{c_v}\int_{\mathbb{V}}  D(u):u\cdot \nabla D(u)\theta\text{d}x+\frac{2\lambda}{c_v} \int_{\mathbb{V}}\Big(  (\nabla u)^\top:\nabla u  + u\cdot  \nabla \text{div}u  \Big)\text{div}u\theta \text{d}x\\
\leq & \frac{1}{c_v}\frac{d}{dt}\int_{\mathbb{V}} Q(u)\theta \text{d}x+C|\nabla \dot{u}|_2+C|\nabla u|^3_3+\frac{2\mu}{c_v}\int_{\mathbb{V}}  |D(u)|^2 \text{div}u\theta\text{d}x\\
&+\frac{2\mu}{c_v}\int_{\mathbb{V}}  |D(u)|^2  u\cdot \nabla  \theta \text{d}x+\lambda \int_{\mathbb{V}} (|\text{div}u|^3\theta+|\text{div}u|^2 u \cdot \nabla \theta) \text{d}x\\
\leq & \frac{1}{c_v}\frac{d}{dt}\int_{\mathbb{V}} Q(u)\theta \text{d}x+C|\nabla \dot{u}|_2+C|\nabla u|^3_3+C\int_{\mathbb{V}} |\nabla u|^2|u||\nabla \theta|\text{d}x.
\end{split}
\end{equation}
From Lemmas \ref{tvd1} and \ref{s2}, we quickly have 
 \begin{equation}\label{zhu54}\begin{split}
|\theta|_{D^2}\leq C(|\rho \dot{\theta}|_2+|\rho \theta \text{div}u|_2+|\nabla u|^2_4+1)\leq C|\sqrt{\rho} \dot{\theta}|_2+C|\nabla u|^2_4+C.
\end{split}
\end{equation}

Then via H\"older's inequality, Young's inequality and Gagliardo-Nirenberg inequality, 
 \begin{equation}\label{ppy10}
\begin{split}
\int_{\mathbb{V}} |\nabla u|^2|u||\nabla \theta|\text{d}x\leq& C|\nabla u|^2_4|u|_6|\nabla \theta|^{\frac{1}{2}}_2|\nabla \theta|^{\frac{1}{2}}_6
\leq  \frac{1}{20}|\sqrt{\rho} \dot{\theta}|^2_2+C|\nabla u|^4_4+C|\nabla \theta|^2_2+C,
\end{split}
\end{equation}
which, together with (\ref{ppyue1}), implies that 
\begin{equation}\label{ppyue10}
\begin{split}
M_{8}
\leq & \frac{1}{c_v}\frac{d}{dt}\int_{\mathbb{V}} Q(u)\theta \text{d}x+C|\nabla u|^4_4+C|\nabla \dot{u}|_2+C|\nabla \theta|^2_2+\frac{1}{20}|\sqrt{\rho} \dot{\theta}|^2_2+C.
\end{split}
\end{equation}
Similarly, we have
\begin{equation}\label{ppyue11}
\begin{split}
M_9
\leq &C\int_{\mathbb{V}} |\nabla u|^2 |u| |\nabla \theta|\text{d}x\leq \frac{1}{20}|\sqrt{\rho} \dot{\theta}|^2_2+C|\nabla u|^4_4+C|\nabla \theta|^2_2+C,\\
M_{10}
\leq & C |\triangle \theta|_2 |u|_6 |\nabla \theta|_3\leq C |\triangle \theta|_2 |\nabla u|_2 |\nabla \theta|^{\frac{1}{2}}_2\|\nabla \theta\|^{\frac{1}{2}}_1\\
\leq & C |\nabla \theta|^{\frac{1}{2}}_2\|\nabla \theta\|^{\frac{3}{2}}_1
\leq \frac{1}{20}|\sqrt{\rho} \dot{\theta}|^2_2+C|\nabla u|^4_4+C|\nabla \theta|^2_2+C.
\end{split}
\end{equation}
Next, we consider the radiation terms 
\begin{equation}\label{ppyue111}\begin{split}
M_{11}=&\frac{1}{c_v}\int_0^\infty \int_{S^2}  \int_{\mathbb{V}} \dot{\theta}N_r \text{d}x\text{d}\Omega \text{d}v 
\leq \frac{1}{c_v} \frac{d}{dt}\int_0^\infty \int_{S^2}  \int_{\mathbb{V}} \Big(1-\frac{u\cdot \Omega}{c}\Big)S \theta \text{d}x\text{d}\Omega \text{d}v\\
&+\int_0^\infty \int_{S^2}  \int_{\mathbb{V}} \frac{(\dot{u}-u\cdot \nabla u)\cdot \Omega}{cc_v}S \theta \text{d}x\text{d}\Omega \text{d}v \\
&+\frac{1}{c_v}\int_0^\infty \int_{S^2}  \int_{\mathbb{V}} \Big(1-\frac{u\cdot \Omega}{c}\Big)\big(-S_t \theta+S u\cdot \nabla \theta\big)\text{d}x\text{d}\Omega \text{d}v\\
& +C|\rho|^{\frac{1}{2}}_\infty |\sqrt{\rho}\dot{\theta}|_2\|I\|_{L^2(\mathbb{R}^+\times S^2; L^2(\mathbb{V}))}\big(\alpha+\alpha+\|\sigma\|_{L^2(\mathbb{R}^+\times S^2; L^2(\mathbb{V}))}\big)\\
\leq&\frac{1}{c_v} \frac{d}{dt}\int_0^\infty \int_{S^2}  \int_{\mathbb{V}} \Big(1-\frac{u\cdot \Omega}{c}\Big)S \theta \text{d}x\text{d}\Omega \text{d}v+C |\sqrt{\rho}\dot{\theta}|_2\\
&+C\big(|\dot{u}|_6|\theta|_3+|u\cdot \nabla u|_2|\theta|_\infty\big)\|S\|_{L^1(\mathbb{R}^+\times S^2; L^2(\mathbb{V}))}\\
&+C(1+\|u\|_1)\big(|\theta|_\infty\|S_t\|_{L^1(\mathbb{R}^+\times S^2; L^2(\mathbb{V}))}+|u\cdot \nabla \theta|_3\|S\|_{L^1(\mathbb{R}^+\times S^2; H^1(\mathbb{V}))}\big)\\
\leq &\frac{1}{c_v} \frac{d}{dt}\int_0^\infty \int_{S^2}  \int_{\mathbb{V}} \Big(1-\frac{u\cdot \Omega}{c}\Big)S \theta \text{d}x\text{d}\Omega \text{d}v\\
&+C|\nabla \dot{u}|_2+ C|\sqrt{\rho}\dot{\theta}|_2+C|\nabla \theta|^2_2+C|u\cdot \nabla u|_2+C.
\end{split}
\end{equation}
Then combining (\ref{ppyue1jkk})-(\ref{ppyue111}), we have
\begin{equation}\label{ppyue20}
\begin{split}
&\frac{1}{c_v}\frac{d}{dt}\int_{\mathbb{V}} \Big(\frac{\kappa}{2}|\nabla \theta|^2-Q(u)\theta+\int_0^\infty \int_{S^2}   \Big(1-\frac{u\cdot \Omega}{c}\Big)S \theta \text{d}\Omega \text{d}v\Big)\text{d}x+\int_{\mathbb{V} }\rho |\dot{\theta}|^2\text{d}x\\
\leq&C(|\nabla u|^4_4+|\nabla \theta|^2_2+|\nabla \dot{u}|_2+ |\sqrt{\rho}\dot{\theta}|_2+|u\cdot \nabla u|^2_2+1).
\end{split}
\end{equation}
Now multiplying (\ref{ppyue20}) by $2C$, and adding the resulting inequality into (\ref{mou5}), we have
\begin{equation}\label{ppyue21}
\begin{split}
&\frac{d}{dt}\int_{\mathbb{V}} \Big(\frac{\kappa}{2c_v}|\nabla \theta|^2-\frac{1}{c_v}Q(u)\theta+\frac{1}{c_v}\int_0^\infty \int_{S^2}   \Big(1-\frac{u\cdot \Omega}{c}\Big)S \theta \text{d}\Omega \text{d}v+\frac{1}{2}\rho |\dot{u}|^2\Big)\text{d}x\\
&+\int_{\mathbb{V} }(\rho |\dot{\theta}|^2+|\nabla \dot {u}|^2_2)\text{d}x
\leq C(|\nabla u|^4_4+|\nabla \theta|^2_2+|\nabla u|^3_3+|u\cdot \nabla u|^2_2+1).
\end{split}
\end{equation}
Similarly to (\ref{yue5}), from H\"older's inequality we get
\begin{equation}\label{ppyue22}
\begin{split}
&|\nabla u|^4_4+|\nabla u|^3_3\leq C|\nabla u|^4_4
\leq C(|G|^4_4+|\omega|^4_4+1)
\leq C(1+|(\nabla G,\nabla \omega)|^3_2),
\end{split}
\end{equation}
(\ref{yue11}), (\ref{yuekk}) and (\ref{ppyue21}), we quickly have 
\begin{equation}\label{ppyue21kkk}
\begin{split}
|\nabla \theta|^2_2+|\sqrt{\rho} \dot{u}|^2_2+
+\int_0^t( |\sqrt{\rho}\dot{\theta}|^2_2+|\nabla \dot {u}|^2_2)\text{d}s\leq & C\int_0^t |\sqrt{\rho}u_t|^3_2\text{d}s+C.
\end{split}
\end{equation}
According to (\ref{yue6}), we have
\begin{equation}\label{ppyue22}
\begin{split}
 |\sqrt{\rho}u_t|_2 \leq C(|\sqrt{\rho}\dot{u}|_2+|\sqrt{\rho}u\cdot \nabla u|)\leq C(|\sqrt{\rho}\dot{u}|_2+|\nabla u|_2+1)+\epsilon |\sqrt{\rho}u_t|_2.
\end{split}
\end{equation}

Then substituting (\ref{ppyue22}) into (\ref{ppyue21}), via Gronwall's inequality, we obtain
\begin{equation*}
\begin{split}
|\theta(t)|^2_{D^1}+|\sqrt{\rho} \dot{u}(t)|^2_{2} + \int_{0}^{t}( |\dot{u}|^2_{D^1}+|\sqrt{\rho} \dot{\theta}|^2_2)\text{d}t\leq C, \quad 0\leq t \leq T,
\end{split}
\end{equation*}
which, together with Sobolev imbedding theorem,  (\ref{yuekk}), (\ref{zhu54}), (\ref{ppyue22}) and 
\begin{equation*}
\begin{cases}
|\text{div} u|_\infty\leq C(|G|_\infty+1)\leq C(\|G\|_{W^{1,6}}+1)\leq C(|\nabla u|_6+|\nabla \dot{u}|_2+1),\\[4pt]
|\omega|_\infty\leq C(\|\omega\|_{W^{1,6}}+1)\leq C(|\nabla u|_6+|\nabla \dot{u}|_2+1),\\[4pt]
|\nabla u|_6\leq C(|\text{div}u|_6+|\omega|_6)\leq C(\|G\|_1+\|\omega\|_1+1),
\end{cases}
\end{equation*}
implies the desired conclusions for $  0\leq t \leq T$.

\end{proof}

\subsection{The higher order estimate for $|(u,\theta)|_{L^\infty([0,\overline{T}]; D^2(\mathbb{V})}$}
 \begin{lemma}\label{sk4nn}
\begin{equation*}
\begin{split}
|\theta(t)|_{D^2}+|\sqrt{\rho}\theta_t(t)|_{2}+\int_0^T  |\theta_t|^2_{D^1} \text{d}t\leq C,\quad 0\leq t<  T.
\end{split}
\end{equation*}
\end{lemma}

\begin{proof}
Differentiating $(\ref{eq:1.2})_4$ with respect to $t$,  multiplying  by $\theta_t$ and integrating over $\mathbb{V}$,
\begin{equation}\label{zhouppy}
\begin{split}
&\frac{1}{2}\frac{d}{dt} \int_{\mathbb{V}}\rho |\theta_t|^2 \text{d}x+\frac{\kappa}{c_v}\int_{\mathbb{V}}|\nabla \theta_t|^2 \text{d}x\\
=& \int_{\mathbb{V}} \Big(-\rho_t\Big( \frac{\theta_t}{2}+u\cdot \nabla \theta+R\theta \text{div}u\Big)-\rho(u_t\cdot \nabla \theta+u\cdot \nabla \theta_t+ R\theta_t \text{div}u\big) \theta_t\\
&- \frac{1}{c_v}P_m\text{div} u_t \theta_t+ \frac{1}{c_v}Q(u)_t\theta_t+\frac{1}{c_v}(N_r)_t\theta_t\Big)   \text{d}x
\equiv:\sum_{i=12}^{15}M_i+E^*.
\end{split}
\end{equation}

For $M_{12}$, via (\ref{yue11}), (\ref{yuekk}) and Lemma \ref{ablem:4-1}, we have
\begin{equation}\label{yufang1}
\begin{split}
M_{12}=&\int_{\mathbb{V}} \text{div}(\rho u)\Big( \frac{\theta_t}{2}+u\cdot \nabla \theta+R\theta \text{div}u\Big) \theta_t \text{d}x\\
=&- \int_{\mathbb{V}} \rho u\cdot \nabla \theta_t\Big( \frac{\theta_t}{2}+u\cdot \nabla \theta+R\theta \text{div}u\Big)  \text{d}x-\int_{\mathbb{V}}\rho u\cdot \frac{\nabla \theta_t}{2}\theta_t \text{d}x\\
&-\int_{\mathbb{V}} \rho u\cdot  (\nabla u\cdot \nabla \theta+u \cdot \nabla \nabla \theta)\theta_t 
\text{d}x-R\int_{\mathbb{V}} \rho u\cdot (\nabla \theta \text{div}u+\theta \nabla \text{div}u) \theta_t \text{d}x\\
\leq & \frac{\kappa}{20c_v}|\theta_t|^2_{D^1}+C|\rho u \theta_t|^2_2+C|\rho |u|^2\nabla \theta|^2_2+C|\rho u \theta \text{div}u|^2_2+C|\rho \theta_t|^2_2\\
&+C||\nabla u| |\nabla \theta||^2_2+C|u\cdot \nabla \nabla \theta|^2_2-\frac{R}{2\mu+\lambda}\int_{\mathbb{V}} \rho \theta u\cdot \big(\nabla G+\nabla (R\rho \theta)\big) \theta_t \text{d}x\\
\leq & \frac{\kappa}{10c_v}|\theta_t|^2_{D^1}+\frac{R^2}{2\mu+\lambda}\int_{\mathbb{V}} \frac{1}{2}\rho^2 \theta^2( \text{div}u \theta_t +u\cdot \nabla \theta_t)\text{d}x+C(1+|\sqrt{\rho}  \theta_t|^2_2)\\
&+C(|\nabla^2 \theta|^2_2+|\nabla G|^2_2)
\leq C+ \frac{\kappa}{10c_v}|\theta_t|^2_{D^1}+C|\sqrt{\rho}  \theta_t|^2_2+C|\nabla^2 \theta|^2_2.
\end{split}
\end{equation}
For $M_{13}$ and $M_{14}$, we have
\begin{equation}\label{yufang22}
\begin{split}
M_{13}=&\int_{\mathbb{V}} -\rho(u_t\cdot \nabla \theta+u\cdot \nabla \theta_t+ R\theta_t \text{div}u\big) \theta_t\text{d}x\\
\leq &\int_{\mathbb{V}}\rho \theta_t \big(- \dot{u}\cdot \nabla \theta  + (u\cdot \nabla)u\cdot \nabla \theta \big) \text{d}x +\frac{\kappa}{20c_v}|\theta_t|^2_{D^1}+(|\text{div}u|_\infty+1)|\sqrt{\rho}  \theta_t|^2_2\\
\leq &C( |\dot{u}|_6+ |\nabla u|_6)|\sqrt{\rho} \theta_t|_2|\nabla \theta|_3+\frac{\kappa}{20c_v}|\theta_t|^2_{D^1}+(|\text{div}u|_\infty+1)|\sqrt{\rho}  \theta_t|^2_2\\
\leq & \frac{\kappa}{10c_v}|\theta_t|^2_{D^1}+C(|\text{div}u|_\infty+|\nabla \dot{u}|^2_2+1)|\sqrt{\rho}  \theta_t|^2_2+C|\theta|^2_{D^2}+C,\\
M_{14}=&\int_{\mathbb{V}} -\frac{R}{c_v}\rho  \theta \text{div}\dot{u}\theta_t\text{d}x+\int_{\mathbb{V}} \frac{R}{c_v}\rho  \theta \text{div}(u\cdot \nabla u)\theta_t\text{d}x,\\
\leq & C|\rho \theta \theta_t|^2_2+C|\text{div}\dot{u}|^2_2+C\int_{\mathbb{V}} \rho \theta |\nabla u|^2|\theta_t|\text{d}x+\int_{\mathbb{V}} \frac{R}{c_v}\rho \theta \theta_t u\cdot \nabla \text{div}u\text{d}x\\
\leq & C|\sqrt{\rho}  \theta_t|^2_2+C|\text{div} \dot{u}|^2_2+C|\nabla u|^4_4+\frac{R}{(2\mu+\lambda)c_v}\int_{\mathbb{V}} \rho \theta \theta_t u\cdot \nabla G\text{d}x\\
&+\frac{R}{(2\mu+\lambda)c_v}\int_{\mathbb{V}} \rho^2 \theta \theta_t u\cdot \nabla \theta\text{d}x+\frac{R}{(2\mu+\lambda)c_v}\int_{\mathbb{V}} \rho \theta^2 \theta_t u\cdot \nabla \rho \text{d}x\\
\leq &C+ C|\sqrt{\rho}  \theta_t|^2_2+C|\text{div} \dot{u}|^2_2+\frac{R}{(2\mu+\lambda)c_v}\int_{\mathbb{V}} \rho \theta^2 \theta_t u\cdot \nabla \rho\text{d}x.
\end{split}
\end{equation}
Now we consider the last term in (\ref{yufang22}) that
\begin{equation}\label{yufang4}
\begin{split}
\int_{\mathbb{V}} \rho \theta^2 \theta_t u\cdot \nabla \rho\text{d}x
=&-\int_{\mathbb{V}} \frac{1}{2}\Big(\rho^2 \theta^2 \theta_t \text{div}u+ \rho^2 \theta^2 u\cdot \nabla \theta_t\text{d}x+2 \rho^2 \theta \theta_t u\cdot \nabla \theta\Big) \text{d}x\\
\leq& C|\sqrt{\rho}  \theta_t|^2_2+\frac{c_v\kappa(2\mu+\lambda)}{20R}|\nabla \theta_t|^2_2+C,
\end{split}
\end{equation}
which, together with (\ref{yufang22}), implies that 
\begin{equation}\label{yufang5}
\begin{split}
M_{14}
\leq & C|\sqrt{\rho}  \theta_t|^2_2+\frac{\kappa}{20c_v}|\nabla \theta_t|^2_2+C|\text{div} \dot{u}|^2_2+C.
\end{split}
\end{equation}
For $M_{15}$, from Lemma \ref{pang},  we have
\begin{equation}\label{yufang6}
\begin{split}
M_{15}=& \int_{\mathbb{V}} \frac{1}{c_v}Q(u)_t\theta_t  \text{d}x=\frac{1}{c_v}\int_{\mathbb{V}} \Big(4\mu D(u):D(\dot{u}) +2\lambda \text{div}u \text{div}\dot{u}\Big)  \theta_t \text{d}x\\
&-\frac{1}{c_v}\int_{\mathbb{V}} \Big(4\mu D(u):D(u\cdot \nabla u) \theta+2\lambda \text{div}u \text{div}(u\cdot \nabla u\big) \theta\Big) \text{d}x\\
\leq&  C |\nabla \dot{u}|_2 |\nabla u|_3|\theta_t|_6+C ||\nabla u|^3\theta_t|_1+C|\nabla u|^4_4+\frac{\kappa}{20c_v}|\nabla \theta_t|^2_2\\
\leq& \frac{\kappa}{10c_v}|\nabla \theta_t|^2_2+C|\sqrt{\rho}  \theta_t|^2_2+ C|\nabla \dot{u}|^2_2+C.
\end{split}
\end{equation}
Next considering $E^*$:
\begin{equation*}
\begin{split}
E^*=\frac{1}{c_v}\int_0^\infty \int_{S^2} \int_{\mathbb{V}}\Big( \Big(1-\frac{u\cdot \Omega}{c}\Big) (A_r)_t \theta_t+ \Big(-\frac{u_t\cdot \Omega}{c}\Big) A_r\theta_t\Big) \text{d}x \text{d}\Omega \text{d}v=\sum_{j=1}^{8} G_j.
\end{split}
\end{equation*}
From H\"older's inequality,  (\ref{jia345}) and Young's inequality, we have
\begin{equation}\label{op11}
\begin{split}
G_1=&\frac{1}{c_v}\int_0^\infty \int_{S^2} \int_{\mathbb{V}} \Big(1-\frac{u\cdot \Omega}{c}\Big) S_t\theta_t \text{d}x \text{d}\Omega \text{d}v\\
\leq &C\big(1+|u|_{\infty}\big)|\theta_t|_{2}\int_0^\infty \int_{S^2} |S_t |_{L^{2}}\text{d}\Omega \text{d}v
\leq  \frac{\kappa}{20c_v}(|\sqrt{\rho}\theta_t|^2_2+|\nabla \theta_t|^2_2)+C,\\
\end{split}
\end{equation}
\begin{equation}\label{op1}
\begin{split}
G_2=&-\frac{1}{c_v}\int_0^\infty \int_{S^2} \int_{\mathbb{V}} \Big(1-\frac{u\cdot \Omega}{c}\Big) (\sigma_\theta \theta_t \rho I+\sigma \rho_t I+\sigma \rho I_t)\theta_t \text{d}x \text{d}\Omega \text{d}v\\
=&-\frac{1}{c_v}\int_0^\infty \int_{S^2} \int_{\mathbb{V}} \Big(1-\frac{u\cdot \Omega}{c}\Big) (\sigma_\theta \theta_t \rho I-\text{div}(\rho u)\sigma  I+\sigma \rho I_t)\theta_t \text{d}x \text{d}\Omega \text{d}v\\
\leq &  F_r(1+|u|^2_\infty)(1+|\sqrt{\rho}\theta_t|^2_2+|\nabla \theta_t|_2),
\end{split}
\end{equation}
where 
$$
F_r=\int_0^\infty \int_{S^2} \big(|\sigma_\theta|_\infty|I|_\infty+|\sigma|_\infty(|I_t|_2+|\nabla I|_2+|I|_\infty)+|\nabla I|_2|\sigma|_\infty\big)\text{d}\Omega \text{d}v\leq C.
$$
Similarly, via the similar argument used for $G_2$, we have
\begin{equation}\label{op2}
\begin{split}
G_3=&\frac{1}{c_v}\int_\mathbb{I} \int_{\mathbb{V}}\frac{v}{v'}\Big(1-\frac{u\cdot \Omega}{c}\Big)(\overline{\sigma}_s \rho I'_t+(\overline{\sigma}_s)_t \rho I'+\overline{\sigma}_s \rho_t I')
\theta_t\text{d}x \text{d}\mathbb{I}\\
\leq& G_r(1+|u|^2_\infty)(1+|\sqrt{\rho}\theta_t|^2_2+|\nabla \theta_t|_2),\\
G_4=&\frac{1}{c_v}\int_\mathbb{I} \int_{\mathbb{V}}-\Big(1-\frac{u\cdot \Omega}{c}\Big)(\overline{\sigma}'_s \rho I_t+(\overline{\sigma}'_s)_t \rho I+\overline{\sigma}'_s \rho_t I)
\theta_t\text{d}x \text{d}\mathbb{I}\\
\leq& J_r(1+|u|^2_\infty)(1+|\sqrt{\rho}\theta_t|^2_2+|\nabla \theta_t|_2),
\end{split}
\end{equation}
where 
\begin{equation}\label{opfang3}
\begin{split}
\displaystyle
G_r=&\int_\mathbb{I} \frac{v}{v'}\big(\overline{\sigma}_s(|I'_t|_2+|I'|_2+|\nabla I'|_2+|I'|_\infty)+|I'|_2(|(\overline{\sigma}_s)_t|_\infty+|\nabla \overline{\sigma}_s|_2)\big) \text{d}\mathbb{I}, \\
\displaystyle
J_r=&\int_\mathbb{I} \big(\overline{\sigma}_s(|I_t|_2+|I|_2+|\nabla I|_2+|I|_\infty)+|I|_2(|(\overline{\sigma}_s)_t|_\infty+|\nabla \overline{\sigma}_s|_2)\big)\text{d}\mathbb{I}.
\end{split}
\end{equation}
Finally, we consider the terms $G_5$-$G_8$:
\begin{equation}\label{opfang2}
\begin{split}
G_5=&\frac{1}{c_v}\int_0^\infty \int_{S^2} \int_{\mathbb{V}} \frac{-u_t \cdot \Omega}{c} S\theta_t \text{d}x \text{d}\Omega \text{d}v\\
\leq&  C|u_t|_{6}|\theta_t|_{6}\int_0^\infty \int_{S^2} | S|_{\frac{3}{2}}\text{d}\Omega \text{d}v
 \leq \frac{\kappa}{20c_v}(|\theta_t|^2_{D^1}+|\sqrt{\rho}\theta_t|^2_{2})+C|u_t|^2_{D^1},\\
G_6=&\frac{1}{c_v}\int_0^\infty \int_{S^2} \int_{\mathbb{V}} \frac{u_t \cdot \Omega}{c}  \sigma_a I\theta_t \text{d}x \text{d}\Omega \text{d}v\\
\leq& C|u_t|_{6}|\rho|^{\frac{1}{2}}_\infty|\sqrt{\rho}\theta_t|_{2}\int_0^\infty \int_{S^2} |\sigma|_\infty|I|_{3}\text{d}\Omega \text{d}v
\leq C|\sqrt{\rho}\theta_t|^2_{2}+C|u_t|^2_{D^1},\\
G_7=&\frac{1}{c_v}\int_\mathbb{I} \int_{\mathbb{V}} -\frac{v}{v'}\frac{u_t \cdot \Omega}{c}
\sigma_s I'\theta_t\text{d}x \text{d}\mathbb{I}\\
\leq& C|u_t|_{6}|\rho|^{\frac{1}{2}}_\infty|\sqrt{\rho}\theta_t|_{2}
\int_\mathbb{I} \frac{v}{v'}
|\overline{\sigma}_s|_\infty |I'|_3\text{d}\mathbb{I}
\leq  C|\sqrt{\rho}\theta_t|^2_{2}+C|u_t|^2_{D^1},\\
\end{split}
\end{equation}
\begin{equation}\label{op3}
\begin{split}
G_8=&\frac{1}{c_v}\int_\mathbb{I} \int_{\mathbb{V}} \frac{u_t \cdot \Omega}{c}
\sigma'_s I\theta_t\text{d}x \text{d}\mathbb{I}
\leq C|u_t|_{6}|\rho|^{\frac{1}{2}}_\infty|\sqrt{\rho}\theta_t|_{2}
\int_\mathbb{I}
|\overline{\sigma}'_s|_\infty |I|_3\text{d}\mathbb{I}
\leq  C|\sqrt{\rho}\theta_t|^2_{2}+C|u_t|^2_{D^1}.
\end{split}
\end{equation}

Then combining (\ref{zhouppy})-(\ref{opfang2}), via Gronwall's inequality, we quickly have
\begin{equation}\label{zhouppyas}
\begin{split}
|\theta(t)|^2_{D^2}+|\sqrt{\rho}\theta_t(t)|_{2}+\int_0^T  |\theta_t|^2_{D^1} \text{d}t\leq C,\quad 0\leq t<  T.
\end{split}
\end{equation}

\end{proof}

Finally, we give the upper bounds of  $| (\rho,I)|_{D^{1,q}}$, $|(u,\theta)|_{D^{2,q}}$, $ |u_t|^2_{D^1}$  and so on.
  \begin{lemma}\label{sk4nnpppy}
\begin{equation*}
\begin{split}
|( \rho,  I)(t)|_{D^{1,q}}+|(\rho_t,I_t)(t)|_q
+|u(t)|_{D^2}+\int_0^T  (|u_t|^2_{D^1} +|(u,\theta)|^2_{D^{2,q}})\text{d}t\leq& C,
\end{split}
\end{equation*}
 \end{lemma}

\begin{proof}

In the following estimates we will use 
 \begin{equation}\label{zhu55a}
\begin{split}
|\nabla^2 u|_q \leq &  C(|\nabla \rho |_q+|\nabla \dot{u}|_2+1),\  \   |\nabla^2 u|_2 \leq C(|\nabla \rho |_2+1), \\
|\nabla u|_{\infty} \leq &   C\big(     |\text{div}u|_\infty+|\omega|_\infty \big) \text{ln} (e+|\nabla^2  u|_q) +C|\nabla u|_2+C\\
 \leq &  C\big(     |\text{div}u|_\infty+|\omega|_\infty \big) (\text{ln} (e+|\nabla \rho |_q)+   \text{ln} (e+|\nabla \dot{u}|_2)+1),\\
|\theta|_{D^{2,q}}\leq&  C(|\nabla \theta_t|_2+|u|_{D^2}+|\nabla u |_\infty|\nabla u|_q+1),
\end{split}
\end{equation}
where we have used the equations $(\ref{eq:1.2})_3$-$(\ref{eq:1.2})_4$,  Lemmas \ref{tvd1} and   \ref{s2}-\ref{sk4nn}.

First,  applying $\nabla$ to  $(\ref{eq:1.2})_2$, multiplying the resulting equations by $q|\nabla \rho|^{q-2} \nabla \rho$,
and  integrating  over $\mathbb{V}$,  via (\ref{zhu55a}) we immediately obtain
\begin{equation}\label{zhu200}
\begin{split}
\frac{d}{dt}|\nabla \rho|_q
\leq& C|\nabla u|_\infty|\nabla \rho|_q+C|\nabla^2 u|_q\leq C(1+|\nabla u|_\infty)|\nabla \rho|_q+C|\nabla \dot{u}|_2+C.
\end{split}
\end{equation}
Via   (\ref{zhu55a}), (\ref{zhu200}) and notations:
$$
f=e+|\nabla \rho|_q,\quad g=1+\big(     |\text{div}u|_\infty+|\omega |_\infty \big) \text{ln} (e+|\nabla \dot{u}|_2),
$$
 we quickly have
$$
f_t\leq Cgf +Cf\ln f+Cg,
$$
which, together with Lemma \ref{ablem:4-1} and Gronwall's inequality,  implies that 
$$
\ln f(t)\leq C, \quad 0\leq t<  T.
$$
So we obtained the desired estimate for $|\nabla \rho|_q$, and the upper  bound of  $|\nabla I|_q$  can be deduced easily via the similar argument used in Lemma \ref{lem:3}.
The estimates for $\rho_t$ and  $I_t$  can be obtained  via equations (\ref{eq:1.2}).
Finally, via (\ref{zhu55a}), we only need to check that
$$ |\nabla u_t|_2\leq |\nabla \dot{u}|_2+|\nabla (u\cdot \nabla u)|_2\leq C(|\nabla \dot{u}|_2+1),$$
which, combining with Lemma \ref{ablem:4-1}, implies the desired conclusions.
 \end{proof}

These estimates will be enough to extend the strong solution $(I,\rho,u,\theta)$ beyond $t\geq \overline{T}$.

In truth, via the estimates obtained in  Lemmas \ref{s2}-\ref{sk4nnpppy}, we quickly know that the functions $(I,\rho,u, \theta)|_{t=\overline{T}} =\lim_{t\rightarrow \overline{T}}(I,\rho,u, \theta)$ satisfy the conditions imposed on the initial data $(\ref{gogo})$-$(\ref{kkk})$. Therefore, we can take $(I,\rho,u, \theta)|_{t=\overline{T}}$ as the initial data and apply  Theorem \ref{th1} to extend the local solution beyond $t\geq \overline{T}$. This contradicts the assumption on $\overline{T}$.

\bigskip

{\bf Acknowledgement:} The research of  Y. Li and S. Zhu were supported in part
by National Natural Science Foundation of China under grant 11231006 and Natural Science Foundation of Shanghai under grant 14ZR1423100. S. Zhu was also supported by China Scholarship Council.

\end{document}